\documentclass[psamsfonts]{amsart}
\usepackage{mymacros}

\title{Genera from an algebraic index theorem for Supermanifolds}

\date{}
\author{Araminta Amabel}

\begin{document}

\begin{abstract}
We prove a super-version of Nest-Tsygan's algebraic index theorem. 
This work is inspired by the appearance of the same cobordism invariants in three related stories: 
index theory, trace methods in the deformation theory of algebras, and partition functions in quantum field theory. 
We show that one can recover the cobordism invariant appearing in  supersymmetric quantum mechanics using trace methods for the associated deformation quantization problem.

\end{abstract}
 
\maketitle
\tableofcontents

\section{Introduction}

The Atiyah-Singer index theorem \cite{AtiyahSinger} states that the analytic index of an elliptic differential operator agrees with its topological index. 
Since its announcement in the early 1960s, various proofs and generalizations of the index theorem have been given; see \cite{Index1,Index2,Index3,Index4,Index5} and also \cite{IndexFreed} for a nice overview.
An algebraic analogue of the index theorem \cite{FedosovIndex,Nest-Tsygan,FFSh} 
was given in the early 1990s, 
replacing the analytic index with a trace on the algebra of differential operators. 
In 1996, Nest and Tsygan \cite{NestTsygan2} showed that the algebraic index theorem implies Atiyah-Singer's result. 

Here, we prove a super-version of \cite{NestTsygan2}. 
That is, we study symplectic supermanifolds and deformation quantizations of their superalgebras of smooth functions. 
Before describing our result and related work, 
we review the classical algebraic index theorem in a way that will make our generalizations more visible.

Let $X$ be a closed manifold. 
The Atiyah-Singer index theorem relates two invariants of an elliptic differential operator $D$ on $X$. 
Let $\mrm{Diff}_X$ denote the algebra of all, not necessarily elliptic, differential operators on $X$. 
In the last 30
years, it has proven fruitful to investigate how the two invariants appearing in the index theorem (the analytic and topological indices)  
are related to the algebraic structure of $\mrm{Diff}_X$. 
A key component of this investigation is the appearance of $\mrm{Diff}_X$ as a deformation quantization of the algebra of smooth functions $\mathcal{O}_{T^*X}$ on the symplectic manifold $T^*X$. 
Deformation quantization is a rich subject, pioneered by De Wilde-Lecomte \cite{DWL} and Fedosov \cite{Fedosov} who first described the space of deformations of an arbitrary symplectic manifold, 
and greatly generalized by Kontsevich \cite{Kontsevich} to the study of Poisson manifolds. 
Here, we will focus on the middle level of generality, working with an arbitrary symplectic manifold $(M,\omega)$. 
A deformation quantization for $(M,\omega)$ is determined by a choice of what is now called a Fedosov connection.

These deformed algebras possess additional algebraic data (a trace) that one can exploit to give a new approach to the index theorem. 
Pointwise, a Fedosov deformation of a $(M,\omega)$ looks like the Weyl algebra.  
The Weyl algebra has a well-known ``derived" trace (a Hochschild cohomology class with coefficients in the dual). 
In his original paper, Fedosov shows that this trace determines a global trace on the deformation quantization \cite[Def. 5.1]{Fedosov}. 
This trace is essentially the trace appearing in the algebraic index theorem. 

We will be interested in the following phrasing of \cite[Thm. 1.1.1]{Nest-Tsygan} which appears in \cite[Thm. 1.1]{FFSh}. 

\begin{thm}[Algebraic Index Theorem]\label{thm-AITintro}
Let $(M,\omega)$ be a compact symplectic manifold. 
For each Fedosov connection $\nabla$ on $M$,  
there exists a unique normalized trace $t_M^\nabla$ on the corresponding Fedosov deformation. 
Moreover, one has 
\begin{align}\label{eq-AIT}
t_M^\nabla(1)=\frac{1}{(2\pi i)^n}\int_M\hat{\mrm{A}}(TM)\exp(-\Omega/\hbar),
\end{align}
for $\Omega$ the Weyl curvature of $\nabla$.
\end{thm}

\begin{rmk}
Other proofs of the algebraic index theorem can be found in \cite{PPT,Vasiliy}. 
In \cite{PPT}, they also prove a version of the algebraic index theorem for orbifolds. 
In \cite{Cher}, an algebraic index theorem for Cherednik algebras is shown. 
\end{rmk}

The right-hand side of (\ref{eq-AIT}) looks strikingly similar to the topological index of a differential operator. 
In particular, the same cobordism invariant appears in both settings: the $\widehat{A}$-genus. 
Nest and Tsygan \cite{Nest-Tsygan,NestTsygan2} have shown that the left-hand side of \ref{eq-AIT} can be related to the analytic index. 
In brief, this is done by evaluating the trace $t^\nabla_{T^*X}$ on the 
projection onto the kernel of an elliptic differential operator. 

Thus by comparing the algebraic index theorem to the Atiyah-Singer index theorem, we have a bridge between deformation theory and index theory. 

In a different vein, deformation theory is also fundamentally related to the problem of quantizing field theories. 
Given a classical field theory, one can form a commutative algebra of ``classical observables" 
which is roughly the measurements one can take on the fields. 
A quantization of the classical field theory has a new collection of measurements, leading to a new algebra of ``quantum observables." 
The quantum observables of a quantum field theory form a deformation of the classical observables of the associated classical field theory \cite{CG1}. 
The relationship between Fedosov and Kontsevich's deformation theory and perturbative methods in field theory goes back further; see for example \cite{Cattaneo}. 

Using this translation between field theory and deformation theory, 
Grady-Li-Li and Gui-Li-Xu \cite{GradyLiLi,newSiLi} derive the algebraic index theorem by computing the partition function of topological quantum mechanics (sometimes referred to as a one-dimensional generalized Chern Simons theory) using the BV formalism. 
See \cite{GwilliamGrady} for an analysis of one-dimensional Chern Simons theory in the BV formalism. 
This provides an exciting new interplay between trace methods in deformation theory and partition function results in quantum field theory. 

We thus see three interconnected stories: index theory, trace methods in deformation theory, and partition functions of quantum field theories. 
The goal of this paper is to study the super-analogue of the algebraic index theorem part of these stories. 
Just as \cite{GradyLiLi,newSiLi} showed that the algebraic index theorem was related to quantum mechanics, 
the super-version should be related to \emph{supersymmetric quantum mechanics}. 

\begin{rmk}
An important class of theories one can study in the BV formalism is given by those equipped with a $\bb{Z}/2$-grading: supersymmetric theories. 
The possible applications to supersymmetry is in large part our motivation to study a super analogue of the algebraic index theorem. 
Additionally, there are interesting relationships between supersymmetric field theories and chromatic homotopy theory, following Stolz-Teichner \cite{ST1,ST2,WG,Pokman}. 
Here and in \cite{My2}, symplectic supermanifolds have \emph{even} symplectic form. 
Odd symplectic supermanifolds also have interesting connections to the BV formalism; see for example \cite{EzraSean}. 
\end{rmk}

\subsubsection{Manifold Invariants}

In the Atiyah-Singer theorem, the algebraic index theorem, and the partition function of topological quantum mechanics 
we see the same cobordism invariant appearing: the $\widehat{A}$-genus. 

For example, the topological index of an elliptic differential operator $D$ on a compact manifold $X$ is 
\[\int_X\mrm{Td}(X)\mrm{ch}(D),\]
the integral of the Todd class of the manifold and the Chern character of the operator $D$. 
Viewed as an invariant of real, rather than complex bundles, 
the $\widehat{A}$-genus of a spin manifold $Y$ is $e^{-c_1(Y)/2}\mrm{Td}(Y)$;  
see \cite[Pg . 165]{Hirz}. 

From the relationships explained above, the $\widehat{A}$-genus can also be extracted 
from the trace on the Fedosov deformation (Theorem \ref{thm-AITintro}) and 
from the partition function of topological quantum mechanics. 

\begin{rmk}
One dimension higher, the two-dimensional $\beta\gamma$ holomorphic sigma model (also called holomorphic Chern-Simons theory) has partition function related to the Witten genus; see \cite{Costello,GGW}. 
In this two-dimensional story, the algebra of differential operators is replaced with the vertex algebra of chiral differential operators \cite{chiral}.
See \cite{Pokman2} for a discussion of how this story relates to supermanifolds.
\end{rmk}

One of the main goals of this paper is to discover what genus (i.e., cobordism invariant) replaces $\widehat{A}$ in the super-version of the algebraic index theorem. 
In \cite{Berwick-Evans}, Berwick-Evans shows that the partition function of supersymmetric quantum mechanics is related to Hirzebruch's L-genus. 
The L-genus assigns to a $4m$-dimensional manifold its signature \cite[\S 4.1]{Hirz}. 
It has associated characteristic series
\[\frac{\sqrt{z}}{\tanh(\sqrt{z})}.\]
A very similar formal power series also appears in Engeli's special case of the superalgebraic index theorem \cite[Lem. 2.25]{Engeli}. 
See also \cite[Appendix E]{WillwacherE}. 
We therefore expect the L-genus to replace the $\widehat{A}$-genus in some cases. 

\begin{rmk}
Note that the L-genus and the $\widehat{A}$-genus appear in partition functions of field theories of dimension one. 
These genera also naturally land in cohomology theories of chromatic height one. 
Analogously, the Witten genus, coming from a two-dimensional theory, naturally lands in a cohomology theory of height two \cite[Ch. 10]{tmfbook}.
Linking chromatic height to field theory dimension is a key aspect of the Stolz-Teichner program \cite{ST2}.
\end{rmk}

The L-genus is closely related to the theory of quadratic forms and their signatures \cite{RanickiAlg}. 
Given a quadratic vector space $(V,Q)$, one has an associated Clifford algebra 
\[\mrm{Cliff}(V,Q)=T^\bullet(V)/\{v^2=\hbar Q(v)\}.\]
Just as pointwise a Fedosov deformation looked like the Weyl algebra, 
pointwise deformations for symplectic supermanifolds will look like a Weyl algebra tensor a Clifford algebra.  
A case of particular interest (which is considered by Engeli in \cite[Lem. 2.25]{Engeli}) is when the Clifford algebra comes from a quadratic form of signature $(a,a)$. 
The superalgebraic index theorem of Engeli thus relates quadratic forms of a \emph{fixed} signature to the L-genus. 
A natural question to ask is what happens for quadratic forms of other signatures, 
which is exactly the generalization of Engeli's result that we prove here.

\begin{rmk}
While our motivation comes from a variety of areas, this paper itself keeps a tight focus on supertraces of superalgebras. 
We assume the reader is familiar with things like Hochschild homology, Lie algebra cohomology, and Chern-Weil theory; 
however, no knowledge of chromatic homotopy theory or physics is necessary to understand the content of this paper.
\end{rmk}

\subsubsection{Overview of Results}

Let $(\bb{M},\omega)$ be a symplectic supermanifold with $2n$ even dimensions, $a+b$ odd dimensions, and the symplectic form $\omega$ in the odd directions has signature $(a,b)$. 
We refer to the triple of numbers $(2n\vert a,b)$ as the \emph{type} of $\bb{M}$. 
To prove a super-version of Theorem \ref{thm-AITintro}, 
we need to understand the deformation quantizations of functions on $\bb{M}$.  

Deformation theory of symplectic supermanifolds was first studied by Bordemann \cite{Bordemann1,Bordemann2}. 
We studied this theory in \cite{My2} using Gelfand-Kazhdan descent. 
More explicitly, in \cite{My2} we construct a version of Gelfand-Kazhdan descent for $\bb{M}$ that depends on an \emph{$\hbar$-formal exponential} $\sigma$ \cite[Def. 5.7]{My2}. 
This super-Gelfand-Kazhdan descent is a sort of Borel construction for a specific Harish-Chandra pair $(\mfrk{g},K)$ and a principal $(\mfrk{g},K)$-bundle $\mrm{Fr}_\bb{M}$, of symplectic frames on $\bb{M}$. 
This descent functor sends a deformation $\AQ$ of functions on the formal disk to a deformation $\mathcal{A}_\sigma(\bb{M})$ of $\bb{M}$. 
The superalgebra $\AQ$ looks like a Weyl algebra tensored with a Clifford algebra. 

\begin{goal}
The goal of this paper is to study supertraces on the deformed superalgebra $\mathcal{A}_\sigma(\bb{M})$. 
\end{goal}

Corollary \ref{cor-TmTrace} shows that a supertrace on the Weyl-Clifford algebra $\AQ$ induces a supertrace on the deformation $\mathcal{A}_\sigma(\bb{M})$. 
By a straightforward Hochschild cohomology computation, 
a supertrace on $\AQ$ is unique up to a scalar. 
We discuss this computation and a unique normalization condition of supertraces on $\mathcal{A}_\sigma(\bb{M})$ in \S\ref{sec-UniquenessofSupertraces}. 

The algebraic index theorem computes the value of the normalized trace on the unit. 
Analogously, we define an invariant $\msf{Ev}$ of normalized supertraces given by evaluating the supertrace on a particular element in the deformed algebra; see Definition \ref{def-VF}. 
There is a local version of this invariant $\msf{Ev}_\mrm{loc}$ for supertraces on $\AQ$ which lives in the relative Lie algebra cohomology group $C^\bullet(\mfrk{g},\mrm{Lie}(K))$. 

Our analogue of Theorem \ref{thm-AITintro} will be a computation of the invariant $\msf{Ev}$ for the unique normalized supertrace on a symplectic supermanifold. 
 
The underpinnings of algebraic index type theorems is the ability to compute $\msf{Ev}$ of a descended supertrace in terms of Chern-Weil style characteristic classes for the local invariant $\msf{Ev}_\mrm{loc}$. 
That is, assuming one can compute $\msf{Ev}_\mrm{loc}$, one would like a formula for the global invariant $\msf{Ev}$. 
This formula should look something like applying the classical Chern-Weil map and integrating over the manifold. 
In the purely even case, this style of formula is proven in \cite[Thm. 4.3]{FFSh}, 
by appealing to uniqueness results.

Below, in Theorem \ref{thm-GlobalizeInvariant}, we give an alternative proof of a more general result. 
Roughly, this is a piece of Chern-Weil theory for $(\mfrk{g},K)$-bundles. 
To motivate our result, let us recall the classical Chern-Weil construction. 

Let $P\rta Y$ be a principal $K$-bundle on a closed $2n$-manifold $Y$. 
Choose a connection on $P$ and let $F$ denote its curvature. 
Let $\mfrk{k}$ be the Lie algebra of $K$. 
Chern-Weil theory gives a map 
\[\mrm{CW}_P\colon\mrm{Sym}^\bullet(\mfrk{k}^\vee)^K\rta \Omega^{2\bullet}(Y)\]
by evaluating an ad-invariant polynomial on the curvature $F$. 

Passing to cohomology, we get a map 
\[\mrm{Sym}^n(\mfrk{k}^\vee)^K\xrta{\mrm{CW}_P} H^{2n}(Y)\xrta{\int_Y}\bb{R}\]
Alternatively, one can identify $\mrm{Sym}^\bullet(\mfrk{k}^\vee)^K$ with the cohomology of the classifying space $H^\bullet(BK;\bb{R})$. 
Pulling back along the classifying map $\rho_P$ for the bundle $P\rta Y$, 
we get another map 
\[\mrm{Sym}^n(\mfrk{k}^\vee)^K\xrta{\rho_P^*} H^{2n}(Y)\xrta{\int_Y}\bb{R}.\]
The diagram
\[\begin{xymatrix}
{
H^{2n}(BK)\arw[rd]^{\rho_P^*}\ar@{=}[d] & & \\
\mrm{Sym}^n(\mfrk{k}^\vee)^K\arw[r]_{\mrm{CW}_P}& H^{2n}(Y)\arw[r]_-{\int_Y} & \bb{R}
}
\end{xymatrix}\]
commutes. 
That is, for any ad-invariant polynomial $I$ of degree $n$, 
\begin{align}\label{eq-CWintro}
\int_Y\rho^*(I)=\int_Y\mrm{CW}_P(I).
\end{align}

There is a version of the Chern-Weil construction for Harish-Chandra pairs; see \cite[\S 1.4]{GGW}. 
For the principal $(\mfrk{g},K)$-bundle $\mrm{Fr}_\bb{M}$ used in our super-Gelfand-Kazhdan descent, 
there is a map 

\[\mrm{char}_{(\bb{M},\sigma)}\colon C^\bullet(\mfrk{g},\mfrk{k})\rta \Omega^\bullet(\bb{M}).\]
Now, instead of an ad-invariant polynomial determining a characteristic class for principal $K$-bundles, 
a relative Lie algebra cohomology class gives a characteristic class for principal $(\mfrk{g},K)$-bundles. 
In particular, given a supertrace $t_{2n\vert a,b}$ on $\AQ$, 
our local invariant $\msf{Ev}_\mrm{loc}(t_{2n\vert a,b})$ determines a characteristic class for the bundle $\mrm{Fr}_\bb{M}$. 

Analogous to the right-hand side of (\ref{eq-CWintro}), 
we will consider 
\[\int_\bb{M}\mrm{char}_{(\bb{M},\sigma)}(\msf{Ev}_{\mrm{loc}}(t_{2n\vert a,b})).\]

We will need to assume that our traces on $\AQ$ satisfy an additional compatibility condition with action of the Lie group $K$, called being \emph{relative}; 
see Definition \ref{def-relative}.

\begin{thm}\label{thm-intro1}
Let $t_{2n\vert a,b}\colon(\AQ)^{\otimes 2n}\rta\bb{K}$ be a relative supertrace. 
For $(\bb{M},\sigma)$ a symplectic supermanifold of type $(2n\vert a,b)$ and an $\hbar$-formal exponential, 
let $t_\bb{M}$ denote the supertrace on $\mathcal{A}_\sigma(\bb{M})$ induced from $t_{2n\vert a,b}$ using Corollary \ref{cor-TmTrace}.
Then
\[\msf{Ev}_\bb{M}(t_\bb{M})=
\int_\bb{M}\mrm{char}_{(\bb{M},\sigma)}(\msf{Ev}_\mrm{loc}(t_{2n\vert a,b})).\]
\end{thm}

\begin{rmk}
The relationship between the map $\mrm{char}_{(\bb{M},\sigma)}$, the classical Chern-Weil map $\mrm{CW}$, and the Chern-Weil map studied in \cite[\S 5.1]{FFSh} can be made explicit, 
a new observation that we record in Lemma \ref{lem-CharCW} below. 
\end{rmk}

The computational heart of this paper is an explicit description of the invariant $\msf{Ev}_\mrm{loc}$ on our chosen supertrace on $\AQ$. 
This local superalgebraic index theorem appears as Theorem \ref{thm-LocalSAIT} below. 
Our proof of this more general result follows ideas in \cite{FFSh} and \cite{Engeli}. 
A superalgebraic index theorem was proven for a specific class of symplectic supermanifolds (those of type $(2n\vert a,a)$) by Engeli in \cite{Engeli}. 
In the special case considered in \cite{Engeli}, 
we recover Engeli's result, up to a sign; see Remark \ref{rmk-CompareEngeli}. 
We also include as many details as possible in our proof so the reader may follow along with the computation.

Relying on the local computations and Theorem \ref{thm-GlobalizeInvariant}, 
we are able to compute $\msf{Ev}$ on the unique normalized supertrace on the deformed superalgebra $\mathcal{A}_\sigma(\bb{M})$. 

Assume $b\geq a$. 
Given a connection on the symplectic frame bundle $\Fr_\bb{M}$, 
write its curvature as a sum
$R+S$ so that 
 $R$ is a $\mfrk{sp}_{2n}$-valued form and
 $S$ is a $\mfrk{so}_{a,b}$-valued form. 
 Let $\int_\bb{M}$ be integration over the supermanifold $\bb{M}$. 
 See \S\ref{subsec-IntandOri}. 
The following is Theorem \ref{thm-main} below.

\begin{thm}[Superalgebraic Index Theorem]
The evaluation $\msf{Ev}_\bb{M}$ of the unique normalized supertrace $\mrm{Tr}_\bb{M}$ is 
\begin{align}\label{eq-SAIT}
\msf{Ev}_\bb{M}(\mrm{Tr}_\bb{M})=(-1)^{n+a+\z }\hbar^n\int_\bb{M}\widehat{A}(R)\widehat{BC}(S)\exp(-\Omega/\hbar)
\end{align}
where $\widehat{A}$ has characteristic power series 
$\frac{x/2}{\sinh x/2}$ 
and $\widehat{BC}$ has characteristic power series on a Cartan subalgebra given by the product of $\cosh(y/2)$ and $\cos(z/2)$.
\end{thm}

\begin{rmk}
The characteristic classes determined by Chern-Weil theory are independent of the choice of connection. 
For the same reasons, the right-hand side of (\ref{eq-SAIT}) does not depend on the choice of connection. 
\end{rmk}

This result is a super-version of Theorem \ref{thm-AITintro}. 
The invariant $\msf{Ev}_\bb{M}(\mrm{Tr}_\bb{M})$ has replaced evaluation on the identity function, $t^\nabla_M(1)$, appearing in the left-hand side of (\ref{eq-AIT}). 
On the right-hand side of (\ref{eq-SAIT}), we see (up to a constant multiplier) the manifold invariant given by integrating some characteristic series over the supermanifold $\bb{M}$. 
One should view this in analogue with the right-hand side of (\ref{eq-AIT}), 
where we saw the integral of the $\widehat{A}$-class. 
In particular, the $\widehat{A}$-class is replaced in the super-case with the product of the classes $\widehat{A}$, $\widehat{B}$, and $\widehat{C}$.

\subsection{Linear Overview}

This paper is broken into two parts. 
The first part establishes the conceptual framework we use, 
proving the general results we will employ in the second part. 
The second part contains all of the computations and constructions. 

We begin in \S\ref{sec-Background} by reviewing the main results of \cite{My2}.
In \S\ref{subsec-SGKD}, we recall the super version of Gelfand-Kazhdan descent (sGK descent) that will be our main tool for globalizing local results. 
Subsequently, in Theorem \ref{thm-DescA} we introduce the deformation of functions on a symplectic supermanifold 
and then record some basic results on integration over supermanifolds in \S\ref{subsec-IntandOri}. 

In \S\ref{sec-DescendingSupertraces}, we prove that sGK descent takes supertraces to supertraces. 
We then introduce normalization conditions in \S\ref{sec-UniquenessofSupertraces} that uniquely determine the traces we will construct in Part 2. 
The supertrace invariants we will be interested in are defined in \S\ref{sec-Evaluation}. 
Therein, we also prove that the local version of this invariant determines the global invariant of the descended trace, 
see Theorem \ref{thm-GlobalizeInvariant}. 

Part 2 begins in \S\ref{sec-Quad} with a review of quadratic forms, 
including useful results about their corresponding Clifford algebras. 
In \S\ref{sec-ConstructingtheSupertrace}, 
we give a formula for the supertrace on the Weyl-Clifford algebra following \cite{Engeli}. 
We justify how this rather complicated formula comes from a more obviously canonical one in \S\ref{subsec-Lifting}. 
The bulk of computational content of this paper is contained in \S\ref{sec-CompEval}. 
Therein, we compute the invariants defined in \S\ref{sec-Evaluation}. 
The superalgebraic index theorem is proven in \S\ref{sec-GSAIT}. 

\subsection{Conventions}

\begin{notation}
We set the following notation
\begin{itemize}

\item We let $\mrm{HH}_\bullet(A)$ denote Hochschild homology with coefficients in $A$, and $\mrm{Hoch}_\bullet(A)$ the Hochschild chains. 

\item We let $\mrm{HH}^\bullet(A)$ denote Hochschild cohomology with coefficients in $A^*$, and $\mrm{Hoch}^\bullet(A)$ the Hochschild cochains. 

\item Let $\bb{k}$ be either $\bb{K}$ or $\bb{C}$. 

\item $\hbar$ is a free variable we use as our deformation parameter.  

\item Let $\bb{K}=\bb{k}[[\hbar]]$ 

\item Given an object $R$ with a $\bb{Z}/2$-grading, we let $\Pi R$ denote the parity shift of $R$ with opposite grading. In particular, $\Pi\bb{K}=\bb{k}^{0\vert 1}[[\hbar]]$. 

\item We let $(\bb{M},\omega)$ denote a symplectic supermanifold of type $(2n\vert a,b)$. 

\item The coordinates of $\bb{R}^{2n\vert a,b}$ are given by $p_1,\dots,p_n,q_1,\dots,q_n,\theta_1,\dots,\theta_{a+b}$ where $p_i,q_i$ are even and $\theta_i$ are odd.

\end{itemize}
\end{notation}

\subsection{Acknowledgements}
The author would like to thank Owen Gwilliam and Brian Williams for many useful conversations and comments on previous drafts, 
and Dylan Wilson for his help with the computation in \S\ref{subsec-PfThm}. 
Thanks are also due to Michele Vergne for finding errors in a previous version. 
The author thanks Giovanni Felder for pointing out an error in a previous version that resulted in a contradiction with Engeli's work \cite{Engeli}. 
The author was supported by NSF Grant No. 1122374 while completing this work.

\part{Formal Theory}

\section{Background}\label{sec-Background}

We recall the results of \cite{My2} that we will use below. 
For a review of symplectic supermanifolds; see \cite[\S 2]{My2}. 
In \S\ref{subsec-SGKD}, 
we review the super version of Gelfand-Kazhdan descent, 
which allows us to globalize results from the formal disk to an entire supermanifold. 
We then review the algebras whose supertraces we will be interested in, \S\ref{subsec-WCA}. 
We end this section by discussing how to integrate over a symplectic supermanifold 
and defining an odd volume form. 

\subsection{Super-Gelfand-Kazhdan Descent}\label{subsec-SGKD}

All symplectic supermanifolds $(\bb{M},\omega)$ will be assumed to be of type $(2n\vert a,b)$; see \cite[Def. 2.17]{My2}. 
That is, $\bb{M}$ has $2n$ even dimensions, $a+b$ odd dimensions, 
and the symplectic structure in the odd direction comes from a quadratic form of signature $(a,b)$. 

Recall the super-Harish-Chandra (sHC) pair $\gKh$ from \cite[Conv. 3.18]{My2} and \cite[Notation 5.6]{My2} 
where $\mrm{Sp}(2n\vert a,b)$ is the Lie supergroup of linear automorphisms of a symplectic super vector space,  
and $\mfrk{g}^\hbar_{2n\vert a,b}$ is the Lie superalgebra of derivations of the Weyl-Clifford algebra. 

Super-Gelfand-Kazhdan descent is a fancy version of the Borel construction for modules over the sHC pair $\gKh$. 
By \cite[Thm. 3.22]{My2} and \cite[\S 5.0.1]{My2} 
there is a category $\msf{sGK}^{=,\hbar}_{2n\vert a,b}$ of symplectic supermanifolds equipped with an $\hbar$-formal exponential 
and a functor from $\msf{sGK}^{=,\hbar}_{2n\vert a,b}$ to principal $\gKh$-bundles, 

\begin{align}\label{eq-sGK}
\msf{sGK}^{=,\hbar}_{2n\vert a,b}\rta \msf{Bun}^\mrm{flat}_{\gKh}.
\end{align}

Given a pair $((\bb{M},\omega),\sigma)\in\msf{sGK}^{=,\hbar}_{2n\vert a,b}$ of a symplectic supermanifold $(\bb{M},\omega)$ and $\hbar$-formal exponential $\sigma$, 
one assigns the symplectic frame bundle $\Fr_\bb{M}$ and a flat connection 
\begin{align}\label{eq-A}
A\in\Omega^1(\Fr_\bb{M};\mfrk{g}_{2n\vert a,b}^\hbar)
\end{align}
determined by $\sigma$. 

Composing the functor (\ref{eq-sGK}) with a version of the Borel construction, 
we obtain our desired functor.

\begin{defn}\label{def-sGKDescent}
The \emph{super-Gelfand-Kazhdan descent functors} are the functors obtained from \cite[Ex. 3.50]{My2} and \cite[\S 5.0.1]{My2} by varying $(P,\nu)$ over $\msf{sGK}^{=,\hbar}_{2n\vert a,b}$, 

\[\mrm{desc}^\mrm{sGK}\colon \msf{sGK}^{=,\hbar}_{2n\vert a,b}\times \msf{Mod}_{\widehat{\mathcal{O}}_{2n\vert a,b}}(\msf{Mod}_{\gKh})\rta\msf{Pro}(\msf{VB})^\mrm{flat}\]
and

\[\mathbf{desc}^\mrm{sGK}\colon(\msf{sGK}^{=,\hbar}_{2n\vert a,b})^\mrm{op}\times \msf{Mod}_{\widehat{\mathcal{O}}_{2n\vert a,b}}(\msf{Mod}_{\gKh})\rta\msf{Mod}_{\Omega^\bullet_\bb{M}}.\]

For $((\bb{M},\omega),\sigma)\in \msf{sGK}_{2n\vert a,b}$, let $\mathbf{desc}_{(\bb{M},\sigma)}$ denote the resulting functor between module categories.
\end{defn}

\noindent This is the version of \cite[Def. 3.51]{My2} considered in \cite[\S 5.0.1]{My2}.

More concretely, 
$\mathbf{desc}_{(\bb{M},\sigma)}$ sends a $\widehat{\mathcal{O}}_{2n\vert a,b}$-module $V$ in $\gKh$-modules to the de Rham forms with differential induced from the connection $\nabla$, 

\begin{align}\label{eq-desc}
\mathbf{desc}_{(\bb{M},\sigma)}(V)=(\Omega^\bullet(\bb{M},\Fr_\bb{M}\times_{\mrm{Sp}(2n\vert a,b)}V),d_\nabla),
\end{align}

where $\nabla$ is a connection on the Borel construction induced via $A$ from the $\hbar$-formal exponential~$\sigma$.

\subsubsection{Weyl and Clifford Algebras}\label{subsec-WCA}

The main application in \cite{My2} of super-Gelfand-Kazhdan descent was to produce a deformation of $\mathcal{O}_{\bb{M}}$ using a deformation of functions on the formal disk. 
Here, we will produce supertraces on these formally local and global deformations. 
Fix an object $(\bb{M},\sigma)$ in $\msf{sGK}^{=,\hbar}_{2n\vert a,b}$.

The deformation of $\widehat{\mathcal{O}}_{2n\vert a,b}$ that is studied in \cite{My2} is 

\[\AQ\cong\widehat{\mrm{Weyl}}(T^*\bb{R}^n,\omega_0)\otimes\mrm{Cliff}(\bb{R}^{0\vert a+b},Q).\]
Here, $\widehat{\mrm{Weyl}}$ is the completed Weyl algebra and $\mrm{Cliff}$ is the Clifford algebra; see \cite[Def. 5.4]{My2} and \cite[\S 5.1]{My2}. 

The following is \cite[Thm. 5.12]{My2}. 

\begin{thm}\label{thm-DescA}
The algebra 
\[\mathcal{A}_\sigma(\bb{M}):=\Gamma_\nabla\left(\bb{M},\mathbf{desc}_{(\bb{M},\sigma)}(\AQ)\right)\]
is a deformation of the super Poisson algebra $\mathcal{O}_{\bb{M}}$.
\end{thm}

Note that the algebra $\mathcal{A}_\sigma(\bb{M})$ is the zeroth cohomology of the $\Omega_\bb{M}^\bullet[[\hbar]]$-algebra $\Dsc(\AQ)$.

\subsubsection{Connections}\label{sec-Connections}

The Lie superalgebra $\mfrk{g}_{2n\vert a,b}^\hbar$ is defined to be derivations of $\AQ$, 
\[\mfrk{g}_{2n\vert a,b}^\hbar=\mrm{Der}(\AQ).\]
We have a central extension of Lie superalgebras 
\[0\rta Z\rta \AQ\rta\mrm{Der}(\AQ)\rta 0\]
where $Z$ is the center of $\AQ$, viewed as an abelian Lie algebra. 
The connection 1-form $A$ from (\ref{eq-A}) can be lifted to a connection 1-form $\tilde{A}$ with values in $\AQ$. 
The curvature $F_{\tilde{A}}$ of $\tilde{A}$ is then a closed 2-form with values in the center $Z$ of $\AQ$. 
Moreover, $F_{\tilde{A}}$ is a $\mrm{Sp}(2n\vert a,b)$-basic form. 
We can therefore view $F_{\tilde{A}}$ as an element
\[F_{\tilde{A}}\in H^2(\bb{M};\Fr_\bb{M}\times_{\mrm{Sp}(2n\vert a,b)} Z).\]
This class is independent of the choice of lift $\tilde{A}$ of $A$. 
Since $Z\simeq\bb{K}$, then $F_{\tilde{A}}$ becomes a class in $H^2(\bb{M};\bb{K})$. 
This $\bb{K}$-valued de Rham form is what is called the \emph{characteristic class} of the deformation $\AQ$.

One should compare the above to \cite[Pg. 18]{FFSh} and \cite[\S 4]{BK}.

\subsubsection{Integration and Orientations}\label{subsec-IntandOri}

Our manifolds are assumed to be compact and without boundary. 
Ordinary symplectic manifolds $(M,\omega)$ are orientable. 
If $M$ has dimension $2n$, then $\omega^n$ is a common choice of orientation. 
The underlying manifold of a symplectic supermanifold is symplectic, and hence orientable; 
let $[M]$ denote the fundamental class. 
The main use of an orientation for us will be to integrate over the manifold. 
See \cite{Urs} for an introduction to integrating over supermanifolds.
One can integrate along symplectic supermanifolds using a combination of integration along the underlying manifold and a Berezin integral. 
We define this locally in coordinates.

Locally, $\bb{M}$ is modeled on the symplectic supermanifold $\bb{R}^{2n\vert a+b}$. 
Let $\theta_1,\dots,\theta_{a+b}$ be local odd coordinates. 
Given a compactly supported function $f$ on $\bb{R}^{2n\vert a,b}$, we take 

\[\int_{\bb{R}^{2n\vert a,b}}f d\bb{M}=\int_{\bb{R}^{2n}}d[M]\int_{\Lambda^{a+b}}f d\Theta\]
where $\int_{\Lambda^{a+b}}d\Theta$ is the Berezin integral on $a+b$ Grassmann variables such that 
\[\int_{\Lambda^{a+b}}\theta_1\cdots\theta_n d\Theta=1.\]
\noindent See \cite[Part 1. Ch. 2. \S 2]{Berezin}.

To check that this globalizes to define an integration on $\bb{M}$, 
we need to check that it is invariant under coordinate changes coming from local super symplectomorphisms $\phi$. 
The change of variables formula for the Berezin integral can be found in \cite[Thm. 2.1, Part 2. Ch. 2. \S 2]{Berezin}. 
The Berezin transforms under $\phi$ by the Berezinian (\cite[2.2.11]{Berezin}) of the Jacobian of $\phi$. 
It therefore suffices to check that the Berezinian
of the Jacobian matrix of $\phi$ is 1, 
and that $\theta_1\cdots\theta_{a+b}$ globalizes to a function on $\bb{M}$. 

The Jacobian of a super symplectomorphism $\phi$ is a super symplectic matrix $M_\phi$, analogously to the even case. 
Just as the determinant of a symplectic matrix is 1, 
the Berezinian of a super symplectic matrix is 1. 
Indeed, since $M_\phi$ is a super symplectic matrix, we have 
\[M^{\mrm{sT}}_\phi H_QM_\phi= H_Q\]
where $H_Q$ is as in \S\ref{sec-Quad}. See also \cite[Rmk. 2.12]{My2}. 
Since the Berezinian of a product is the product of the Berezinians, we may reduce to computing $\mrm{Ber}(H_Q)$. 
Then Berezinian of the block matrix $H_Q$ maybe computed by the formula \cite[\S 2]{Berezin}

\[\mrm{Ber}\begin{bmatrix}A& B\\ C & D\end{bmatrix}=\frac{\det(A-BD^{-1}C)}{\det(D)}.\]

We defer the following computation to \S\ref{sec-Quad}. 

\begin{lem}\label{lem-ThetaM}
One can choose an ordering of local odd coordinates $\theta_1,\dots,\theta_{a+b}$ on the symplectic supermanifold $\bb{R}^{2n\vert a+b}$ 
so that the function 
\[\Theta\colon\bb{R}^{2n\vert a,b}\rta\bb{R}\]
given by $\Theta=\theta_1\cdots\theta_{a+b}$ is invariant under local symplectomorphisms of $\bb{M}$ and therefore defines a global function on $\bb{M}$, 
\[\Theta_\bb{M}\in\mathcal{O}_\bb{M}.\] 
\end{lem}

The term \emph{volume form} is sometimes used to mean the data needed to integrate against. 
As $\mathcal{A}_\sigma(\bb{M})$ is a deformation of $\mathcal{O}_\bb{M}$, 
we have an equivalence of modules 
\[\mathcal{A}_\sigma(\bb{M})\simeq\mathcal{O}_\bb{M}[[\hbar]].\]
We may therefore view $\Theta_\bb{M}$ as an element of $\mathcal{A}_\sigma(\bb{M})$. 

\begin{defn}\label{def-VF}
The \emph{volume form} on $\bb{M}$ is $\Theta_\bb{M}\in\mathcal{A}_\sigma(\bb{M})$.
\end{defn}

We can see the volume form in terms of the Weyl-Clifford algebra as well.

\begin{rmk}
Recall that $\mathcal{A}_\sigma(\bb{M})$ is the result of applying super-Gelfand-Kazhdan descent to the tensor product

\[\AQ\cong\widehat{\mrm{Weyl}}(T^*\bb{R}^n,\omega_0)\otimes\mrm{Cliff}(\bb{R}^{0\vert a+b},Q).\]

After possibly scaling, 
the element $\Theta_\bb{M}$ locally looks like the unit in the completed Weyl algebra tensored with 
\[\Theta=\theta_1\cdots\theta_{a+b}\]
 in the Clifford algebra. 
\end{rmk}

\begin{ex}
Identify $\bb{M}$ as $E[1]$ for $E\rta M$ a quadratic vector bundle on a symplectic manifold as in Rothstein's theorem \cite{Rothstein}. 
Then functions on $\bb{M}$ are given by sections of the exterior product bundle, $\mathcal{O}_\bb{M}=\Gamma(\bb{M},\Lambda^\bullet E)$. 
If $E$ is oriented, 
we can take $\Theta_\bb{M}$ to be given by the corresponding section of $\Lambda^{\mrm{rank}(E)}E$, 
sending a point of $M$ to the top dimensional form on $E$ from the orientation. 
This is also explained in, for example, \cite[Pg. 26]{Engeli}.

\begin{defn}\label{def-Berezin}
Let $\bb{M}$ be a symplectic supermanifold with reduced manifold $M$.  
Let $d[M]$ denote the volume form on $M$ induced from the symplectic structure. 
\emph{Integration} over $\bb{M}$ is given by the map 
$\int_\bb{M}\colon \Omega^{2n}_\mrm{dR}(\bb{M})\rta\bb{K} $ 
defined by 
\[\int_\bb{M}=\int\left(\int(-)d\Theta_\bb{M}\right) d[M].\] 
\end{defn}
\end{ex}

\section{Descending Supertraces}\label{sec-DescendingSupertraces}

The goal of this section is to show how the notion of supertraces interacts with the super-Gelfand-Kazhdan descent functor. 
In particular, we would like a way of procuring supertraces on the deformed superalgebra 
$\mathcal{A}_\sigma(\bb{M})$ of Theorem \ref{thm-DescA}. 
By definition, the superalgebra $\mathcal{A}_\sigma(\bb{M})$ is obtained by applying super-Gelfand-Kazhdan descent to the superalgebra $\AQ$. 
We will show that a derived supertrace on $\AQ$ descends to a supertrace on $\mathcal{A}_\sigma(\bb{M})$. 

We begin by defining the types of maps we want. 

\noindent Note that for a superalgebra $A$ over $\bb{K}$, 
the Hochschild homology groups $\mrm{HH}_i^\bb{K}(A;A)$ are $\bb{Z}/2$-graded as well. 

\begin{rmk}
Ordinarily, we would like to use super-Gelfand-Kazhdan descent
for the super-Harish-Chandra pair. 
$\gKh$.
This descent functor involves the principal $\mathrm{Sp}(2n\vert a,b)$
bundle $\mathrm{Fr}_{\bb{M}}^{\mathrm{Sp}(2n\vert a,b)}$.
By Rothstein's theorem \cite{Rothstein},
we can view our symplectic supermanifold $\bb{M}$ 
as coming from an ordinary symplectic manifold equipped with a vector bundle with nondegenerate metric and a connection.
Using this structure, 
we get a reduction of structure group of 
$\mathrm{Fr}_{\bb{M}}^{\mathrm{Sp}(2n\vert a,b)}$
from $\mathrm{Sp}(2n\vert a,b)$ to the subgroup
$\mathrm{Sp}(2n)\times\mathrm{SO}(a,b)$. 
Following \cite[Ex. 3.6]{My2}, 
we get a sHC pair 
\[(\mathfrak{g}^\hbar_{2n\vert a,b},\mathrm{Sp}(2n)\times\mathrm{SO}(a,b)).\]
We can then instead consider super-Gelfand-Kazhdan descent
for this sHC pair. 

As remarked by Willwacher in \cite[Appendix E]{WillwacherE},
the supertrace we will consider here
and that in \cite{Engeli} 
has nice invariance under $\mathrm{Sp}(2n)\times\mathrm{SO}(a,b)$
but not under general elements of $\mathrm{Sp}(2n\vert a,b)$.
We will therefore need to consider descent for the former group.
Discussion of traces and descent along the latter group is 
included as well.
When the sHC pair is implicit or either choice is allowed,
we will use the notation $\Dsc$.
\end{rmk}

\begin{defn}\label{def-HCTrace}
Let $(\mathfrak{g},K)$ be an super-Harish-Chandra pair and let $A\in\mrm{Alg}(\mrm{Mod}_{(\mfrk{g},K)})$. 
A ($i$-derived) \emph{supertrace} on $A$ is a morphism 
$t\colon A^{\otimes_\bb{K} i+1}\rightarrow\bb{K} $ 
in $\mrm{Mod}_{(\mfrk{g},K)}$ such that the composition of $t$ with the differential $\del$ of the Hochschild complex,

\[
A^{\otimes i+2}\xrta{\del} A^{\otimes i+1}\xrta{t}\bb{K},\]
is zero.
\end{defn}

We allow $t$ to be either an even or odd map.

\begin{ex}
The superalgebra $\AQ$ is an algebra object in $\mrm{Mod}_{\gKh}$. 
Restricting the action to a sub-Lie group, 
we can view this as an object in 
\[\mrm{Mod}_{(\mathfrak{g}^\hbar_{2n\vert a,b},\mathrm{Sp}(2n)\times\mathrm{SO}(a,b))}\]
as well.
We construct a $2n$-derived supertrace $(\AQ)^{\otimes 2n+1}\rta\bb{K}$ in Theorem \ref{thm-DescendedTrace}.
\end{ex}

We will eventually restrict to supertraces satisfying an additional property.

\begin{defn}\label{def-relative}
Let $t$ be an $i$-derived supertrace on $A\in\mrm{Alg}(\mrm{Mod}_{(\mfrk{g},K)})$. 
Let $\rho\colon\mfrk{g}\rta\mrm{End}(A)$ denote the action of $\mfrk{g}$. 
Say $t$ is a \emph{relative} supertrace if for every $x\in\mrm{Lie}(K)$ and $a_0,\dots,a_{i-1}\in A$ we have 
\[\sum_{j=1}^{i}(-1)^jt(a_0\otimes a_1\otimes\cdots\otimes a_{i-1}\otimes a\otimes a_i\otimes\cdots\otimes a_{i-1})=0\]
where $a=\rho(x)\cdot 1$.
\end{defn}

We will see later (Lemma \ref{lem-center}) that relative supertraces correspond to relative Lie algebra cohomology classes.

\begin{defn}\label{def-TraceModR}
Let $R$ be a $\bb{K} $-superalgebra. Let $B\in\mrm{Alg}(\mrm{Mod}_R)$. 
An ($i$-derived) \emph{supertrace} on $B$ is a morphism 
$t\colon B^{\otimes_R i+1}\rightarrow R$ in $\mrm{Mod}_R$ 
such that the composition of $t$ with the differential $\del$ of the Hochschild complex,
\[B^{\otimes i+2}\xrta{\del}B^{\otimes i+1}\xrta{t}\bb{K},\]
is zero.
\end{defn}

\begin{ex}
Take $R=\Omega^\bullet_\mrm{dR}(\bb{M};\bb{K})$. 
Then 
\[\Dsc(\AQ)=\Omega^\bullet(\bb{M}; \Fr_\bb{M}\times_{\mrm{Sp}(2n\vert a,b)}\AQ)\]
is an algebra in $\mrm{Mod}_{R}$. 
Note that we have an isomorphism
\[\Dsc(\AQ)\otimes_R\Dsc(\AQ)\cong\Omega^\bullet\left(\bb{M}; \Fr_\bb{M}\times_{\mrm{Sp}(2n\vert a,b)}(\AQ\otimes_\bb{K}\AQ)\right)\]
and similarly for $\mathrm{Sp}(2n)\times\mathrm{SO}(a,b)$.
\end{ex}

\begin{ex}
Take $R=\bb{K}$. 
Then $\mathcal{A}_\sigma(\bb{M})=\Gamma_\nabla(\bb{M}, \Dsc(\AQ))$ is an algebra in $\mrm{Mod}_\bb{K}$.
\end{ex}

\begin{lem}
The functor 
\[\Dsc\colon\mrm{Alg}(\mrm{Mod}_{\gKh})\rta\mrm{Alg}(\mrm{Mod}_{\Omega^\bullet_\bb{M}})\]
from Equation (\ref{eq-sGK}) 
sends derived supertraces to derived supertraces,
and similarly for $\mathrm{Sp}(2n)\times\mathrm{SO}(a,b)$.
\end{lem}

\begin{proof}
Let $A$ be either in $\mrm{Alg}(\mrm{Mod}_{\gKh})$ 
or for $\mathrm{Sp}(2n)\times\mathrm{SO}(a,b)$, the proofs are the same. 
Let $t\colon A\rta\bb{K}$ be a supertrace on $A$. 
Applying $\Dsc$ to $t$, 
we obtain a map
\[\Dsc(t)\colon\Dsc(A)\rta\Dsc(\bb{K}).\]
By \cite[Cor. 3.47]{My2} and \cite[Lem. 5.10]{My2},  
$\Dsc(\bb{K})\cong\Omega^\bullet_\mrm{dR}(\bb{M};\bb{K})$. 
Thus $\Dsc(t)$ has the right domain and codomain to be a supertrace on $\Dsc(A)$, 
as in Definition \ref{def-TraceModR}.
Since $t$ is a supertrace, 
we have a commutative diagram

\[\begin{xymatrix}
{A\otimes A\arw[r]^{\mrm{mult}}\arw[dd]_{\mrm{swap}} & A\arw[dr]^t & \\
& & \bb{K}\\
A\otimes A\arw[r]_{\mrm{mult}} & A\arw[ur]_t & }
\end{xymatrix}\]
where $\mrm{mult}$ is the multiplication map for $A$. 
Since the functor $\Dsc$ is symmetric monoidal, 
it takes the swap map to the swap map, 
and $\mrm{mult}$ to the multiplication map $\mrm{mult}'$ for $\Dsc(A)$. 
Therefore, applying $\Dsc$ to this commutative diagram, 
we obtain a new commutative diagram 

\[\begin{xymatrix}
{\Dsc(A)\otimes \Dsc(A)\arw[r]^-{\mrm{mult}'}\arw[dd]_{\mrm{swap}} & \Dsc(A)\arw[dr]^{\Dsc(t)} & \\
& & \Omega^\bullet_{\bb{M}}[[\hbar]]\\
\Dsc(A)\otimes \Dsc(A)\arw[r]_-{\mrm{mult}'} & \Dsc(A)\arw[ur]_{\Dsc(t)} & }
\end{xymatrix}\]
which implies that $\Dsc(t)$ is a supertrace.

For a derived supertrace $T\colon A^{\otimes i+1}\rta\bb{K}$, 
the functoriality of $\Dsc$ implies that the face and degeneracy maps in the Hochschild complex of $A$ map to those in $\mrm{Hoch}_\bullet(\Dsc(A))$. 
Thus, the composition 
\[(\Dsc(A))^{\otimes i+1})\xrta{\del}(\Dsc(A))^{\otimes i+1}\xrta{\Dsc(T)}\Dsc(\bb{K})\] 
which shows that $\Dsc(T)$ is a derived supertrace on $\Dsc(A)$.
\end{proof}

\begin{ex}
Let $t_{2n\vert a,b}$ be a $2n$-derived supertrace on $\AQ$. 
Then $\mathbf{desc}_{(\bb{M},\sigma)}(t_{2n\vert a,b})$, 
as a map
\[\mathbf{desc}_{(\bb{M},\sigma)}(t_{2n\vert a,b})\colon\mathbf{desc}_{(\bb{M},\sigma)}(\AQ)^{\otimes 2n+1}\rta\Omega_{\mrm{dR}}^\bullet(\bb{M};\bb{K})\]
is a $2n$-derived trace on $\mathbf{desc}_{(\bb{M},\sigma)}(\AQ)$. 
\end{ex}

We would like $t_{2n\vert a,b}$ to determine an \emph{underived} trace on $\mathcal{A}_\sigma(\bb{M})$. 
For this, we will need the internal and external product maps on Hochschild homology. 
The following holds for $\mathrm{Sp}(2n\vert a,b)$ or $\mathrm{Sp}(2n)\times\mathrm{SO}(a,b)$.
First, note that 
\[\mathbf{desc}_{(\bb{M},\sigma)}(\AQ)=\Omega^\bullet(\bb{M};\Fr_\bb{M}\times_{\mrm{Sp}(2n\vert a,b)}\AQ)\]
is graded by degree of forms. 
Its Hochschild complex is therefore bigraded: 
one grading from degree of forms and one grading from the Hochschild complex.

Let $d_\mrm{dR}$ denote the de Rham differential on $\Omega^\bullet_\mrm{dR}(\bb{M};\bb{K})$, 
and $d_\nabla$ the differential on 
$\mathbf{desc}_{(\bb{M},\sigma)}(\AQ)^{\otimes 2n+1}$ 
induced from the covariant derivative. 
By the proof of \cite[Thms. 2.14 and 2.15]{Engeli}, 
we have 
\[d_\mrm{dR}\circ\mathbf{desc}_{(\bb{M},\sigma)}(t_{2n\vert a,b})=\mathbf{desc}_{(\bb{M},\sigma)}(t_{2n\vert a,b})\circ d_\nabla.\]
Thus, $\mathbf{desc}_{(\bb{M},\sigma)}(t_{2n\vert a,b})$ descends to a map on cohomology 
\[H^\bullet_\nabla\left(\bb{M};\Fr_\bb{M}\times_{\mrm{Sp}(2n\vert a,b)}(\AQ)^{\otimes 2n+1}\right)\rta H^\bullet_\mrm{dR}(\bb{M};\bb{K})\]
which does not preserve degree, as we have moved the $\otimes (2n+1)$ inside.

We would like to create an \emph{underived} trace on $\mathcal{A}_\sigma(\bb{M})$ from $\mathbf{desc}_{(\bb{M},\sigma)}(t_{2n\vert a,b})$. 
To do this, we will use the product structure on forms on $\bb{M}$ valued in $\Fr_\bb{M}\times_{\mrm{Sp}(2n\vert a,b)}(\AQ)$.

Let 
\[\tilde{A}\in\Omega^1(\bb{M}; \Fr_\bb{M}\times_{\mrm{Sp}(2n\vert a,b)}\AQ)\simeq\Omega^1(\Fr_\bb{M};\AQ)_\mrm{basic}\] 
denote the connection 1-form on the flat pro-bundle 
$\mrm{desc}_{(\bb{M},\sigma)}(\AQ)$, 
see \S\ref{sec-Connections}.

Wedging with the form $\tilde{A}^{\wedge 2n}$ determines a degree $2n$ map
\[\beta_{\tilde{A}}\colon\Omega^0\left(\bb{M};\Fr_\bb{M}\times_{\mrm{Sp}(2n\vert a,b)}\AQ\right)\rta \Omega^{2n}\left(\bb{M};\Fr_\bb{M}\times_{\mrm{Sp}(2n\vert a,b)}(\AQ)^{\otimes 2n+1}\right)\]
of superalgebras. 
Here, the algebra structure on both sides is induced from the algebra structure on $\AQ$. 
One should compare $\beta_{\tilde{A}}$ to part of what is denoted $\chi_0$ in \cite[\S\S 2.5.1-2.5.2]{Engeli}. 
Taking cohomology, the domain of $\beta_{\tilde{A}}$ becomes $\mathcal{A}_\sigma(\bb{M})$. 
We can then form the composite
\[\mathcal{A}_\sigma(\bb{M})\xrta{\beta_{\tilde{A}}}H^\bullet_\nabla\left(\bb{M};\Fr_\bb{M}\times_{\mrm{Sp}(2n\vert a,b)}(\AQ)^{\otimes 2n+1}\right)\xrta {\mathbf{desc}_{(\bb{M},\sigma)}(t_{2n\vert a,b})}H^\bullet_\mrm{dR}(\bb{M};\bb{K}).\]
The next proposition follows from \cite[Thms. 2.14 and 2.15]{Engeli}.

\begin{prop}\label{prop-inducedtrace}
Let $t_{2n\vert a,b}$ be a $2n$-derived supertrace on $\AQ$. 
The map on cohomology
\[\mathcal{A}_\sigma(\bb{M})\rta H^{2n}_\mrm{dR}(\bb{M};\bb{K}),\]
induced from $\mathbf{desc}_{(\bb{M},\sigma)}(t_{2n\vert a,b})\circ\beta_{\tilde{A}}$, 
is linear and 
factors through $\mrm{HH}_0(\mathcal{A}_\sigma(\bb{M}))$. 
\end{prop}

To obtain a $\bb{K}$-valued supertrace on $\mathcal{A}_\sigma(\bb{M})$, 
we need to get from $H^{2n}_\mrm{dR}(\bb{M};\bb{K})$ to $\bb{K}$. 
We will do this by integrating out the odd directions with a Berezin integral, 
and then using the orientation on the underlying manifold $M_0$ of $\bb{M}$. 
See Definition \ref{def-Berezin}.

\begin{cor}\label{cor-TmTrace}
The composite

\[t_\bb{M}=\int_\bb{M}\left(\mathbf{desc}_{(\bb{M},\sigma)}(t_{2n\vert a,b})\circ\beta_{\tilde{A}}\right)\colon\mathcal{A}_\sigma(\bb{M})\rta\bb{K}\]
is a supertrace of $\bb{K}$-algebras.
\end{cor}

\begin{proof}
The map $\int_\bb{M}$ is $\bb{K}$-linear. 
By Proposition \ref{prop-inducedtrace}, 
the composite $t_\bb{M}$ is linear and factors through 
$\mrm{HH}_0(\mathcal{A}_\sigma(\bb{M})$.
\end{proof}

\section{Uniqueness of supertraces}\label{sec-UniquenessofSupertraces}

We discuss reasonable properties we would like our supertraces (locally on $\AQ$ and globally on $\mathcal{A}_\sigma(\bb{M})$) to satisfy, 
and show that these properties uniquely determine such a supertrace. 

\subsection{Local Uniqueness} 

Here, we show that there is a unique derived supertrace on $\AQ$ up to scalar multiple. 

We can reinterpret a supertrace as an element of the Hochschild cohomology using the following lemma. 
For $A\in\mrm{Alg}_{\bb{K}}(\mrm{Mod}_{(\mfrk{g},K)})$, 
let $A^*$ denote the dual, 
$A^*=\mrm{Hom}(A,\bb{K})$. 
The following can be found after Theorem 2.1 in \cite{FFSh}.

\begin{lem}\label{lem-tracehh}
Let $A\in\mrm{Alg}_{\bb{K}}(\mrm{Mod}_{(\mfrk{g},K)})$. 
There is an equivalence 

\[\mrm{Hom}\left(\mrm{Hoch}_\bullet^{\bb{K}}(A;A),\bb{K}\right)\cong \mrm{Hoch}^\bullet_{\bb{K}}(A;A^*). \]
\end{lem}

To show that there is a unique derived supertrace on $\AQ$ up to scalar multiple, 
it suffices to prove that $\mathrm{HH}^{\bb{K}}_\bullet(\AQ)$ is one-dimensional.

By \cite[\S 5.1]{My2}, 
we have an equivalence
\[\AQ\cong\widehat{\mrm{Weyl}}(T^*\bb{R}^n,\omega_0)\otimes_\bb{K}\mrm{Cliff}(\bb{R}^{0\vert a+b},Q).\]
Hochschild homology satisfies a K\"unneth formula. 
That is, as the Hochschild complex of a tensor product is the tensor product of the Hochschild complexes, we have 
\begin{align*}
\mrm{HH}^{\bb{K}}_\bullet(\AQ)\cong \mrm{HH}^{\bb{K}}_\bullet\left(\widehat{\mrm{Weyl}}(T^*\bb{R}^n,\omega_0)\right)\otimes_\bb{K} \mrm{HH}^{\bb{K}}_\bullet\left(\mrm{Cliff}(\bb{R}^{a+b},Q)\right).
\end{align*}

The Hochschild homology of the Weyl algebra is computed in \cite[Thm. 2.1]{FFSh}, 
where $\mrm{Weyl}(T^*\bb{R}^n,\omega_0)$ is denoted $\cal{A}^\mrm{pol}_{2n}$. 
We have an isomorphism
\[\mrm{HH}^{\bb{K}}_\bullet(\mrm{Weyl}(T^*\bb{R}^n,\omega_0))\cong {\bb{K}}[2n].\]

A computation of the Hochschild homology of the Clifford algebra can be found in 
\cite[\S 6 Proof of Prop. 1]{Kassel}
where it is shown that 
\[\mrm{HH}_\bullet^{\bb{K}}(\mrm{Cliff}(\bb{R}^{a+b},Q)=
\begin{cases}
\bb{K}[0] & \text{$a+b$ is even}\\
\Pi\bb{K}[0] & \text{$a+b$ is odd }
\end{cases}.
\]
See also \cite[Thm. 2.10]{Engeli}. 

Putting this together, we have the following computation.

\begin{cor}\label{cor-AQHH}
Let $\AQ$ be the super Fedosov quantization as in Theorem \ref{thm-DescA}.
There are isomorphisms
\[\mathrm{HH}^{\bb{K}}_\bullet(\AQ)\cong
\begin{cases}
\bb{K}[2n] & a+b\text{ is even}\\
\Pi\bb{K}[2n] & a+b\text{ is odd}
\end{cases}.\]
\end{cor}

\noindent See also \cite[\S 2.3]{Engeli}.

In particular, this means that for $a+b$ even, there is a unique (up to scalar) supertrace on $\AQ$ that is an \emph{even} map, 
and for $a+b$ odd, there is a unique (up to scalar) \emph{odd} supertrace on $\AQ$.

\subsection{Normalization Condition}\label{subsec_NormalCond}

We would like to put conditions on the type of supertraces $\mathcal{A}_\sigma(\bb{M})\rta\bb{K}$ that will uniquely determine it. 
Note that $\mathcal{A}_\sigma(\bb{M})$ is a deformation of $\mathcal{O}_\bb{M}$ in sheaves of algebras on $\bb{M}$. 
We will specify what our supertrace should look like locally over $\bb{M}$. 

Let $p_1,\dots,p_n,q_1,\dots,q_n,\theta_1,\dots,\theta_{a+b}$ be local coordinates for $\bb{M}$ as in Lemma \ref{lem-ThetaM}.

\begin{defn}\label{def-normalizedtrace}
Let $\mathcal{A}$ be a deformation of $\mathcal{O}_\bb{M}$ in sheaves of algebras on $\bb{M}$. 
A supertrace $t_\bb{M}\colon\mathcal{A}\rta\bb{K}$ on $\mathcal{A}$ is \emph{normalized} if 
on sufficiently small neighborhoods $\bb{R}^{2n\vert a+b}\subset\bb{M}$, the map is given by 
\[(t_{\bb{M}})\vert_U(f)=\left((-1)^{n+a+t}\hbar^n\int_{\bb{R}^{2n}}\left(\int f d\Theta\right)dq_1\wedge dp_1\wedge\cdots dq_n\wedge dp_n\right).\]
\end{defn} 

By \cite[\S 2.6]{Engeli} and \cite[Thm. 4.2]{FFSh}, 
a normalized supertrace on $\mathcal{A}_\sigma(\bb{M})$ is unique. 
Below, in Theorem \ref{thm-DescendedTrace}, we show the existence of such a normalized supertrace.

\section{Evaluation on a Volume Form}\label{sec-Evaluation}

In this section, 
we will define and study an invariant of symplectic supermanifolds using their normalized supertraces. 
In \S\ref{subsec-LSAIT} we will compute this invariant.

\begin{rmk}
The results and proofs in these sections work for either 
$\mathrm{Sp}(2n\vert a,b)$ or $\mathrm{Sp}(2n)\times\mathrm{SO}(a,b)$.
We have only stated them in the former case for simplicity.
\end{rmk}

Recall from \S\ref{subsec-IntandOri} that we have chosen a volume form 
$\Theta_\bb{M}\in\mathcal{A}_\sigma(\bb{M})$. 

\begin{defn}\label{def-EvalOnVol}
Let $t_{\bb{M}}$ be a supertrace on $\mathcal{A}_\sigma(\bb{M})$. 
The \emph{evaluation of $t_{\bb{M}}$ on the volume form} $\Theta_\bb{M}$ is 
\[\msf{Ev}_\bb{M}(t_\bb{M})=t_{\bb{M}}(\Theta_\bb{M})\in\bb{K}.\]
\end{defn}

Note that the definition of $\msf{Ev}_\bb{M}$ depends on a choice of volume form, though we suppress this from the notation.

Although $\msf{Ev}_\bb{M}$ may be defined for any supertrace on $\mathcal{A}_\sigma(\bb{M})$, 
we will only show that it is an invariant for \emph{normalized} supertraces.

\begin{lem}
Let $\varphi\colon (\bb{M},\omega,\sigma)\rta(\bb{M}',\omega',\sigma)$ be a morphism in 
$\msf{sGK}^{=,\hbar}_{2n\vert a,b}$. 
If $T$ and $T'$ are normalized supertraces on $\bb{M}$ and $\bb{M}'$, respectively, 
then $\msf{Ev}_\bb{M}(T)=\msf{Ev}_{\bb{M}'}(T')$.
\end{lem}

\begin{proof}
The local symplectomorphism $\varphi$ induces a pullback map $\varphi^*\colon\mathcal{O}_{\bb{M'}}\rta\mathcal{O}_{\bb{M}}$. 
By the functoriality of $\Dsc$, 
the pullback extends to a map 
\[\varphi^*\colon\mathcal{A}_{\sigma'}(\bb{M}')\rta\mathcal{A}_\sigma(\bb{M}).\]
We can check that the diagram 

\[\begin{xymatrix}
{
\mathcal{A}_\sigma(\bb{M})\arw[d]^T & \mathcal{A}_{\sigma'}(\bb{M}')\arw[l]_{\varphi^*}\arw[dl]^{T'}\\
\bb{K} & 
}
\end{xymatrix}\]
commutes by examining it locally in $\bb{M}$ and $\bb{M}'$.  
Over sufficiently small open subsets, 
the normalization condition guarantees that supertraces $T$ and $T'$ agree. 

Since $\varphi$ is a local symplectomorphism, $\varphi^*(\Theta_{\bb{M}'})=\Theta_{\bb{M}}$. 
Thus, we have 
\[\msf{Ev}_{\bb{M}'}(T')=T'(\Theta_{\bb{M}'})=T\varphi^*(\Theta_{\bb{M}'})=T(\Theta_\bb{M})=\msf{Ev}_\bb{M}(T).\]
\end{proof}

We will show that, for normalized supertraces, $\msf{Ev}_\bb{M}(t_\bb{M})$ can be, in a way, 
computed locally. 
In \S\ref{subsec-LocalInvariant}, 
we define an analogue of $\msf{Ev}_\bb{M}$ for supertraces on $\AQ$. 
One should think of this as making sense of $\msf{Ev}_{\widehat{\bb{D}}^{2n\vert a,b}}$. 
Then, in \S\ref{subsec-GlobalInvariant}, 
we show how this local analogue descends to $\msf{Ev}_\bb{M}$. 

\subsection{Formally Local Invariant}\label{subsec-LocalInvariant}

Consider a derived supertrace $t_{2n\vert a,b}$ on $\AQ$ as in Definition \ref{def-HCTrace}. 
If we think of $\AQ$ as a deformation of the formal disk $\widehat{\bb{D}}^{2n\vert a,b}$, 
then a formally local analogue of Definition \ref{def-EvalOnVol} would be to evaluate $t_{2n\vert a,b}$ on a volume form for $\widehat{\bb{D}}^{2n\vert a+b}$.  
If we write 
\[\mathcal{O}_{\widehat{\bb{D}}^{2n\vert a+b}}=\mathcal{O}_{\widehat{\bb{D}}^{2n}}\otimes\Lambda[\theta_1,\dots,\theta_{a+b}],\]
then our volume form will be $1\otimes \theta_1\cdots\theta_{a+b}=1\otimes\Theta$. 

If $t$ is an underived supertrace, then we may take $t_{2n\vert a,b}(1\otimes \Theta)$ to obtain an element of $\bb{K}$. 
Taking in to consideration that $t_{2n\vert a,b}$ is morphism of $\gKh$-modules, 
we will see that, in general, evaluation of $t_{2n\vert a,b}$ on $1\otimes\Theta$ naturally lives in $C^\bullet_\mrm{Lie}\left(\mfrk{g}^\hbar_{2n\vert a,b};\mfrk{sp}_{2n\vert a,b};\bb{K}\right)$. 

To interpolate between derived supertraces and Lie algebra cochains, we need the following observation. 
Given an algebra $A$, we have an anti-symmetrization map
\[(-)^\mrm{Lie}\colon \mrm{Hoch}^\bullet(A;A^*)\rta C^\bullet_\mrm{Lie}(A;A^*)\]
where on the right-hand side, we are viewing $A$ as a Lie algebra under the commutator. 
See, for example, \cite[Pg. 11]{FFSh}. 
The map $(-)^\mrm{Lie}$ is given by 
\[(b^\mrm{Lie})(a_1\otimes\cdots a_k)(a_0)=\sum_{s\in\Sigma_k}\mrm{sign}(s)b(a_0\otimes a_{s(1)}\otimes\cdots a_{s(k)}).\]
We would like to use this map to view a derived trace $t_{2n\vert a,b}\in\mrm{HH}^\bullet(\AQ)$ as a Lie algebra cocycle. 
To do so, we need a Hochschild cocycle representative $\widetilde{t}_{2n\vert a,b}$ of $t_{2n\vert a,b}$. 
The Lie algebra cocycle $\widetilde{t}_{2n\vert a,b}^\mrm{Lie}$ will, in fact, 
live in a \emph{relative} Lie algebra group. 
For an introduction to relative Lie algebra (co)homology; see \cite[\S 2.3]{Solleveld} or \cite[\S2.8.1]{Engeli}.

In particular, for a Lie subalgebra $\mfrk{h}\rta\mfrk{g}$, we have a map
\[C^\bullet_\mrm{Lie}(\mfrk{g},\mfrk{h})\rta C^\bullet_\mrm{Lie}(\mfrk{g}).\] 
For a $\mfrk{g}$-module $M$, the group $C^\bullet_\mrm{Lie}(\mfrk{g},\mfrk{h};M)$ is 
\[C^p_\mrm{Lie}(\mfrk{g},\mfrk{h};M)=\mrm{Hom}_\mfrk{h}(\wedge^p(\mfrk{g}/\mfrk{h}),M)\]
consisting of those cochains $c\in C^p_\mrm{Lie}(\mfrk{g};M)$ so that 

\begin{itemize}
\item $c(x_1\wedge\cdots\wedge x_p)=0$ if $x_i\in\mfrk{h}$ for any $i=1,\dots,p$, and 
\item $c$ is an $\mfrk{h}$ invariant map.
\end{itemize}
See \cite[Def. 2.17]{Solleveld} for the definition of relative Lie algebra cohomology with coefficients.

Recall the notion of relative supertraces from Definition \ref{def-relative}.

\begin{lem}\label{lem-center}
Let $Z$ denote the center of $\AQ$. 
If $t_{2n\vert a,b}$ is a derived relative supertrace on $\AQ$, 
then the element $t_{2n\vert a,b}^\mrm{Lie}\in C^\bullet_\mrm{Lie}(\AQ;(\AQ)^*)$ is in the image of the map
\[C^\bullet_\mrm{Lie}\left(\AQ,\mfrk{sp}_{2n\vert a,b}\oplus Z;(\AQ)^*\right)\rta C^\bullet_\mrm{Lie}\left(\AQ;(\AQ)^*\right) .\]
\end{lem}

For the purely even case, see \cite[\S 4.2 (ii)-(iii)]{FFSh}.

\begin{proof}
The invariance under $\mfrk{sp}_{2n\vert a,b}$ follows from the fact that $t_{2n\vert a,b}$ is $\gKh$-module map; see \cite[Thm. 2.11(ii)]{Engeli}. 
The condition that $t_{2n\vert a,b}$ is a relative supertrace translates to $\widetilde{t}^\mrm{Lie}_{2n\vert a,b}$ vanishing on $\mfrk{sp}_{2n\vert a,b}$ after noting that one can reduce from the sum over all elements of the symmetric group to just transpositions. 

Reduction to a cocycle relative to $Z$ follows from the fact that $t_{2n\vert a,b}$ came from a normalized cocycle. 
See \cite[Cor. 3.1]{FFSh}. 
\end{proof}

By \cite[Notation 5.6]{My2}, $\mfrk{g}^\hbar_{2n\vert a,b}$ is the Lie algebra of derivations of $\AQ$.  
Using the more general fact that $\mrm{Der}(A)=A/Z$, 
we see that the pair $(\mrm{Der}(A),\mfrk{h})$ is equivalent to $(A,\mfrk{h}\oplus Z)$. 
We may therefore view $\widetilde{t}_{2n\vert a,b}^\mrm{Lie}$ as an element in 

\[C^\bullet_\mrm{Lie}\left(\mfrk{g}^\hbar_{2n\vert a,b},\mfrk{sp}_{2n\vert a,b};\left(\AQ\right)^*\right).\]

Given an element $u\in\AQ$, we have an evaluation map 
\[\mrm{ev}_u\colon C^\bullet_\mrm{Lie}\left(\mfrk{g}^\hbar_{2n\vert a,b};\mfrk{sp}_{2n\vert a,b};(\AQ)^*\right)\rta C^\bullet_\mrm{Lie}\left(\mfrk{g}^\hbar_{2n\vert a,b};\mfrk{sp}_{2n\vert a,b};\bb{K}\right).\]
We will be interested in $\mrm{ev}_u(\widetilde{t}_{2n\vert a,b}^\mrm{Lie})$ for  
$u=1\otimes\Theta$. 

\begin{defn}\label{defn-EvalOnTraceLocal}
Given a $k$-derived relative supertrace $t_{2n\vert a,b}$ on $\AQ$ with cocycle representative $\widetilde{t}_{2n\vert a,b}$, 
the \emph{evaluation of $t_{2n\vert a,b}$ on the formal volume form} $1\otimes \Theta$ is 
\[\msf{Ev}_\mrm{loc}(\widetilde{t}_{2n\vert a,b})=\mrm{ev}_{1\otimes\Theta}(t_{2n\vert a,b}^\mrm{Lie})\in C^k_\mrm{Lie}\left(\mfrk{g}^\hbar_{2n\vert a,b};\mfrk{sp}_{2n\vert a,b};\bb{K}\right).\]
\end{defn}

\begin{ex}
If $t_{2n\vert a,b}\colon \AQ\rta\bb{K}$ is underived, 
then $\msf{Ev}_\mrm{loc}(t_{2n\vert a,b})$ lives in the zeroth cocycle group
\[C^0_\mrm{Lie}\left(\mfrk{g}^\hbar_{2n\vert a,b};\mfrk{sp}_{2n\vert a,b};\bb{K}\right)=\mrm{Hom}(\bb{K},\bb{K})\simeq \bb{K};\]
see \cite[Prop. 2.16(1)]{Solleveld}. 
We can identify $\msf{Ev}_\mrm{loc}(t_{2n\vert a,b})$ with $t_{2n\vert a,b}(1\otimes \Theta)$. 
Indeed, the map $(-)^\mrm{Lie}$ is the identity in degree zero.  
\end{ex}

\subsection{Globalizing the Invariant}\label{subsec-GlobalInvariant}

Let $t_{2n\vert a,b}$ be a derived supertrace on $\AQ$. 
By Corollary \ref{cor-AQHH}, $t_{2n\vert a,b}$ must be in degree $2n$, 
\[t_{2n\vert a,b}\in\mrm{HH}^{2n}(\AQ,(\AQ)^*).\]
By Corollary \ref{cor-TmTrace}, 
$t_{2n\vert a,b}$ determines a supertrace $t_\bb{M}$ on $\mathcal{A}_\sigma(\bb{M})$. 
The goal of this section is to recover $\msf{Ev}_\bb{M}(t_\bb{M})$ from $\msf{Ev}_\mrm{loc}(t_{2n\vert a,b})$.
This will be done using a variation of the Chern-Weil map.

Let $(\mfrk{g},K)$ be an HC pair. 
Let $P$ be a $(\mfrk{g},K)$-bundle. 
As in \cite[Def. 1.18]{GGW}, we have a natural transformation 

\[\mrm{char}_{P}\colon C^\bullet_\mrm{Lie}(\mfrk{g},\mrm{Lie}(K);-)\Rightarrow \mathbf{desc}_{P}(-)\]
between functors $\mrm{Mod}_{(\mfrk{g},K)}\rta \mrm{Ch}_\bb{K}$. 

Consider now the case of the sHC pair $\gKh$. 
Let $\mrm{char}_{(\bb{M},\sigma)}$ denote the natural transformation coming from the principal $\gKh$-bundle structure on $\Fr_\bb{M}$; 
see \cite[Cor. 3.35]{My2} and \cite[\S 5.0.1]{My2}. 
Then $\mrm{char}_{(\bb{M},\sigma)}$ is given as follows. 
Let $A\in \Omega^1(\Fr_\bb{M};\mfrk{g}^\hbar_{2n\vert a,b})$ be the flat connection 1-form. 
Then we have
\[A^{\wedge i}\in\Omega^i(\Fr_\bb{M};\Lambda^i\mfrk{g}^\hbar_{2n\vert a,b}).\]
View an element 
$r\in C^i_\mrm{Lie}(\mfrk{g}^\hbar_{2n\vert a,b},\mrm{sp}_{2n\vert a,b};V)$
 as a map 
 \[r\colon\Lambda^i(\mfrk{g}^\hbar_{2n\vert a,b}/\mfrk{sp}_{2n\vert a,b})\rta V.\] 
 Then $r$ induces a map on $\mrm{Sp}(2n\vert a,b)$-basic forms 
 \[r_*\colon\Omega^\bullet(\Fr_\bb{M};\Lambda^i\mfrk{g}^\hbar_{2n\vert a,b})_\mrm{basic}\rta\Omega^\bullet(\Fr_\bb{M};V)_\mrm{basic}=\mathbf{desc}_{(\bb{M},\sigma)}(V).\]
 By definition, $\mrm{char}_{(\bb{M},\sigma)}(r)=r_*(A^{\wedge i})$.

\begin{thm}\label{thm-GlobalizeInvariant}
Let $t_{2n\vert a,b}$ be a $2n$-derived relative supertrace on $\AQ$. 
For $(\bb{M},\sigma)\in\msf{sGK}^{=,\hbar}_{2n\vert a,b}$, 
let $t_\bb{M}$ denote the supertrace on $\mathcal{A}_\sigma(\bb{M})$ induced from $t_{2n\vert a,b}$ using Corollary \ref{cor-AQHH}.
Then
\[\msf{Ev}_\bb{M}(t_\bb{M})=
\int_\bb{M}\mrm{char}_{(\bb{M},\sigma)}(\bb{K})(\msf{Ev}_\mrm{loc}(t_{2n\vert a,b})).\]
\end{thm}

\begin{proof}

By naturality, we get a commutative diagram 

\[\begin{xymatrix}
{
C^\bullet_\mrm{Lie}(\mfrk{g}^\hbar_{2n\vert a,b},\mfrk{sp}_{2n\vert a,b};((\AQ)^*)\arw[rrr]^-{\mrm{char}_{(\bb{M},\sigma)}(\AQ)^*)}\arw[d]^{\mrm{ev}_{1\otimes\Theta}} & & & \mathbf{desc}_{(\bb{M},\sigma)}((\AQ)^*)\arw[d]^{\mrm{char}_{(\bb{M},\sigma)}(\mrm{ev}_{1\otimes\Theta})}\\
C^\bullet_\mrm{Lie}(\mfrk{g}^\hbar_{2n\vert a,b},\mfrk{sp}_{2n\vert a,b};\bb{K})\arw[rrr]_-{\mrm{char}_{(\bb{M},\sigma)}(\bb{K})} &  & & \mathbf{desc}_{(\bb{M},\sigma)}(\bb{K})
}
\end{xymatrix}.\]

We can construct a morphism of $\Omega^\bullet_\mrm{dR}(\bb{M};\bb{K})$-modules
\[B\colon \mathbf{desc}_{(\bb{M},\sigma)}\left((\AQ)^*\right)\rta \mrm{Hom}_{\Omega^\bullet_\mrm{dR}(\bb{M};\bb{K})}\left(\mathbf{desc}_{(\bb{M},\sigma)}(\AQ),\mathbf{desc}_{(\bb{M},\sigma)}(\bb{K}))\right)\]
as follows. 

By Construction (\ref{eq-desc}), the descent functor is given by 
\[\mathbf{desc}_{(\bb{M},\sigma)}(-)=\Omega^\bullet(\bb{M};\Fr_\bb{M}\times_{\mrm{Sp}(2n\vert a,b)}(-)).\]
Since the bundle of homomorphisms is formed fiberwise, 
we have an identification 
\[\Fr_\bb{M}\times_{\mrm{Sp}(2n\vert a,b)}(\AQ)^*\simeq \mrm{Hom}_\mrm{Bun}(\Fr_\bb{M}\times_{\mrm{Sp}(2n\vert a,b)}\AQ,\Fr_\bb{M}\times_{\mrm{Sp}(2n\vert a,b)}\bb{K}).\]
Letting $E$ denote the bundle $\Fr_\bb{M}\times_{\mrm{Sp}(2n\vert a,b)}\AQ$ and $E^\vee$ the bundle dual, 
we are therefore looking for a map 
\[B\colon\Omega^\bullet(\bb{M};E^\vee)\rta\mrm{Hom}_{\Omega^\bullet(\bb{M};\bb{K})}\left(\Omega^\bullet(\bb{M};E),\Omega^\bullet(\bb{M};\bb{K})\right).\]
By adjunction, this is the same as a map 
\[\Omega^\bullet(\bb{M};E^\vee)\otimes_{\Omega^\bullet(\bb{M};\bb{K})}\Omega^\bullet(\bb{M};E)\rta\Omega^\bullet(\bb{M};\bb{K}).\]
Such a map is given by the product of forms together with the evaluation map $E^\vee\otimes E\rta\underline{\bb{K}}_\bb{M}$. 
Note that this map sends a degree $i$ form and a degree $j$ form to a degree $i+j$ form.

We can therefore view 
$\left(B\circ\mrm{char}_{(\bb{M},\sigma)}\left((\AQ)^*\right)\right)(t_{2n\vert a,b}^\mrm{Lie})$ 
as a map 
\[\Omega^\bullet(\bb{M};\Fr_\bb{M}\times_{\mrm{Sp}(2n\vert a,b)}\AQ)\rta\Omega^{\bullet+2n}_\mrm{dR}(\bb{M};\bb{K}),\] 
where the degree shift is since $t_{2n\vert a,b}^\mrm{Lie}$ is a coycle of degree $2n$. 

In degree 0, this map is given by sending 
$a\in\Omega^0(\bb{M};\Fr_\bb{M}\times_{\mrm{Sp}(2n\vert a,b)}\AQ)$ 
to the $2n$ form~$(t_{2n\vert a,b}^\mrm{Lie})_*(A^{\wedge 2n})(a)$. 
Note that fiberwise we are evaluating the antisymmetrization $t_{2n\vert a,b}^\mrm{Lie}$ on a class $A^{\wedge 2n}$ in the diagonal, 
\[(t^\mrm{Lie}_{2n\vert a,b})_*(A^{\wedge 2n})(a)=\sum_{s\in \Sigma_{2n}}\mrm{sign}(s)t_{2n\vert a,b}(a\otimes A\otimes\cdots\otimes A)=t_{2n\vert a,b}(a\otimes A\otimes\cdots\otimes A).\]
Since $\tilde{A}$ is a lift of $A$ (see \S\ref{sec-Connections}), the map defined fiberwise by 
$t_{2n\vert a,b}(a\otimes A\otimes\cdots\otimes A)$ 
is exactly $\mathbf{desc}_{(\bb{M},\sigma)}(t_{2n\vert a,b})\circ\beta_{\tilde{A}}$ 
where $\beta_{\tilde{A}}$ is as in Proposition \ref{prop-inducedtrace}. 
Thus, in degree zero we have an equivalence of maps 
\[\left(B\circ\mrm{char}_{(\bb{M},\sigma)}\left((\AQ)^*\right)\right)(t_{2n\vert a,b}^\mrm{Lie} )=\mathbf{desc}_{(\bb{M},\sigma)}(t_{2n\vert a,b})\circ\beta_{\tilde{A}}.\]

Lastly, we claim that the diagram

\[\begin{xymatrix}
{
\mathbf{desc}_{(\bb{M},\sigma)}\left((\AQ)^*\right)\arw[r]^-B\arw[d]_{\mrm{char}_{(\bb{M},\sigma)}(\mrm{ev}_{1\otimes\Theta})} & \mrm{Hom}_{\Omega^\bullet_\mrm{dR}(\bb{M};\bb{K})}\left(\mathbf{desc}_{(\bb{M},\sigma)}(\AQ),\mathbf{desc}_{(\bb{M},\sigma)}(\bb{K}))\right)\arw[dl]^{\mrm{ev}_{\Theta_\bb{M}}}\\
\Omega^\bullet_\mrm{dR}(\bb{M};\bb{K}) & 
}
\end{xymatrix}\]
commutes. Indeed, $\mrm{char}_{(\bb{M},\sigma)}(\mrm{ev}_{1\otimes\Theta})$ is given by fiberwise evaluating on $1\otimes\Theta$. 
The composition $\mrm{ev}_{\Theta_\bb{M}}\circ B$ is likewise given fiberwise over $x\in\bb{M}$ by evaluating on $\Theta_\bb{M}$ restricted to $x$. 
The volume form $\Theta_\bb{M}$ was defined to restrict to $1\otimes\Theta$ over each point; see Definition \ref{def-VF}.

By definition (Corollary \ref{cor-AQHH}), the supertrace $t_\bb{M}$ is given by the formula 
\[t_\bb{M}=\int_\bb{M}\left(\mathbf{desc}_{(\bb{M},\sigma)}(t_{2n\vert a,b})\circ\beta_{\tilde{A}}\right).\]
Putting this all together, we have

\begin{align*}
\msf{Ev}_\bb{M}(t_\bb{M})&=\mrm{ev}_{\Theta_\bb{M}}(t_\bb{M} )\\
&= \mrm{ev}_{\Theta_\bb{M}}\left(\int_\bb{M}\left(\mathbf{desc}_{(\bb{M},\sigma)}(t_{2n\vert a,b})\circ\beta_{\tilde{A}}\right)\right)\\
&=\int_\bb{M}\mrm{ev}_{\Theta_\bb{M}}\left(\mathbf{desc}_{(\bb{M},\sigma)}(t_{2n\vert a,b})\circ\beta_{\tilde{A}}\right)\\
&=\int_\bb{M}\left(\mrm{ev}_{\Theta_\bb{M}}\circ B\circ\mrm{char}_{(\bb{M},\sigma)}\left((\AQ)^*\right)\right)(t_{2n\vert a,b}^\mrm{Lie} )\\
&=\int_\bb{M}\mrm{char}_{(\bb{M},\sigma)}(\mrm{ev}_{1\otimes\Theta})\left(\mrm{char}_{(\bb{M},\sigma)}\left((\AQ)^*\right)(t_{2n\vert a,b}^\mrm{Lie})\right)\\
&=\int_\bb{M}\mrm{char}_{(\bb{M},\sigma)}(\bb{K})(\mrm{ev}_{1\otimes\Theta}(t_{2n\vert a,b}^\mrm{Lie}))\\
&=
\int_\bb{M}\mrm{char}_{(\bb{M},\sigma)}(\bb{K})(\msf{Ev}_\mrm{loc}(t_{2n\vert a,b})).
\end{align*}

\end{proof}

Below, in \S\ref{sec-GSAIT}, 
we will see that the characteristic functor $\mrm{char}_{(\bb{M},\sigma)}$ relates to the classical Chern-Weil map \cite[Appendix C]{MilnorStasheff}. 

\part{Computations and Constructions}

In this part, we will construct a normalized supertrace on $\mathcal{A}_\sigma(\bb{M})$ 
and compute its evaluation on the volume form. 
The purely even portion of these computations can be found in \cite{FFSh}. 
We begin with preliminaries on Clifford algebras that we will need in our computations.

Given a quadratic vector space $(V,Q)$, 
we can form several different structures

\begin{itemize}

\item the special orthogonal group $\mrm{SO}(V,Q)$ and its Lie algebra $\mfrk{so}(V,Q)$, 

\item the Clifford algebra $\mrm{Cliff}(V,Q)$, and 

\item a symplectic structure $\omega_Q$ on the supermanifold $\bb{R}^{0\vert \dim(V)}$.

\end{itemize}

These objects are related. 
For example, the Lie algebra $\mfrk{so}(V,Q)$ embeds in the Clifford algebra. 
The following can be found, for example, in \cite[Pg. 61]{Knus}. 

\begin{lem}\label{lem-SOCl}
Consider the Lie subalgebra 
\[[V,V]_-=\{vw-wv:v,w\in\mrm{Cliff}(V,Q)\}\]
 in $\mrm{Cliff}(V,Q)$. 
There is an isomorphism of Lie algebras

\[\Phi\colon [V,V]_-\rta \mfrk{so}(V,Q)\]
given by sending $w\in[V,V]_-$ to the endomorphism $[-,w]_-\colon V\rta V$.

\end{lem}

\noindent Note that $\Phi$ allows us to view $\mfrk{so}(V,Q)$ as an \emph{even} subspace of $\mrm{Cliff}(V,Q)$. 

Just as the Weyl algebra is a deformation of a polynomial algebra, 
the Clifford algebra is a deformation of an exterior algebra. 
The local results in \cite{My2} (and \cite{Engeli}) show that the symplectic supermanifold $(\bb{R}^{0\vert \dim(V)},\omega_Q)$ has a canonical deformation by $\mrm{Cliff}(V,Q)$, as $\mrm{SO}(V,Q)$-modules.

\section{Quadratic Forms of Signature $(a,b)$}\label{sec-Quad}

Over $\bb{R}$, a quadratic form $Q$ on a vector space $V$ is determined by its signature $(a,b)$.  
In this section, we will analyze the various constructions (orthogonal groups, Clifford algebras, symplectic superspaces) for signature $(a,b)$. 

\begin{notation}
Unless otherwise noted, 
$Q$ will denote a quadratic form on $\bb{R}^{a+b}$ of signature $(a,b)$. 
A quadratic form $Q$ on a vector space $V$ has associated matrix $H_Q$ with $Q(v)=v^TH_Qv$, 
and bilinear form $B_Q$ with $Q(v)=B_Q(v,v)$. 

We will use the following shorthands:

\begin{align*}
\mrm{Cliff}(V,Q)&=\mrm{Ciff}_{a,b}\\
\mrm{SO}(V,Q)&=\mrm{SO}(a,b)\\
\mfrk{so}(V,Q)&=\mfrk{so}_{a,b}.
\end{align*}
\end{notation}

Note that there are equivalences of Lie groups 
\[\mrm{SO}(a,b)\simeq \mrm{SO}(b,a)\]
An explicit isomorphism can be found right above \cite[Def. 1.1.1]{UCSD}. 
Throughout the rest of the paper, we will assume $a\leq b$.

Let $\z =\lfloor\frac{b-a}{2}\rfloor$. 
Note that if $b-a$ is odd, then $2a+2\z+1=a+b$ is the dimension of $V$. 
We will fix a basis 
\[\{\zeta_1,\dots,\zeta_a,\eta_1,\dots,\eta_a,\xi_1,\dots,\xi_\z,\mu_1,\dots,\mu_\z,\upsilon\}\]
of $\bb{R}^{0\vert a+b}$ (where $\upsilon$ is only included if $b-a$ is odd) 
with

\begin{align*}
&B_Q(\zeta_i,\eta_i)=1/2  \text{ for all }  i=1,\dots,a\\
&B_Q(\xi_i,\xi_i)=-1/2  \text{ for all } i=1,\dots,\z\\
&B_Q(\mu_i,\mu_i)=-1/2  \text{ for all } i=1,\dots, \z \\
&B_Q(\upsilon,\upsilon)=-1/2.
\end{align*}

Note that in the Clifford algebra we then have 

\begin{align*}
\zeta_i^2=&\hbar Q(\zeta_i)=0 \text{ for all }  i=1,\dots,a\\
\eta_i^2=&\hbar Q(\eta_i)=0 \text{ for all }  i=1,\dots,a\\
[\eta_i.\zeta_i]=&\hbar Q(\zeta+\eta)=\hbar\text{ for all }  i=1,\dots,a\\
\xi_i^2=&\hbar Q(\xi_i)=-\frac{\hbar}{2}  \text{ for all } i=1,\dots,\z\\
\mu_i^2=&\hbar Q(\mu_i)=-\frac{\hbar }{2} \text{ for all } i=1,\dots, \z \\
\upsilon^2=&\hbar Q(\upsilon)=-\frac{\hbar }{2}.
\end{align*}

In our chosen basis, the matrix $H_Q$ associated to $Q$ is 

\[
H_Q=\frac{1}{2}(h^{ij}_Q)=\frac{1}{2}
\sbox0{$\begin{matrix} 0 & \mrm{Id}_a\\ \mrm{Id}_a & 0 \end{matrix}$}
\left[
\begin{array}{c|c}
\usebox{0}&\makebox[\wd0]{\large $0$}\\
\hline
  \vphantom{\usebox{0}}\makebox[\wd0]{\large $0$}&\makebox[\wd0]{\large $-\mrm{Id}_{b-a}$}
\end{array}
\right].
\]

As noted in Definition \ref{def-normalizedtrace}, 
the supertraces we will define will depend on an orientation. 

\begin{notation}\label{not-Orient}
We will give $\bb{R}^{0\vert a+b}$ the orientation 
\[\Theta=\zeta_1\eta_1\cdots\zeta_a\eta_a\xi_1\mu_1\cdots\xi_\z \mu_\z \upsilon.\]
We will sometimes also use the notation $\theta_1,\dots,\theta_{a+b}$ 
so that $\theta_1=\zeta_1$, $\theta_2=\eta_1$, and so on giving 
$\Theta=\theta_1\cdots\theta_{a+b}$.
Note that this is the same choice as was made in \S\ref{subsec-LocalInvariant}.
\end{notation}

\begin{proof}[Proof of Lemma \ref{lem-ThetaM}]
Take $\theta_1,\dots,\theta_{a+b}$ as in Notation \ref{not-Orient}.
\end{proof}

\subsubsection{Cartan Subalgebra}

We describe a Cartan subalgebra of $\mfrk{so}_{a,b}$. 
We will use this later in the proof of Theorem \ref{thm-LocalSAIT}.

\begin{lem}\label{lem-CartanSO}
A Cartan subalgebra of $\mfrk{so}_{a,b}$ is given by the subset 

\[\{\mfrk{h}=
\begin{bmatrix}
-\mrm{diag}(h_1,\dots, h_a) & 0 & 0 & 0\\
0 & \mrm{diag}(h_1,\dots, h_a) & 0 & 0\\
0 & 0 & 0 & \mrm{diag}(u_1,\dots, u_\z )\\
0 & 0 & -\mrm{diag}(u_1,\dots, u_\z ) & 0
\end{bmatrix}
:h_i,u_i\in\bb{K}\}\]
if $b-a$ is even. If $b-a$ is odd, there is an additional row and column of zeroes, corresponding to the $\upsilon$ basis element.
\end{lem}

\noindent See \cite[Pg. 402]{Sugiura}. 

Let $E_{i,j}$ denote the matrix with a 1 in the $ij$th entry and zeroes elsewhere. 
A basis for our chosen Cartan subalgebra is given by 

\[\mathcal{B}=\{-E_{i,i}+E_{a+i,a+i},E_{2a+j,2a+\z +j}-E_{2a+\z +j,2a+j} \}_{i=1,\dots,a}^{j=1,\dots,\z }\]

We would like to view these basis elements in the Clifford algebra under the isomorphism $\Phi$ from Lemma \ref{lem-SOCl}.
In other words, given a basis element $U\in\mathcal{B}$, 
we would like find an element $v\in\mrm{Cliff}_{a,b}$ so that the morphisms 
$\Phi(v)=[-,v]_-$ and $U\colon V\rta V$ agree. 

By direct computation, one obtains the following.

\begin{lem}\label{lem-CartanSOPHI}
For each $i=1,\dots, a$ we have 
\[\Phi(\eta_i\zeta_i)=-E_{i,i}+E_{a+i,a+i}\]
and for each $j=1,\dots, \z $ we have 
\[\Phi(-\xi_j\mu_j)=E_{2a+j,2a+\z +j}-E_{2a+\z +j,2a+j}.\]
\end{lem}

\subsubsection{Symplectic Form}

The symplectic form on $\bb{R}^{0\vert a+b}$ determined by $Q$ is given by 

\[\omega_Q=\sum_{i,j}\frac{1}{2}h_Q^{ij}\frac{\del}{\del\theta_i}\frac{\del}{\del\theta_j}\]
for entries $\frac{1}{2}(h^{ij}_Q)$ of the matrix $H_Q$; see \cite[Ex. 2.8]{My2}.

Using our naming convention (Notation \ref{not-Orient}), 
we can rewrite $\omega_Q$ as 

\[\omega_Q=\frac{1}{2}\sum_{i=1}^a\left(\frac{\del}{\del \zeta_i}\otimes\frac{\del}{\del\eta_i}+\frac{\del}{\del \zeta_i}\otimes\frac{\del}{\del\eta_i}\right)-\frac{1}{2}\sum_{j=1}^\z \left(\frac{\del}{\del \xi_j}\otimes\frac{\del}{\del \xi_j}+\frac{\del}{\del \mu_j}\otimes \frac{\del}{\del \mu_j}\right)-\frac{1}{2}\left(\frac{\del}{\del \upsilon}\right)^{\otimes 2}.\]

\section{Constructing the supertrace}\label{sec-ConstructingtheSupertrace}

The goal of this section is to construct a supertrace on $\AQ$. 
This construction will depend on a choice of orientation of $\bb{R}^{2n\vert a+b}$, 
see \S\ref{subsec-IntandOri}. 
By Corollary \ref{cor-TmTrace}, 
this supertrace will descend to define a supertrace on $(\bb{M},\omega)$. 
Recall that $\AQ$ is the tensor product of a Weyl and a Clifford algebra. 
By Lemma \ref{lem-tracehh}, 
we would like to construct an element of 
\[\mrm{HH}^n(\mrm{Weyl}_{2n}\otimes\mrm{Cliff}_{a,b})=\mrm{HH}^n(\mrm{Weyl}_{2n})\otimes\mrm{HH}^0(\mrm{Cliff}_{a,b})\]
In fact, we would like a cocycle representative of the cohomology class so we may apply Definition \ref{defn-EvalOnTraceLocal}. 

In the purely even case, 
an appropriate element of $\tau_{2n}\in\mrm{Hoch}^n(\mrm{Weyl}_{2n})$ was constructed in \cite{FFSh}. 
We review the definition of $\tau_{2n}$ below. 
The construction of the element is inspired by Kontsevich formality; see \cite[Rmk. Pg. 7]{FFSh}. 

In the purely odd case, 
the Clifford algebra $\mrm{Cliff}_{a,b}$ has a canonical underived supertrace. 

\begin{lem}\label{lem-trace1}
Let $(V,Q)$ be a quadratic real vector space. 
View $V$ as an even space. 
Given an orientation $\det(V)\simeq\bb{K}$, 
the quotient map 

\[\mrm{Cliff}(V,Q)\rta\mrm{Cliff}(V,Q)/\mrm{Cliff}_{(n-1)}(V,Q)\simeq\det (V)\simeq \bb{K}\]
defines a supertrace on $\mrm{Cliff}(V,Q)$.
\end{lem}

\noindent For a proof, see \cite[Prop. 2.10]{Toronto}.

\begin{rmk}\label{rmk-TauBer}
If we use $\Theta\in\mrm{det}(V)$ to identify $\mrm{det}(V)$ with $\bb{K}$, 
then the supertrace of Lemma \ref{lem-trace1} is the same map as the Berezin integral $\int(-)d\Theta$.
\end{rmk}

\begin{rmk}
When $a=b$, and we choose the orientation $\zeta_1\eta_1\cdots\zeta_a\eta_a$ (as in Notation \ref{not-Orient}) 
 this supertrace can also be described using the spinor representation  \cite{Deligne}.
We can identify the quadratic space $(V,Q)$ with $(W\oplus W^*,\mrm{ev})$. 
The spinor representation of the Clifford algebra $\mrm{Cliff}(V,Q)$ is then a map
\[\rho_\mrm{spin}\colon \mrm{Cliff}(V,Q)\rta\mrm{End}(\mrm{Sym}(W[1])).\]
This map (of algebras) is an isomorphism. 
As $W[1]$ is odd, we can identify $\mrm{End}(\mrm{Sym}(W[1]))$ with finite dimensional matrices. 
Taking the supertrace of matrices, 
we obtain a supertrace 
\[t_2\colon \mrm{Cliff}(V,Q)\rta\bb{K}.\]
Note that $t_2$ is the map inducing the Morita equivalence between the Clifford algebra and $\bb{K}$. 

We claim that $t_2$ agrees with the supertrace in Lemma \ref{lem-trace1}, up to a scalar.
As 
\[\mrm{Cliff}_{a,a}=(\mrm{Cliff}_{1,1})^{\otimes a},\]
it suffices to prove this when $a=1$. 
When $a=1$, the map
\[\rho_\mrm{spin}\colon\mrm{Cliff}_{1,1}\rta M_{2\times 2}(\bb{R})\]
given by
\begin{align*}
\rho_\mrm{spin}(\zeta_1)&=\begin{bmatrix}
0 & 1\\
1& 0
\end{bmatrix}\\
\rho_\mrm{spin}(\eta_1)&=\begin{bmatrix}
0 & 1\\
-1& 0
\end{bmatrix}\\
\rho_\mrm{spin}(\zeta_1\eta_1)&=\begin{bmatrix}
-1 & 0\\
0& 1
\end{bmatrix}
\end{align*}

Thus, $t_2(\zeta_1\eta_1)=t_2(\Theta)=-2$, which is $-2$ times the Berezin integral $\int(\Theta)d\Theta=1$. 

\end{rmk}

This supertrace determines an element $\tau_{0\vert a,b}\in\mrm{Hoch}^0(\mrm{Cliff}_{a,b})$. Together, $\tau_{2n}$ and $\tau_{0\vert a,b}$ determine a cohomology class 
\[[\tau_{2n}]\otimes[\tau_{0\vert a,b}]\in \mrm{HH}^n(\mrm{Weyl}_{2n}\otimes\mrm{Cliff}_{a,b})\]
In \S\ref{subsec-Lifting}, we produce a cocycle representative of this class. 
That is, a map 
\[(\AQ)^{\otimes 2n+1}\rta\bb{K}\]
that has a corresponding Hochschild cocycle $[\tau_{2n}]\otimes[\tau_{0\vert a,b}]$. 
As $\tau_{0\vert a,b}$ was defined from a map out of $\mrm{Cliff}_{a,b}$, 
we need a way of lifting this map to a map out of the $(2n+1)$-fold tensor product of the Clifford algebra. 

\subsection{Lifting to Degree $2n$}\label{subsec-Lifting}

We have a class in $\mrm{HH}^{2n}(\mrm{Weyl}_{2n}\otimes\mrm{Cliff}_{a,b})$. 
We would like a class in~$\mrm{Hoch}^{2n}(\mrm{Weyl}_{2n}\otimes\mrm{Cliff}_{a,b})$. 

\subsubsection{General Argument}
In general, 
one has an external product map in Hochschild cohomology
\[\vee\colon \mrm{HH}^i(\Lambda,B)\otimes\mrm{HH}^j(\Gamma,B')\rta \mrm{HH}^{i+j}(\Lambda\otimes\Gamma,B\otimes B')\]
which is constructed in \cite[Ch. XI \S 6]{HomAlg} and defined by what they call $g$ on \cite[Pg. 219 (3)]{HomAlg}. 

We are interested in the case 
\[\mrm{HH}^{2n}(\mrm{Weyl}_{2n};(\mrm{Weyl}_{2n})^*)\otimes\mrm{HH}^0(\mrm{Cliff}_{a,b};(\mrm{Cliff}_{a,b})^*)\rta\mrm{HH}^{2n}(\mrm{Weyl}_{2n}\otimes\mrm{Cliff}_{a,b};(\mrm{Weyl}_{2n}\otimes\mrm{Cliff}_{a,b})^*).\] 
In this case, the product is given by 
\[(f\vee g)(w_0\otimes c_0\vert\cdots\vert w_{2n}\otimes c_{2n})=f(w_0\vert\cdots\vert w_{2n})\otimes g(c_0\cdots c_{2n})\]
where $w_i\in\mrm{Weyl}_{2n}$, $c_{i}\in\mrm{Cliff}_{a,b}$, and the bars denote the tensor product in the bar complex defining Hochschild cohomology. 

Completing the Weyl algebra, we obtain a class $[\tau_{2n}]\vee[\tau_{0\vert a,b}]$ in $\mrm{HH}^{2n}(\AQ;(\AQ)^*)$ which involves evaluating $\tau_{0\vert a,b}$ on a product $c_0\cdots c_{2n}$. 
The product in $\AQ$, 
as defined in \cite[Def. 5.4]{My2} or \cite[\S 1.4]{Engeli}, 
is given by 
\[x\star y=m\left(\left(\exp\left(\frac{\hbar}{2}(\alpha+g)\right)\right)(x\otimes y)\right)\]
where $\alpha+g$ is the bivector 

\begin{align}\label{eq-Bivector}
\alpha+g=\sum_{i=1}^n\frac{\del}{\del p_i}\otimes\frac{\del}{\del q_i}-\frac{\del}{\del q_i}\otimes\frac{\del}{\del p_i}-\sum_{i=1}^{a+b} h_Q^{ij}\left(\frac{\del}{\del\theta_i}\otimes\frac{\del}{\del\theta_j}\right).
\end{align} 

Here, for $x=p_i,q_i,\theta_i$, 
the endomorphism $\frac{\del}{\del x}$ of $\AQ$ 
is given by identifying the underlying module of $\AQ$ with $\widehat{\mathcal{O}}_{2n\vert a,b}[[\hbar]]$ 
using the standard Poincar\'e-Birkhoff-Witt isomorphism in characteristic zero given by (super)-symmetrization, 
then taking partial derivatives of polynomials as usual. 

\begin{rmk}
Note that $\frac{\del}{\del\theta_i}$ is an odd map on $\AQ$ since $\theta_i$ is an odd coordinate. 
This will be important for our calculations in \S\ref{sec-CompEval}. 
\end{rmk}

Manipulating the exponential power series appearing in $c_0\star\cdots\star c_{2n}$ will result in a complicated formula that appears in $\omega_{2n\vert a,b}$ below, \S\ref{subsec-OurCase}.

\begin{rmk}
One could use this same argument to reduce from type $(2n\vert a,b)$ to type $(2\vert 1,1)$ and type $(2\vert 0,2)$ since $\mrm{Weyl}_{2n}=\mrm{Weyl}_2^{\otimes n}$ and similarly for $\mrm{Cliff}_{a,b}$. 
Applying this process to the Weyl algebra gives some explanation for the complicated formula for $\tau_{2n}$ given below and in \cite{FFSh}.
\end{rmk}

\subsubsection{Description in Our Case}\label{subsec-OurCase}

We now describe our desired cocycle, which we will denote 
$\tau_{2n\vert a,b}\in\mrm{Hoch}^{2n}(\widehat{\mrm{Weyl}}_{2n}\otimes\mrm{Cliff}_{a,b})$. 

The supertrace cocycle is a map 
\[\tau_{2n\vert a,b}\colon (\AQ)^{\otimes 2n+1}\rta\bb{K}.\]
This map should look like $\tau_{2n}$ on the Weyl algebra pieces, and applying $\tau_{0\vert a,b}$ to the product of the Clifford algebra pieces. 
Recall from Remark \ref{rmk-TauBer} 
that $\tau_{0\vert a,b}$ is a Berezin integral $\int(-)d\Theta$. 

We will construct $\tau_{2n\vert a,b}$ as the composite of three maps: 

\begin{enumerate}
\item the counit of the Hopf algebra $\widehat{\mrm{Weyl}}_{2n}\otimes\mrm{Cliff}_{a,b}=U(\mfrk{h}_{2n})\otimes_\bb{K} U( \mrm{cl}_{a,b})$ using the identification of \cite[\S 5.1]{My2}. 
This is a Berezin integral (which depends on a choice of orientation $\Theta$)
\[\Upsilon_{2n\vert a,b}\colon (\AQ)^{\otimes 2n+1} \rta \bb{K},\]
\item a complicated combination of the bidifferential operators coming from $\omega_Q$, mixed with configuration space integrals (which appear in Kontsevich formality. See \cite[Rmk. Pg. 7]{FFSh}.)

\[\int_{\Delta_{2n}}\omega_{2n\vert a,b}\colon(\AQ)^{\otimes 2n+1}\rta(\AQ)^{\otimes 2n+1},\]
\item and a map 
\[\pi_{2n\vert a,b}\colon (\AQ)^{\otimes 2n+1}\rta(\AQ)^{\otimes 2n+1}.\]
\end{enumerate}

Spelled out, we will consider the composite
\[(\AQ)^{\otimes 2n+1}\xrta{\pi_{2n\vert a,b}}(\AQ)^{\otimes 2n+1}\xrta{\int_{\Delta_{2n}}\omega_{2n\vert a,b}} (\AQ)^{\otimes 2n+1}\xrta{\Upsilon_{2n\vert a,b}}\bb{K}\]
which we denote $\tau_{2n\vert a,b}$. 

We will construct and study each of these three maps individually 
and then show that $\tau_{2n\vert a,b}$ satisfies the necessary properties. 
In particular, we will show that there is an equivalence of cohomology classes
\[[\tau_{2n\vert a,b}]=[\tau_{2n}]\otimes[\tau_{0\vert a,b}].\]
\noindent \textit{Berezin Integral.}
Define the map 
$\Upsilon_{2n\vert a,b}\colon (\AQ)^{\otimes 2n+1}\rightarrow\bb{K} $ 
by 
\[\Upsilon_{2n\vert a,b}(a_0\otimes\cdots\otimes a_{2n})=\int(a_0\cdots a_{2n})(y,u)\vert_{y=0}d\theta_1\cdots d\theta_{a+b}.\]
\noindent Here, the product $a_0\cdots a_{2n}$ is taken in the 
supercommutative product structure,
not the star product. 
The notation $\mu_{2n\vert a,b}$ is used for $\Upsilon_{2n\vert a,b}$ in \cite{Engeli}. We avoid this notation to prevent conflicts with the elements $\mu_i$ in $\AQ$.

The notation $a(y,u)\vert _{y=0}$ means the following. 
Assume $a\in\AQ$ is the tensor product $w\otimes c$ of $w\in\widehat{\mrm{Weyl}}_{2n}$ and $c\in\mrm{Cliff}_{a,b}$. 
Then $w$ can be viewed as a power series in $2n$ even variables $y$ and $c$ can be viewed as a polynomial in $a+b$ odd variables $u$. 
Then $a(y,u)\vert_{y=0}$ means $w(0)\otimes c$, 
where we evaluated the power series $w$ at $0$. 

Then $\Upsilon_{2n\vert a+b}$ is the counit for the Hopf superalgebra $\AQ=U(\mfrk{h}_{2n})\otimes_\bb{K} U(\mfrk{cl}_{a,b})$. \\
\noindent \textit{Bidifferential Operators and Configuration Space Integrals.}
For fixed $k$ and $1\leq i<j\leq k$, let $\alpha_{ij}$ be the endomorphism of $\left(\AQ\right)^{\otimes k+1}$ sending $a_0\otimes\cdots\otimes a_k$ to 
\[\frac{1}{2}\sum_{l=1}^n a_0\otimes\cdots \otimes \frac{\del}{\del p_l}a_i\otimes\cdots\otimes \frac{\del}{\del q_l}a_j\otimes\cdots\otimes a_k-
a_0\otimes\cdots \otimes \frac{\del}{\del q_l}a_i\otimes\cdots\otimes \frac{\del}{\del p_l}a_j\otimes\cdots\otimes a_k.\]
For $i<j$, we let $\alpha_{ji}=-\alpha_{ij}$.
Essentially, $\alpha_{ij}$ acts by applying the two-form $\alpha$ from (\ref{eq-Bivector}) to the $ij$th factor $a_i\otimes a_j$.  
Similarly, we define $g_{ij}\in\mrm{End}\left(\mathcal{A}_{2n\vert a,b}^{\otimes k+1}\right)$ by 

\[g_{ij}(a_0\otimes\cdots\otimes a_k)=\frac{1}{2} \sum_{m,l} h_Q^{ml} a_0\otimes\cdots \frac{\del}{\del \theta_m}a_i\otimes\cdots\otimes \frac{\del}{\del \theta_l}a_j\otimes\cdots\otimes  a_k.\]
For $i<j$, we let $g_{ji}=g_{ij}$. 

Let $\Delta_{2n}$ be the space 

\[\Delta_{2n}=\{(v_1,\dots, v_{2n})\in[0,1]^{2n}\colon j<k\text{ implies }v_j<v_k\}.\]

Let $\omega_{2n\vert a,b}$ be the endomorphism of $(\AQ)^{2n+1}$ given by 

\[\omega_{2n\vert a,b}=\mrm{exp}\left(\sum_{1\leq i<j\leq 2n}\hbar\psi(v_i-v_j)(\alpha_{ij}+g_{ij})\right)\]
where $\psi$ is the 1-periodic function so that $\psi(v)=2v+1$ for $-1\leq v<0$. 
That is, $\psi(v)=2B_1(v)$ for $B_1(v)$ the 1st Bernoulli polynomial. \\

\noindent \textit{Third Map.}
Let $\pi_{2n}\in\mrm{End}((\AQ)^{2n+1})$ be the map given by

\[\pi_{2n}=\frac{1}{n!}\left(\sum_{1\leq j<k\leq 2n}\alpha_{ij} dv_j\wedge dv_k\right)^n.\]

\begin{defn}
Let $\tau_{2n\vert a,b}$ denote the Hochschild cocycle corresponding to the map
\[\tau_{2n\vert a,b}=\Upsilon_{2n\vert a,b}\int_{\Delta_{2n}}\omega_{2n\vert a,b}\circ\pi_{2n\vert a,b}.\]
\end{defn}

\noindent The map $\tau_{2n\vert a,b}$ is an even map if $a+b$ is even and an odd map if $a+b$ is odd, 
see Corollary \ref{cor-AQHH}.

We now need to check that the cocycle $\tau_{2n\vert a,b}$ is a representative of our chosen cohomology class. 

\begin{lem}
The cocycle $\tau_{2n\vert a,b}$ has cohomology class $[\tau_{2n}]\otimes[\tau_{0\vert a,b}]$.
\end{lem} 

\begin{proof}
Setting $a=b=0$, we obtain the formula for $\tau_{2n}$ given in \cite[\S2.3 (2)]{FFSh}. 
Taking $n=0$, 
the terms $\omega_{0\vert a,b}$ and $\pi_0$ become the identity map. 
We are left with the Berezin integral $\int(-)d\Theta=\tau_{0\vert a,b}$. 
\end{proof}

\begin{prop}\label{prop-tauProperties}
The cocycle $\tau_{2n\vert a,b}$ corresponds to a derived relative supertrace on $\AQ$ as an object of $\mrm{Alg}(\mrm{Mod}_{\gKH})$.
\end{prop}

\begin{proof}
This is analogous to \cite[Thm. 2.11 (ii) and (iii)]{Engeli}. 
\end{proof}

\begin{thm}\label{thm-DescendedTrace}
The cocycle $\tau_{2n\vert a,b}$ determines a $2n$-derived supertrace on $\AQ$. 
Descending $\tau_{2n\vert a,b}$ to $\bb{M}$ gives a normalized supertrace $\mrm{Tr}_\bb{M}$ on $\mathcal{A}_\sigma(\bb{M})$.
\end{thm}

\begin{proof}
By Corollary \ref{cor-TmTrace}, 
$\tau_{2n\vert a,b}$ determines a supertrace on $\mathcal{A}_\sigma(\bb{M})$ by the map on cohomology 
\[\mathbf{desc}_{(\bb{M},\sigma)}(\tau_{2n\vert a,b})\circ\beta_{\tilde{A}}.\] 
The normalization condition is shown in \cite[Thm. 2.11(i)]{Engeli}. 
\end{proof}

\section{Computation of Evaluation on a Volume Form}\label{sec-CompEval}

The goal of this section is to compute 
$\mrm{Tr}_{\bb{M}}(\Theta_\bb{M})$ 
for the supertrace $\mrm{Tr}_\bb{M}$ as in Theorem \ref{thm-DescendedTrace}. 
Motivated by Theorem \ref{thm-GlobalizeInvariant}, 
we will first compute $\mrm{ev}_\Theta(\tau^{\mrm{Lie}}_{2n\vert a,b})$. 
We view this as the local computation, which is stated below as Theorem \ref{thm-LocalSAIT}. 
We end this section by proving the superalgebraic index theorem, 
generalizing Engeli's results \cite[Thm. 2.26]{Engeli}. 

\subsection{Local Superalgebraic Index Theorem: Set Up}\label{subsec-LSAIT}

By Proposition \ref{prop-tauProperties}, the class 
\[\tau_{2n\vert a,b}\in\mrm{Hoch}^{2n}(\mathcal{A}_{2n\vert a,b})\]
 defines a $\mfrk{sp}_{2n\vert a,b}$ equivariant class. 
Using Definition \ref{defn-EvalOnTraceLocal}, 
we get a class 
\[\msf{Ev}_\mrm{loc}(\tau_{2n\vert a,b})\in C^\bullet_\mrm{Lie}(\mfrk{g}^\hbar_{2n\vert a,b},\mfrk{sp}_{2n}\oplus \mfrk{so}_{a,b})\simeq C^\bullet_\mrm{Lie}(\mfrk{g}_{2n\vert a,b},\mfrk{sp}_{2n}\oplus \mfrk{so}_{a,b}\oplus \mfrk{gl}_1),\]
where we have identified the center $Z$ of $\AQ$ with $\mfrk{gl}_1$; 
see before Definition \ref{defn-EvalOnTraceLocal}.

As our eventual goal is to compute the global invariant $\msf{Ev}_\bb{M}(\mrm{Tr}_\bb{M})$, 
we only need to know the cohomology class of $\msf{Ev}_\mrm{loc}(\tau_{2n\vert a,b})=\mrm{ev}_\Theta(\tau^\mrm{Lie}_{2n\vert a,b})$. 
Indeed, by Theorem \ref{thm-GlobalizeInvariant}, 
the value of $\msf{Ev}_\bb{M}(\mrm{Tr}_\bb{M})$ is given by integrating over a term determined by $\msf{Ev}_\mrm{loc}(\tau_{2n\vert a,b})$. 
By Stoke's theorem, the integral only depends on the cohomology class.

We will show that the class 
$[\msf{Ev}_\mrm{loc}(\tau_{2n\vert a,b})]\in H^{2n}_\mrm{Lie}(\mfrk{g}^\hbar_{2n\vert a,b},\mfrk{sp}_{2n}\oplus \mfrk{so}_{a,b})$ 
comes from an invariant polynomial on $\mfrk{sp}_{2n}\oplus \mfrk{so}_{a,b}\oplus\mfrk{gl}_1$. 

The following can be found in \cite[\S 5.1]{FFSh} or \cite[\S 2.8.1 (2.3)]{Engeli}. 

\begin{defn}\label{def-Chi}
Let $\mfrk{h}\subset\mfrk{g}$ an inclusion of Lie superalgebras.
Let $\mrm{pr}\colon\mfrk{g}\rta\mfrk{h}$ be a projection map (on underlying super vector spaces) that is invariant under the adjoint action of $\mfrk{h}$. 
The \emph{curvature} of $\mrm{pr}$ is the map $C\in\mrm{Hom}(\Lambda^2\mfrk{g},\mfrk{h})$ given by
\[C(v\wedge w)=[\mrm{pr}(v),\mrm{pr}(w)]-\mrm{pr}([v,w]).\]
Let $\chi\colon\mrm{Sym}^m(\mfrk{h}^*)^\mfrk{h}\rta H^{2m}_\mrm{Lie}(\mfrk{g},\mfrk{h})$ 
denote the map sending an ad-invariant polynomial $P$ to the cocycle defined by 
\[\chi(P)(v_1\wedge\cdots\wedge v_{2m})=\frac{1}{m!}\sum_{\sigma\in\Sigma_{2m}/(\Sigma_2)^{\times m}}\mrm{sign}(\sigma)P\left(C(v_{\sigma(1)},v_{\sigma(2)}),\dots,C(v_{\sigma(2m-1)},v_{\sigma(2m)})\right).\]
\end{defn}

The curvature $C$ measures how far $\mrm{pr}$ is from being a Lie superalgebra map. 
Just as the usual Chern-Weil map is independent of the choice of connection, the map $\chi$ is independent of the choice of projection $\mrm{pr}$ \cite[\S 2.8.1]{Engeli}.

\begin{rmk}
We have encountered three Chern-Weil style maps: the map $\chi$, the functor $\mrm{char}_{(\bb{M},\sigma)}(\bb{K})$ from \S\ref{subsec-GlobalInvariant}, and the classical Chern-Weil map \cite[Appendix C]{MilnorStasheff}. 
In Lemma \ref{lem-CharCW} below, we will describe how these three maps are related.  
\end{rmk}

We will need the following observation.
Let $\overline{P}$ be a polynomial $\overline{P}\in\mrm{Sym}^\bullet(\mfrk{h})$ that is possibly \emph{not} ad invariant. 
We can create an ad invariant polynomial by an averaging process. 
Let $G$ be the group we would like $\overline{P}$ to be invariant under. 
Set
\[\overline{P}^\mrm{avg}=\frac{1}{|G|}\sum_{g\in G}\overline{P}\circ g.\]
Then $\overline{P}^\mrm{avg}$ is manifestly $G$ invariant. 

One may still apply the formula for $\chi$ to $\overline{P}$, 
but the resulting $\chi(\overline{P})$ may not live in \emph{relative} Lie algebra cohomology. 
In particular, $C^\bullet_\mrm{Lie}(\mfrk{g},\mfrk{h})$ is the subcomplex of $C^\bullet_\mrm{Lie}(\mfrk{g})$ of those cocycles that vanish on $\mfrk{h}$ and are $\mfrk{h}$ invariant. 
The element $\chi(\overline{P})$ will always vanish on $\mfrk{h}$ since $\mrm{pr}$ is a section of the inclusion $\mfrk{h}\rta\mfrk{g}$. 
However, $\chi(\overline{P})$ may not be $\mfrk{h}$-invariant. 

\begin{lem}
The diagram 
\[\begin{xymatrix}
{\mrm{Sym}^\bullet(\mfrk{h})\arw[r]^\chi\arw[d]_{(-)^\mrm{avg}} & \Lambda^{2\bullet}(\mfrk{g}/\mfrk{h})^*\arw[d]^{(-)^\mrm{avg}}\\
\mrm{Sym}^\bullet(\mfrk{h})^\mfrk{h}\arw[r]^\chi & C^\bullet_\mrm{Lie}(\mfrk{g},\mfrk{h})}
\end{xymatrix}\]
commutes.
\end{lem}

In particular, if we have a polynomial $\overline{P}\in\mrm{Sym}^\bullet(\mfrk{h})$ so that $\chi(\overline{P})$ is invariant, that is 
\[\chi(\overline{P})^\mrm{avg}=\chi(\overline{P}),\]
then $\chi(\overline{P}^\mrm{avg})=\chi(\overline{P})$. 
We will use this below.

\begin{proof}
It suffices to check that the curvature map $C\colon \Lambda^2(\mfrk{g})\rta\mfrk{h}$ is invariant under the adjoint action of $\mfrk{h}$. 
Note that $C$ is defined to be the difference of the composites
\[\Lambda^2\mfrk{g}\xrta{[-,-]}\mfrk{g}\xrta{\mrm{pr}}\mfrk{h}\]
and
\[\Lambda^2\mfrk{g}\xrta{\Lambda^2\mrm{pr}}\Lambda^2\mfrk{h}\xrta{[-.-]}\mfrk{h}.\]
The bracket maps $[-,-]$ are $\mfrk{g}$-invariant by the Jacobi identity. 
The projection map $\mrm{pr}$ in $\mfrk{h}$-invariant by assumption. 
Thus these two maps, and their difference $C$, are $\mfrk{h}$-invariant.
\end{proof}

\begin{ex}
In our case, we take the projection 
$\mrm{pr}\colon\AQ \rta\mfrk{sp}_{2n}\oplus \mfrk{so}_{a,b}\oplus\mfrk{gl}_1$ 
by projecting onto the homogeneous degree 2 piece and the constant part; see \cite[\S 5.7]{FFSh} or \cite[Pg. 29]{Engeli}. 
One should compare this with Lemma \ref{lem-SOCl}. 
We will continue to use the notation $\chi$ for the map 
\[\chi\colon \mrm{Sym}^\bullet(\mfrk{sp}_{2n}\oplus \mfrk{so}_{a,b}\oplus\mfrk{gl}_1)^{\mfrk{sp}_{2n}\oplus \mfrk{so}_{a,b}\oplus\mfrk{gl}_1}\rta C^{2\bullet}_\mrm{Lie}(\AQ,\mfrk{sp}_{2n}\oplus \mfrk{so}_{a,b}\oplus\mfrk{gl}_1)\simeq C^{2\bullet}_\mrm{Lie}(\mfrk{g}_{2n\vert a,b}^\hbar,\mfrk{sp}_{2n}\oplus \mfrk{so}_{a,b}).\]
\end{ex}

We will describe this cohomology class by giving an explicit description of a polynomial $P_n$ so that 
\[(-1)^n[\chi(P_n)]=[\msf{Ev}_\mrm{loc}(\tau_{2n\vert a,b})].\]

Let $\Phi$ denote the map $\mfrk{sp}_{2n}\oplus \mfrk{so}_{a,b}\rta\AQ$ extending the map from Lemma \ref{lem-SOCl} from the Clifford algebra to the Weyl algebra tensor the Clifford algebra. 
For $\Theta\in\AQ$, 
consider the function $P_n$ on $\mfrk{sp}_{2n}\oplus \mfrk{so}_{a,b}\oplus\mfrk{gl}_1$ of degree $n$ defined by the formula 

\begin{align}\label{eq-Pn}
P_n(a_1\otimes m_1,\dots, a_n\otimes m_n)=m_1\cdots m_n\Upsilon_{2n\vert a,b}\circ\int_{[0,1]^n}\omega_{2n\vert a,b}(\Theta \otimes \Phi(a_1)\otimes\cdots\otimes \Phi(a_n))dv_1\cdots dv_n.
\end{align}
for $a_1,\dots, a_n\in\mfrk{sp}_{2n}\oplus \mfrk{so}_{a,b}$, $m_1,\dots,m_n\in\mfrk{gl}_1$, and $v_1,\dots, v_n$ the coordinates for $[0,1]^n$. 
Note here that the bidifferential operators $\alpha_{ij}$ and $g_{ij}$ in $\omega_{2n\vert a,b}$ are acting on $(\AQ)^{\otimes n+1}$. 

By \cite[Lem. 2.24]{Engeli}, we have the following.

\begin{lem}\label{lem-ChiHitsZloc}
There is an equality in Lie cohomology
\[(-1)^n[\chi(P_n)]=[\msf{Ev}_\mrm{loc}(\tau_{2n\vert a,b})].\] 
\end{lem}
\noindent Note that we are not using any bijectivity of $\chi$, just the computation of $\chi(P_n)$.

We will explicitly compute $P_n$ below in the proof of Theorem \ref{thm-LocalSAIT} and see where the function fails to be ad invariant. 
Although the function $P_n$ is \emph{not} ad invariant, 
by Lemma \ref{lem-ChiHitsZloc}, $\chi(P_n)$ is invariant. 
Using our above observation, we can replace $P_n$ with $P_n^\mrm{avg}$ and still have $\chi(P_n^\mrm{avg})=\chi(P_n)$.

We would like a nice description of the ad-invariant function $P_n$. 
This will be done in terms of characteristic series of genera. 
For an overview of genera and their characteristic series; see \cite{Hirz}. 

\subsubsection{Characteristic Series}\label{subsec-MSeq}

We define the characteristic series of interest.

\begin{convention}
For the remainder of this section, 
we let $\mfrk{k}=\mfrk{sp}_{2n\vert a,b}$. 
\end{convention}

An ad-invariant function on the Lie algebra $\mfrk{k}$ is determined by its value on a Cartan subalgebra. 
It therefore suffices to show that a Cartan subalgebra of $\mfrk{k}=\mfrk{sp}_{2n}\oplus \mfrk{so}_{a,b}$ is as described. 
A Cartan subalgebra of $\mfrk{sp}_{2n}$ is given by the diagonal matrices. 
These matrices correspond to the elements $q_lp_l$ of $\AQ$ for $l=1,\dots,n$. 
By Lemmas \ref{lem-CartanSO} and \ref{lem-CartanSOPHI}, 
there is a basis for a Cartan subalgebra of $\mfrk{so}_{a,b}$ whose image under $\Phi$ is $\{\eta_i\zeta_i,-\xi_j\mu_j\}$ for $i=1,\dots,a$ and $j=1,\dots, t$

\begin{notation}
Let $t_i\in\mfrk{k}$ denote the element corresponding to $q_ip_i$. 
Let $s_i\in\mfrk{k}$ denote the element corresponding to $\eta_i\zeta_i$. 
Let $r_i\in\mfrk{k}$ denote the element corresponding to $-\xi_i\mu_i$. 
\end{notation}

\begin{ex}
Consider the ad-invariant function $\widehat{A}(-)$ on $\mfrk{k}$ 
determined by the polynomial 
 
\[\prod_{i=1}^n\frac{t_i/2}{\sinh(t_i/2)}.\]
This is the characteristic series for the $\widehat{A}$-genus.

\end{ex}

\begin{ex}\label{ex-Bhat}
If $a=b$, let $\widehat{B}(-)$ denote the ad-invariant function on $\mfrk{sp}_{2n\vert a,b}$ determined by the polynomial 
\[\prod_{i=1}^a \cosh(s_i/2).\]
\end{ex}

\begin{rmk}\label{rmk-Norm}

The characteristic series of the $\widehat{A}$ genus is $\frac{z}{\sinh(z)}$. 
As the $\widehat{A}$ class is multiplicative, $\widehat{A}(E\oplus E')=\widehat{A}(E)\widehat{A}(E')$, 
the power series $\frac{\sinh(s)}{s}$ in Example \ref{ex-Bhat} looks like the characteristic series of the $\widehat{A}$ genus of $-E$. 
For example, the power series $\frac{\sinh(s)}{s}$ on the tangent bundle determines the $\widehat{A}$ genus of the stable normal bundle.
\end{rmk}

\begin{rmk}\label{rmk-Lgenus}
The characteristic series for the L-genus is 
\[\frac{z}{\tanh(z)}=\frac{z}{\sinh(z)}\cosh(z).\]
We see this power series in the product of the power series in $\widehat{A}$ and $\widehat{B}$. 
\end{rmk}

\begin{ex}\label{ex-hatBC}
We get an ad-invariant function $\widehat{BC}(-)$ on $\mfrk{k}$ from the polynomial 
\[\prod_{i=1}^a \cosh(s_i/2)\prod_{i=1}^\z \cos(r_i/2).\]
\end{ex}

\subsubsection{Theorem Statement}

By Lemma \ref{lem-ChiHitsZloc}, 
the evaluation of $\tau_{2n\vert a,b}$ on the volume form $\Theta$ is given by the formula
\[[\msf{Ev}_\mrm{loc}(\tau_{2n\vert a,b})]=(-1)^n[\chi(P_n)].\]

To give an explicit description of $\msf{Ev}_\mrm{loc}(\tau_{2n\vert a,b})$, 
we need to compute the power series $P_n$ from (\ref{eq-Pn}). 
In the case of type $(2n\vert 0,0)$, 
one should compare the following with \cite[Pg. 28]{FFSh}.
In the case of type $(2n\vert a,a)$, 
one should compare the following with \cite[Lem. 2.25]{Engeli}.

\begin{thm}[Local Superalgebraic Index Theorem]\label{thm-LocalSAIT}
The ad invariant power series $P_n^\mrm{avg}$ obtained from (\ref{eq-Pn}) satisfies the equation

\[P^\mrm{avg}_n(x,\dots, x)=(-1)^{a+\z }\left[\det\left(\widehat{A}(\hbar x_1)\widehat{BC}(\hbar y) \right)^{1/2}\mrm{tr}(e^{x_2})\right]_n\]
where $x=x_1+y+x_2$ with $x_1\in\mfrk{sp}_{2n}$, $y \in\mfrk{so}_{a,b}$, and $x_2\in\mfrk{gl}_1$. 
The notation $[-]_n$ denotes the degree $n$ piece.
\end{thm}

\noindent This is proven in \S\ref{subsec-PfThm} below. 

In the notation of \S\ref{subsec-MSeq}, Theorem \ref{thm-LocalSAIT} says that $P^\mrm{avg}(x,\dots, x)$ is 

\[(-1)^{a+\z }\left[\det\left(\frac{\hbar x_1/2}{\sinh(\hbar x_1/2)}
 \cosh(\hbar y_1/2)\cos(\hbar y_2/2)\right)^{1/2}\mrm{tr}(e^{x_2})\right]_{n}\]

We will prove Theorem \ref{thm-LocalSAIT} by computing the polynomial $P_n$, and then deducing the computation of the averaged polynomial $P_n^\mrm{avg}$.
We therefore need to compute $P_n$ on a Cartan subalgebra of $\mfrk{sp}_{2n\vert a,b}$. 
Let $x$ be in this Cartan subalgebra and $X=\Phi(x)$. 
We will prove some preliminary lemmas that will be useful in the proof of Theorem \ref{thm-LocalSAIT}. 
Before we do this, we need the general set-up for the proof. 

\subsubsection{Proof Set-Up}\label{subsec-ProofSetUp}

Note that the generators $q_lp_l,\eta_i\zeta_i,-\xi_j\mu_j$ of the Cartan subalgebra are of degree at most two in the variables $q_l,p_l,\eta_i.\zeta_i.\xi_j.\mu_j$. 
Thus, only derivatives of order at most two from the $\alpha_{ij}$ and $g_{ij}$ appearing in the exponential $\omega_{2n\vert a,b}$ contribute, 
and cross terms $\alpha_{ij}g_{i'j'}$ vanish. 
The remaining piece of $\omega_{2n\vert a,b}$ that may not vanish is 

\[\omega_{2n\vert a,b}^{\leq 2\colon}=\prod_{0\leq i\leq j\leq n}\left(1+\hbar\psi(v_i-v_j)(\alpha_{ij}+g_{ij})+\frac{1}{2}\hbar^2\psi(v_i-v_j)^2(\alpha_{ij}^2+g_{ij}^2)\right).\]

We would like a workable description of the expanded product of $\omega_{2n\vert a,b}^{\leq 2}(\Theta\otimes X^{\otimes n})$. 
Following \cite[Pg. 26]{FFSh} and \cite[Lem. 2.25]{Engeli}, 
we associate each summand in the expanded product to a labeled graph on $n+1$ vertices. 
For notational consistency, we refer to these vertices as the 0th through $n$th. 
The zeroth vertex will be labeled by $\Theta$. 
The remaining $n$ vertices are labeled by the $n$ copies of $X$. 
A summand of the product expansion of $\omega_{2n\vert a,b}^{\leq 2}(\Theta\otimes X^{\otimes n})$ is obtained by, 
for each $ij$ with $0\leq i< j\leq 2n$, 
choosing either $1$, $\hbar\psi(v_i-v_j)(\alpha_{ij}+g_{ij})$ or the quadratic term $\frac{1}{2}\hbar^2\psi(v_i-v_j)^2(\alpha_{ij}^2+g_{ij}^2)$. 
One then adds edges to the $n+1$ vertex graph according to the following rules:

For $i>0$, if for the $ij$ term, 

\begin{itemize}

\item one chose the constant term $1$, add no edges.

\item one chose the linear term $\hbar\psi(v_i-v_j)(\alpha_{ij}+g_{ij})$, 
add an edge between the $i$ and $j$th vertices (which are labeled by $X$).

\item one chose the quadratic term 
$\frac{1}{2}\hbar^2\psi(v_i-v_j)^2(\alpha_{ij}^2+g_{ij}^2)$, 
add two edges between the $i$ and $j$th vertices (which are labeled by $X$).

\end{itemize}

If for the $0j$ term, 

\begin{itemize}

\item one chose the constant term $1$, add no edges.

\item one chose the linear term $\hbar\psi(v_0-v_j)(\alpha_{0j}+g_{0j})$, 
add an edge between the $0$ vertex (labeled by $\Theta$) and the $j$th vertex (labeled by $X$). 

\item one chose the quadratic term 
$\frac{1}{2}\hbar^2\psi(v_0-v_j)^2(\alpha_{0j}^2+g_{0j}^2)$, 
add two edges between the $0$ vertex (labeled by $\Theta$) and the $j$th vertex (labeled by $X$). 

\end{itemize}

\begin{ex}

The graph that is a disjoint union of two cycles, 
one between the 0th and 1st vertices,
 and one between the 2nd, 3rd, and 4th vertices 
 corresponds to the summand 
 
 \[\frac{1}{2}\hbar^2\psi(v_0-v_1)^2(\alpha_{01}+g_{01})^2\hbar\psi(v_2-v_3)(\alpha_{23}+g_{23})\hbar\psi(v_3-v_4)(\alpha_{34}+g_{34})\hbar\psi(v_2-v_4)(\alpha_{24}+g_{24}).\]

\begin{figure}[H]
\centering
\includegraphics[scale=.6]{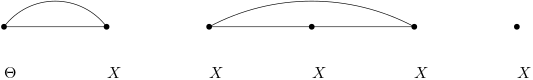}
\caption{Example Graph}
\end{figure}

\end{ex}

Note that a single graph consisting of a disjoint union of subgraphs corresponds to a product in~$\omega_{2n\vert a,b}^{\leq 2}$. 

\subsubsection{Graphs with Vanishing Terms}\label{subsec-GraphVanish}

We can rule out the following types of graphs.  

\begin{lem}
Graphs containing a loop on a single vertex as a connected component correspond to a vanishing summand of 
$\omega_{2n\vert a,b}^{\leq 2}(\Theta\otimes X^{\otimes n})$.
\end{lem}

\begin{proof}
Say the loop is on the $i$th vertex. 
Then the loop subgraph corresponds to the term 

\[\frac{1}{2}\hbar^2\psi(v_i-v_i)(\alpha_{ii}+g_{ii})^2.\]
As the disjoint union of subgraphs correspond to a product in $\omega_{2n\vert a,b}^{\leq 2}$, 
it suffices to show that 

\[\left(\frac{1}{2}\hbar^2\psi(v_i-v_i)(\alpha_{ii}+g_{ii})^2\right)(\Theta\otimes X^{\otimes n})=0.\]
Since partial derivatives commute, the bidifferential operator $\alpha_{ii}$ applies 

\[\sum_{l=1}^n\frac{\del}{\del p_i}\frac{\del}{\del q_i}-\frac{\del}{\del q_i}\frac{\del }{\del p_i}=0\]
to the $i$th term in the tensor product, 
and hence vanishes. 

Similarly, using the fact that $\frac{\del}{\del\theta_i}$ is an odd degree operator, we see that the bidifferential operator $g_{ii}$ applies 

\[\sum_{j=1}^a\frac{\del}{\del \zeta_i}\frac{\del}{\del \eta_i}+\frac{\del}{\del \eta_i}\frac{\del }{\del \zeta_i}=\sum_{j=1}^a\frac{\del}{\del \zeta_i}\frac{\del}{\del \eta_i}-\frac{\del}{\del \zeta_i}\frac{\del }{\del \eta_i}=0\]
to the $i$th term in the tensor product, 
and hence vanishes. 
\end{proof}

\begin{lem}
Graphs containing an $X$ labeled vertex of valence strictly more than two correspond to a vanishing summand of 
$\omega_{2n\vert a,b}^{\leq 2}(\Theta\otimes X^{\otimes n})$.
\end{lem}

\begin{proof}
The vertex of valence more than two corresponds to applying more than two partial derivatives to $X$. 
These vanish as each summand of $X$ has degree at most two in the basis elements. 
\end{proof}

Recall that $\psi(t)=2B_1(v)$. 
We will need the following identities for the Bernoulli polynomials $B_n(t)$ 
which can be found in \cite{Bernoulli}.

\begin{align}\label{eq-Convolution}
B_n*B_m(v)=\int_0^1 B_n(u)B_m(v-u)du=-\frac{n!m!}{(n+m)!}B_{n+m}(v).
\end{align}

In particular, for $v=0$ we have 

\[\int_0^1 B_n(u)B_m(-u)du=-\frac{n!m!}{(n+m)!}B_{n+m},\]
where $B_{n+m}$ is the $(n+m)$th Bernoulli \emph{number}. 

Moreover, we have 

\begin{align}\label{eq-Bnxn}
\int_x^{x+1}B_n(u)du=x^n.
\end{align}

\begin{lem}
Graphs containing a connected component that is a linear subgraph correspond to a vanishing summand of 
$\omega_{2n\vert a,b}^{\leq 2}(\Theta\otimes X^{\otimes n})$.
\end{lem}

\begin{proof}
Say the linear subgraph has length $j$. 
Either the $\Theta$ labeled vertex is in the linear subgraph or not. 

In the first case, when all vertices of the linear subgraph are labeled by $X$, 
after possibly reordering vertices, 
we may assume we are dealing with the graph corresponding to the summand 

\[\Upsilon_{2n\vert a,b}\int_{[0,1]^j}\psi(v_1-v_2)\cdots \psi(v_{j-1}-v_j)dv_1\cdots dv_j \left(\hbar(\alpha_{12}-g_{12})\cdots \hbar(\alpha_{j-1j}-g_{j-1j})(\Theta\otimes X\otimes\cdots\otimes X) \right). \]

Following \cite{Engeli}, we can use the convolution identity (\ref{eq-Convolution}) and see that the integral

\[\int_{[0,1]^j}\psi(v_1-v_2)\cdots \psi(v_{j-1}-v_j)dv_1\cdots dv_j  \]
is proportional to the zeroth Fourier coefficient of the 1-periodic Bernoulli polynomial $B_{j-1}(-)$, which vanishes. 

Now assume that the $\Theta$ labeled vertex is a part of the linear subgraph. 
Say there are $i$ $X$ labeled vertices to one side of $\Theta$ and $j-i$ to the other side. 
This subgraph then corresponds to the summand with integral 

\begin{align*}
&\int_{[0,1]^j}\psi(v_0-v_1)\cdots \psi(v_{i-1}-v_i)\psi(v_0-v_{i+1})\cdots\psi(v_{j-1}-v_j)dv_1\cdots dv_j\\
&=\left(\int_{[0,1]^{j-i}}\psi(-v_1)\cdots \psi(v_{i-1}-v_i)dv_{i+1}\cdots dv_j\right)\left(\int_{[0,1]^i}\psi(-v_{i+1})\cdots\psi(v_{j-1}-v_j)dv_1\cdots dv_i\right).
\end{align*}

Using the convolution identity (\ref{eq-Convolution}), 
this integral is proportional to 

\[\left(\int_0^1 B_{j-i}(v_j)dv_j\right)\left(\int_0^1 B_i(v_i)dv_i\right).\]
By the identity (\ref{eq-Bnxn}), both these integrals vanish. 

Thus, contributions from all linear subgraphs, containing $\Theta$ or not, vanish.
\end{proof}

The remaining types of graphs are disjoint unions of cycles on the $X$ labeled vertices or flowers whose center is the $\Theta$ labeled vertex and whose petals are cycles from the $\Theta$ vertex to $X$ labeled vertices. 
\begin{figure}[H]
\centering
\includegraphics[scale=.6]{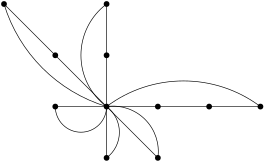}
\caption{$\Theta$-flower}
\end{figure}
\noindent Below, we will simply refer to this second type of graph as a $\Theta$-flower.

Our next step is to compute the non-vanishing contributions of cycles and $\Theta$-flowers on $\Theta\otimes X^{\otimes k}$. 
For this, it will be useful to break $X$ into a sum of three types of terms.

\begin{notation}
Say $X=X_1+Y_1+Y_2+X_2$ with
\begin{align*}
X_1&=\sum_{i=1}^n\gamma_iq_ip_i\\
Y_1&=\sum_{i=1}^a\lambda_i\eta_i\zeta_i\\
Y_2&=\sum_{i=1}^\z  -\kappa_i\xi_i\mu_i\\
X_2&=x_2
\end{align*}

\noindent for some scalars $\gamma_i,\lambda_i,\kappa_i,x_2\in\bb{K}$. 
Here $X_2\in\mfrk{gl}_1$. 
In sections \S\ref{subsec-X1}, \S\ref{subsec-Y1}, \S\ref{subsec-Y2} 
we compute the contributions of the $X_1$, $Y_1$, and $Y_2$ pieces respectively in terms corresponding to cycle and $\Theta$-flower graphs. 
\end{notation}

\subsubsection{Computations for $X_1$ Terms}\label{subsec-X1}

The following appears in \cite[Pg. 34]{Engeli}

\begin{lem}
In $\mathcal{A}_{2n\vert a,b}^{\otimes k+1}$, we have 
\[\alpha_{12}\cdots \alpha_{j-1j}\alpha_{j1}(\Theta\otimes X_1^{\otimes k})=
\frac{1}{2^{j-1}}\sum\limits_{i=1}^n \gamma_i^j(\Theta\otimes 1^{\otimes j}\otimes X_1^{\otimes k-j}) \]
if $j$ is even. 
This term vanishes if $j$ is odd.
\end{lem}

\begin{lem}\label{lem-muX1}
In $\bb{K}$ we have 
\[\Upsilon_{2n\vert a,b}(\Theta\otimes 1^{\otimes j}\otimes X_1^{\otimes k-j})=\delta_{jk}\]
\end{lem}

\begin{proof}
We have 
\[\Upsilon_{2n\vert a,b}(\Theta\otimes 1^{\otimes j}\otimes X_1^{\otimes k-j})=\int \left(\Theta X_1^{k-j}\right)\vert_{y=0}\] 
where $y$ represents the even variables. 
If $k-j$ is nonzero, then $X_1\vert_{y=0}=0$ and this term vanishes. 
When $k-j=0$, we have 
$\int\Theta=1$. 
\end{proof}

\begin{cor}\label{cor-X1}
If $k$ is even, the $X_1$ contribution of a cycle of length $k$ is

\[\Upsilon_{2n\vert a,b}(\alpha_{12}\cdots \alpha_{k-1k}\alpha_{k1}(\Theta\otimes X_1^{\otimes k}))=\frac{1}{2^{k-1}}\sum\limits_{i=1}^n \gamma_i^k.\]
This term vanishes if $k$ is odd.

\end{cor}

For example, if $k=6$ graph corresponding to such a cycle looks like the following.

\begin{figure}[H]
\centering
\includegraphics[scale=.6]{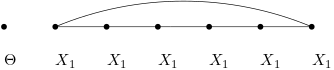}
\caption{$X_1$ Cycle}
\end{figure}

\subsubsection{Computations for $Y_1$ Terms}\label{subsec-Y1}

\begin{lem}\label{lem-Y1g12}
In $\mathcal{A}_{2n\vert a,b}^{\otimes k+1}$, for $j=2,\dots,k$, we have 

\[g_{12}\cdots g_{j-1j}g_{j1}(\Theta\otimes Y_1^{\otimes k})=
\frac{-1}{2^{j-1}}\sum\limits_{r=1}^a \lambda_r^j(\Theta\otimes 1^{\otimes j}\otimes Y_1^{\otimes k-j})
\] 
and 
\[g_{11}(\Theta\otimes Y_1^{\otimes k})=0.\]
\end{lem}

\begin{proof}
Since $Y_1$ does not involve any $\mu_i$ or $\xi_i$ terms, 
$\frac{\del}{\del\mu_i} Y_1$ and $\frac{\del}{\del\xi_i}Y_1$ vanish. 
Thus, we can replace the bidifferential operators $g_{ij}$ by the map sending $(a_0\otimes\cdots \otimes a_k)$ to

\[\frac{1}{2} \sum_{r=1}^a a_0\otimes\cdots \frac{\del}{\del \zeta_r}a_i\otimes\cdots\otimes \frac{\del}{\del \eta_r}a_j\otimes\cdots\otimes  a_k+a_0\otimes\cdots \frac{\del}{\del \eta_r}a_i\otimes\cdots\otimes \frac{\del}{\del \zeta_r}a_j\otimes\cdots\otimes  a_k\]
for this computation.

We have 

\begin{align*}
\frac{\del}{\del\zeta_r}Y_1=\frac{\del}{\del\zeta_r}\left(\sum_{i=1}^a \lambda_i\eta_i\zeta_i\right)=-\lambda_r\eta_r\\ 
\frac{\del}{\del\eta_r}Y_1=\frac{\del}{\del\eta_r}\left(\sum_{i=1}^a \lambda_i\eta_i\zeta_i\right)=\lambda_r\zeta_r.
\end{align*}

For $j=2$, using $g_{21}=g_{12}$, we have 
\begin{align*}
&g_{12}g_{21}(\Theta\otimes Y_1^{\otimes k})\\
&=g_{12}\left(\frac{1}{2}\sum_{r=1}^a
\Theta\otimes\frac{\del}{\del\zeta_r}Y_1\otimes \frac{\del}{\del\eta_r}Y_1\otimes Y_1^{\otimes k-2}+\Theta\otimes\frac{\del}{\del\eta_r}Y_1\otimes\frac{\del}{\del\zeta_r}Y_1\otimes Y_1^{\otimes k-2}\right)\\
&-g_{12}\left(\frac{1}{2}\sum_{r=1}^a
\Theta\otimes(-\lambda_r\eta_r)\otimes\lambda_r\zeta_r\otimes Y_1^{\otimes k-2}+\Theta\otimes\lambda_r\zeta_r\otimes(-\lambda_r\eta_r)\otimes Y_1^{\otimes k-2}\right)\\
&=-\frac{1}{2}\sum_{r=1}^a\lambda_r^2g_{12}\left(\Theta\otimes \eta_r\otimes \zeta_r\otimes Y_1^{\otimes k-2}+\Theta\otimes\zeta_r\otimes\eta_r\otimes Y_1^{\otimes k-2}\right)\\
&=-\frac{1}{4}\sum_{r=1}^a\lambda_r^2(\Theta\otimes\otimes 1\otimes 1\otimes Y_1^{\otimes k-2}+\Theta\otimes 1\otimes 1\otimes Y_1^{\otimes k-2})\\
&=-\frac{1}{2}\sum_{r=1}^a\lambda_r^2(\Theta\otimes 1^{\otimes 2}\otimes Y_1^{\otimes k-2}).
\end{align*}

For $j=3$ we have 

\begin{align*}
g_{12}g_{23}g_{31}(\Theta\otimes Y_1^{\otimes 3})&=\frac{1}{2^3}\sum_{r=1}^a\Theta\otimes(-\lambda_r)\otimes \lambda_r\otimes\lambda_r+\Theta\otimes\lambda_r\otimes(-\lambda_r)\otimes(-\lambda_r)=0.
\end{align*}

This parity continues in general, with the end term either canceling or doubling depending on a sign. 
For general $j$, only the first $j$ copies of $Y_1$ in 
$\Theta\otimes Y_1^{\otimes k}$ 
are acted upon by the bidifferential operators~$g_{12}\cdots g_{j-1j}g_{j1}$. 
Each such copy of $Y_1$ is acted on twice: 
the $i$th copy for $i=2,\dots ,j-1$ is acted on by $g_{i-1i}$ and $g_{ii+1}$, 
the first copy is acted on by $g_{12}$ and $g_{j1}$, 
and the $j$th copy is acted on by $g_{j-1j}$ and $g_{j1}$. 

We therefore have

\begin{align*}
g_{12}\cdots g_{j-1j}g_{j1}(\Theta\otimes Y_1^{\otimes k})&=\frac{1}{2^j}\sum_{r=1}^a -\Theta\otimes \lambda_r^{\otimes j}\otimes Y_1^{\otimes k-j}+(-1)^{j-1}\Theta\otimes\lambda_r^{\otimes j}\otimes Y_1^{\otimes k-j}\\
&=\begin{cases}
\frac{-1}{2^{j-1}}\sum_{i=1}^a\lambda_r^j(\Theta\otimes 1^{\otimes j}\otimes Y_1^{\otimes k-j}) & j\text{ even}\\
=0 & j\text{ odd}.
\end{cases}
\end{align*}

Lastly, we compute $g_{11}(\Theta\otimes Y_1^{\otimes k})$:

\begin{align*}
g_{11}(\Theta\otimes Y_1^{\otimes k})&=\frac{1}{2}\sum_{r=1}^a\left(\Theta\otimes \frac{\del}{\del\zeta_r}\frac{\del}{\del\eta_r}Y_1\otimes Y_1^{\otimes k-1}+\Theta\otimes \frac{\del}{\del\eta_r}\frac{\del}{\del\zeta_r}Y_1\otimes Y_1^{\otimes k-1}\right)\\
&=\sum_{r=1}^a\left(\Theta\otimes (-\lambda_r)\otimes Y_1^{\otimes k-1}\right)+\left(\Theta\otimes \lambda_r\otimes Y_1^{\otimes k-1}\right)\\
&=0.
\end{align*}

\end{proof}

Recall that 
$\Theta=\zeta_1\eta_1\cdots\zeta_a\eta_a\xi_1\mu_1\cdots\xi_t\mu_t \upsilon$ 
(where the $\upsilon$ only appears if $b-a$ is odd). 

\begin{lem}\label{lem-g0jY1}
In $\mathcal{A}_{2n\vert a,b}^{\otimes k+1}$, for $j=1,\dots,k$ odd, we have 

\[g_{01}\cdots g_{j-1j}g_{j0}(\Theta\otimes Y_1^{\otimes k})=
\frac{1}{2^{j}}\sum\limits_{r=1}^a \lambda_r^j\left(\frac{\Theta}{\zeta_r\eta_r}\otimes 1^{\otimes j}\otimes Y_1^{\otimes k-j}\right).
\]
This term vanishes if $j$ is even. 
Moreover $g_{00}(\Theta\otimes Y_1^{\otimes k})=0$. 

\end{lem}

\begin{proof}
The proof is the similar to that of Lemma \ref{lem-Y1g12}, 
after noting that $\frac{\del}{\del\eta_r}\Theta=-\frac{\Theta}{\eta_r}$ and 
$\frac{\del}{\del\zeta_r}\Theta=\frac{\Theta}{\zeta_r}$. 

For example, we have

\begin{align*}
g_{01}g_{10}(\Theta\otimes Y_1)&=\frac{1}{2}g_{01}\left(\sum_{r=1}^a\frac{\Theta}{\zeta_r}\otimes\lambda_r\zeta_r+\left(-\frac{\Theta}{\eta_r}\right)\otimes(-\lambda_r\eta_r)\right)\\
&=\frac{1}{2^2}\sum_{r=1}^a\frac{\Theta}{\eta_r\zeta_r}\otimes\lambda_r+\frac{\Theta}{\zeta_r\eta_r}\otimes\lambda_r\\
&=\frac{1}{2}\sum_{r=1}^a\frac{\Theta}{\zeta_r\eta_r}\otimes\lambda_r.
\end{align*}

In general, we have 

\begin{align*}
g_{01}\cdots g_{j-1j}g_{j0}(\Theta\otimes Y_1^{\otimes k})&=\frac{1}{2^{j+1}}\sum_{r=1}^a\frac{\del}{\del\eta_r}\frac{\del}{\del\zeta_r}\Theta\otimes\lambda_r^{\otimes j}\otimes Y_1^{\otimes k-j}+(-1)^j\frac{\del}{\del\zeta_r}\frac{\del}{\del\eta_r}\Theta\otimes \lambda_r^{\otimes j}\otimes Y_1^{\otimes k-j}.
\end{align*}

As $\frac{\del}{\del\zeta_r}\frac{\del}{\del\zeta_r}=-\frac{\del}{\del\zeta_r}\frac{\del}{\del\eta_r}$, 
the right hand side vanishes if $j$ is even and becomes

\[\frac{1}{2^j}\sum_{r=1}^a\lambda_r^j\left(\frac{\Theta}{\eta_r\zeta_r}\otimes 1^{\otimes j}\otimes Y_1^{\otimes k-j}\right)\] 
if $j$ is odd.

For $g_{00}$, the bidifferential operators are only acting on $\Theta$. 
As $\Theta$ contains all basis elements $\mu_i,\eta_i,\xi_i,\zeta_i,\upsilon_i$, 
we would apriori need to consider the full form of $g_{00}$. 
However, the $\mu_i$ and $\xi_i$ terms vanish here: 

\begin{align*}
g_{00}(\Theta\otimes Y_1^{\otimes k})&=\frac{1}{2}\sum_{ml}h_Q^{ml}\frac{\del}{\del\theta_m}\frac{\del}{\del\theta_l}\Theta\otimes Y_1^{\otimes k}\\
&=\frac{1}{2}\sum_{r=1}^a\left(\frac{\del}{\del\zeta_r}\frac{\del}{\del\eta_r}\Theta\otimes Y_1^{\otimes k}+\frac{\del}{\del\eta_r}\frac{\del}{\del\zeta_r}\Theta\otimes Y_1^{\otimes k}\right)\\
&-\frac{1}{2}\sum_{r=1}^t\left(\frac{\del}{\del\xi_r}\frac{\del}{\del\xi_r}\Theta\otimes Y_1^{\otimes k}+\frac{\del}{\del\mu_r}\frac{\del}{\del\mu_r}\Theta\otimes Y_1^{\otimes k}\right)-\frac{1}{2}\frac{\del}{\del\upsilon}\frac{\del}{\del\upsilon}\Theta\otimes Y_1^{\otimes k}.
\end{align*}

Now $\frac{1}{2}\frac{\del}{\del\upsilon}\frac{\del}{\del\upsilon}\Theta$, $\frac{1}{2}\frac{\del}{\del\xi}\frac{\del}{\del\xi}\Theta$, and $\frac{1}{2}\frac{\del}{\del\mu}\frac{\del}{\del\mu}\Theta$ all vanish, 
and we are left with 

\[g_{00}(\Theta\otimes Y_1^{\otimes k})=\frac{1}{2}\sum_{r=1}^a-\frac{\Theta}{\zeta_r\eta_r}\otimes Y_1^{\otimes k}+\frac{\Theta}{\eta_r\zeta_r}\otimes Y_1^{\otimes k}=0\]
as desired.

\end{proof}

Note that the computation of $g_{00}(\Theta\otimes Y_1)$ holds for $Y_1$ replaced with any element of the Cartan subalgebra, in particular for $X_1$ or $Y_2$. 

The term computed in Lemma \ref{lem-g0jY1} corresponds to a single petal of length $j$ on a $\Theta$-flower. 
Say we have a $\Theta$-flower with $l$ petals of length $j_1,j_2,\dots,j_l$ on a total of $j$ $X$-labeled vertices. 
This graph corresponds to the following computation. 

\begin{cor}
Consider a partition $j_1+\cdots j_l=j$. 
Then 
\[(g_{01}\cdots g_{j_1-1j_1}g_{j_10})(g_{0j_1+1}\cdots g_{j_1+j_2-1j_1+j_2}g_{j_1+j_20})\cdots(g_{0j-j_l}\cdots g_{j-1k}g_{j0})(\Theta\otimes Y_1^{\otimes k})\]
vanishes unless $j_1,\dots, j_l$ are odd in which case it 
is equal to 
\[\frac{1}{2^{j_1}\cdots 2^{j_l}}\sum_{|R|=l}\lambda_{r_1}^{j_1}\cdots\lambda_{r_l}^{j_l}\left(\frac{\Theta}{\zeta_{r_1}\eta_{r_1}\cdots\zeta_{r_l}\eta_{r_l}}\otimes 1^{\otimes j}\otimes Y_1^{k-j}\right)\]
where $R=(r_1,\dots,r_l)$ ranges over ordered subsets of $\{1,\dots, a\}$ of size $l$. 
\end{cor}

In particular, we may have $R=(i,j)$ and $R=(j,i)$. 

\begin{lem}
In $\bb{K}$ for $j=0,\dots, k$, we have 
\[\Upsilon_{2n\vert a,b}(\Theta\otimes 1^{\otimes j}\otimes Y_1^{\otimes k-j})=\delta_{jk}\]
vanishes unless $j=k$.
\end{lem}

\begin{proof}
Since $\zeta_i^2=0=\eta_i^2$ we have

\begin{align}\label{eq-commute}
\zeta_i\eta_i\eta_i\zeta_i=0
\end{align}

Now since $\Theta=\zeta_1\eta_1\cdots\zeta_a\eta_a\xi_1\mu_1\cdots\xi_t\mu_t\upsilon$, 
for any $i$ we have 
\[\Theta\eta_i\zeta_i=0,\]
as we can commute the $\eta_i\zeta_i$ past the $\eta_j\zeta_j$ for $i\neq j$, 
and then use the above observation (\ref{eq-commute}).

We have

\begin{align*}
\Upsilon_{2n\vert a,b}(\Theta\otimes 1^{\otimes j}\otimes Y_1^{\otimes k-j})&=\int\Theta\left(\sum_{r=1}^a\lambda_r\eta_r\zeta_r\right)^{k-j}\\
&=\sum_{|I|=k-j}\lambda_I\int\Theta(\eta_{i_1}\zeta_{i_1}\cdots\eta_{i_{k-j}}\zeta_{i_{k-j}})\\
&=0.
\end{align*}

\end{proof}

\begin{lem}\label{lem-Difference}
Consider a subset $R=(r_1,\dots,r_l)$of $\{1,\dots, a\}$. 
In $\bb{K}$, if $j\neq k$, we have 
\[\Upsilon_{2n\vert a,b}\left(\frac{\Theta}{\zeta_{r_1}\eta_{r_1}\cdots\zeta_{r_l}\eta_{r_l}}\otimes 1^{\otimes j}\otimes Y_1^{k-j}\right)=l!(-1)^l\lambda_{r_1}\cdots\lambda_{r_l}.\]
This term vanishes if $j=k$. 

In particular, for $l=1$ we have 

\[\Upsilon_{2n\vert a,b}\left(\frac{\Theta}{\zeta_r\eta_r}\otimes 1^{\otimes j}\otimes Y_1^{\otimes k-j}\right)=
\begin{cases}
-\lambda_r & k-j=1\\
0 & \text{otherwise}.
\end{cases}\]

\end{lem}

\begin{proof}
Let $\Theta_R=\frac{\Theta}{\zeta_{r_1}\eta_{r_1}\cdots\zeta_{r_l}\eta_{r_l}}$.
We have 

\[\Upsilon_{2n\vert a,b}(\Theta_R\otimes 1^{\otimes j}\otimes Y_1^{\otimes k-j})=\int\Theta_R Y_1^{k-j}=\sum_{|I|=k-j}\lambda_I\int\Theta_R(\eta_{i_1}\zeta_{i_1}\cdots\eta_{i_{k-j}}\zeta_{i_{k-j}}).\] 
If there is a term $r'\in I$ with $r'\notin R$, 
then we see a $\zeta_{r'}\eta_{r'}\eta_{r'}\zeta_{r'}=0$ in the product 

\[\Theta_R(\eta_{i_1}\zeta_{i_1}\cdots\eta_{i_{k-j}}\zeta_{i_{k-j}}).\] 
These summands therefore vanish, and so we may assume $I$ consists only of terms in $R$.  
We have
\[\frac{\Theta}{\zeta_r\eta_r}\eta_r\zeta_r=\frac{\Theta}{\zeta_r\eta_r}(-\zeta_r\eta_r)=-\Theta\]
and

\[-\Theta(\eta_r\zeta_r)=0.\] 
Thus inductively we see that, 
\[\frac{\Theta}{\zeta_r\eta_r}(\eta_r\zeta_r)^l=
\begin{cases}
-\Theta & l=1\\
0 & \text{otherwise}
\end{cases}.\]

If $j=k$, or $R\not\subset I$, 
then there is some $r\in R$ that is not in $I$. 
We would then be taking the Berezin integral of $\frac{\Theta}{\zeta_r\eta_r}$, which is zero.

Thus, for $|R|=1$, we must have $I=(r,r,\dots,r)$. We get

\[\Upsilon_{2n\vert a,b}\left(\frac{\Theta}{\zeta_r\eta_r}\otimes 1^{\otimes j}\otimes Y_1^{\otimes k-j}\right)=
\begin{cases}
-\lambda_r\int\Theta=-\lambda_r & k-j=1\\
0 & \text{otherwise}
\end{cases}.\]
More generally, given a partition $S=(s_1,\dots, s_l)$ with 
$s_1+\cdots +s_l=k-j$, 
we have 
\[\Upsilon_{2n\vert a,b}\Theta_R(\eta_{r_1}\zeta_{r_1})^{s_1}\cdots(\eta_{r_l}\zeta_{r_l})^{s_l}=
\begin{cases}
(-1)^l\Upsilon_{2n\vert a,b} \Theta=(-1)^l & (s_1,\dots,s_l)=(1,\dots,1)\\
0 & \text{otherwise}
\end{cases}.\]

We therefore obtain 

\[\Upsilon_{2n\vert a,b}\left(\frac{\Theta}{\zeta_{r_1}\eta_{r_1}\cdots\zeta_{r_l}\eta_{r_l}}\otimes 1^{\otimes j}\otimes Y_1^{k-j}\right)=l!(-1)^l\lambda_{r_1}\cdots\lambda_{r_l}.\]
\end{proof}

In summary, we have shown the following:

\begin{cor}\label{cor-Y1}
For $j=2,\dots,k$, the $Y_1$ contribution from a cycle of length $j$ is

\[\Upsilon_{2n\vert a,b}\left(g_{12}\cdots g_{j-1}g_{j1}(\Theta\otimes Y_1^{\otimes j}) \right)=\frac{-1}{2^{j-1}}\sum_{r=1}^a\lambda_r^j\]
if $j$ is even and vanishes if $j$ is odd.

Moreover, for $j=1,\dots,k-1$, a $\Theta$-flower with $l$ petals of length $j_1,j_2,\dots,j_l$ on a total of $j$ $X$-labeled vertices 
has $Y_1$ contribution 

\[\Upsilon_{2n\vert a,b}\left((g_{01}\cdots g_{j_1-1j_1}g_{j_10})(g_{0j_1+1}\cdots g_{j_1+j_2-1j_1+j_2}g_{j_1+j_20})\cdots(g_{0j-j_l}\cdots g_{j-1k}g_{j0})(\Theta\otimes Y_1^{\otimes k})\right)\]
which is equal to

\[\frac{(-1)^l}{2^{j_1}\cdots 2^{j_l}}\sum_{|R|=l}l!\lambda_{r_1}^{j_1+1}\cdots\lambda_{r_l}^{j_l+1}.\]  
Here $R=(r_1,\dots, r_l)$ is a subset of $\{1,\dots, a\}$.
\noindent This term vanishes if $j=k$.

\end{cor}

For example, if $l=5$, $j_1=3$, $j_2=1$, $j_3=2$, $j_4=1$, $j_5=2$ and $k=14$, 
the $\Theta$-flower considered here looks like

\begin{figure}[H]
\centering
\includegraphics[scale=.6]{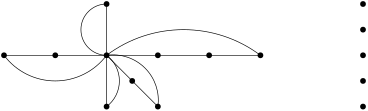}
\caption{$Y_1$ $\Theta$-flower}
\end{figure}
\noindent where all the vertices are labeled $Y_1$, except the center of the flower which is labeled $\Theta$. 
The partition $S$ of the $k-j=5$ spare vertices corresponds graphically to assigning spare vertices to the petals. 

\subsubsection{Computations for $Y_2$ Terms}\label{subsec-Y2}

\begin{lem}\label{lem-g12Y2}
In $\mathcal{A}_{2n\vert a,b}^{\otimes k+1}$, for $j=2,\dots, k$ even we have 

\[g_{12}\cdots g_{j-1j}g_{j1}(\Theta\otimes Y_2^{\otimes k})=
\frac{-(-1)^{j/2}}{2^{j-1}}\sum\limits_{r=1}^{\z } \kappa_r^j(\Theta\otimes 1^{\otimes j}\otimes Y_2^{\otimes k-j}).
\]
This term vanishes if $j$ is odd. Moreover, 
$g_{11}(\Theta\otimes Y_2^{\otimes k})=0$.
\end{lem}

\begin{proof}
Since $Y_2$ does not involve any $\eta_i$, $\zeta_i$, or $\upsilon$ terms, 
$\frac{\del}{\del\eta_i} Y_2$, $\frac{\del}{\del\zeta_i}Y_2$, and $\frac{\del}{\del\upsilon}Y_2$ vanish. 
Thus, we can replace the bidifferential operators $g_{ij}$ by the map sending $(a_0\otimes\cdots \otimes a_k)$ to

\[\frac{1}{2} \sum_{r=1}^{\z } a_0\otimes\cdots \frac{\del}{\del \xi_r}a_i\otimes\cdots\otimes \frac{\del}{\del \xi_r}a_j\otimes\cdots\otimes  a_k+a_0\otimes\cdots \frac{\del}{\del \mu_r}a_i\otimes\cdots\otimes \frac{\del}{\del \mu_r}a_j\otimes\cdots\otimes  a_k\]
for this computation.

In the expression for $g_{11}(\Theta\otimes Y_2^{\otimes k})$, 
the bidifferential operators $\frac{\del}{\del\xi_r}\frac{\del}{\del\xi_r}$ and $\frac{\del}{\del\mu_r}\frac{\del}{\del\mu_r}$ act on the first copy of $Y_2$. 
Since $Y_2=\sum\limits_{i=1}^\z -\kappa_i\xi_i\mu_i$ only has one copy of each $\xi_i$ and $\mu_i$, 
these operators are zero. 
This proves the second claim that $g_{11}(\Theta\otimes Y_2^{\otimes k})$ vanishes.

We have 
$\frac{\del}{\del\xi_r}Y_2=-\kappa_r\mu_r$ and 
$\frac{\del}{\del\mu_r}Y_2=\kappa_r\xi_r$.

\noindent For $j=2$, using $g_{21}=g_{12}$, we have 
\begin{align*}
&g_{12}g_{21}(\Theta\otimes Y_2^{\otimes k})\\
&=g_{12}\left(\frac{1}{2}\sum_{r=1}^{\z }
\Theta\otimes\frac{\del}{\del\xi_r}Y_2\otimes \frac{\del}{\del\xi_r}Y_2\otimes Y_2^{\otimes k-2}+\Theta\otimes\frac{\del}{\del\mu_r}Y_2\otimes\frac{\del}{\del\mu_r}Y_2\otimes Y_2^{\otimes k-2}\right)\\
&=g_{12}\left(\frac{1}{2}\sum_{r=1}^{\z }
\Theta\otimes(-\kappa_r\mu_r)\otimes(-\kappa_r\mu_r)\otimes Y_2^{\otimes k-2}+\Theta\otimes(\kappa_r\xi_r)\otimes(\kappa_r\xi_r)\otimes Y_2^{\otimes k-2}\right)\\
&=\frac{1}{2}\sum_{r=1}^{\z }\kappa_r^2g_{12}\left(\Theta\otimes \mu_r\otimes \mu_r\otimes Y_2^{\otimes k-2}+\Theta\otimes\xi_r\otimes\xi_r\otimes Y_2^{\otimes k-2}\right)\\
&=\frac{1}{4}\sum_{r=1}^{\z }\kappa_r^2(\Theta\otimes\otimes 1\otimes 1\otimes Y_2^{\otimes k-2}+\Theta\otimes 1\otimes 1\otimes Y_2^{\otimes k-2})\\
&=\frac{1}{2}\sum_{r=1}^{\z }\kappa_r^2(\Theta\otimes 1^{\otimes 2}\otimes Y_1^{\otimes k-2}).
\end{align*}

In general, if $j$ is even we have 
\begin{align*}
g_{12}\cdots g_{j-1j}g_{j1}(\Theta\otimes Y_2^{\otimes k})
&=\frac{-(-1)^{j/2}}{2^j}\sum_{r=1}^\z \Theta\otimes\kappa_r^{\otimes j}\otimes Y_2^{\otimes k-j}+\Theta\otimes\kappa_r^{\otimes j}\otimes Y_2^{\otimes k-j}.
\end{align*}
This term vanishes if $j$ is odd.
\end{proof}

\begin{lem}
In $\mathcal{A}_{2n\vert a,b}^{\otimes k+1}$, we have, for $j=1,\dots, k$ odd
\[g_{01}\cdots g_{j-1j}g_{j0}(\Theta\otimes Y_2^{\otimes k})=
\frac{(-1)^{(j+1)/2}}{2^j}\sum\limits_{r=1}^{\z } \kappa_r^j\left(\frac{\Theta}{\xi_r\mu_r}\otimes 1^{\otimes j}\otimes Y_2^{\otimes k-j}\right). \]
This term vanishes if $j$ is even. Moreover
$g_{00}(\Theta\otimes Y_2^{\otimes k})=0$.
\end{lem}

The claim for $g_{00}$ follows from the proof of Lemma \ref{lem-g0jY1}.

\begin{proof}
Note that $\frac{\del}{\del\xi_r}\Theta=\frac{\Theta}{\xi_r}$
and
$\frac{\del}{\del\mu_r}\Theta=-\frac{\Theta}{\mu_r}$. 
For $j=1$ we have 
\begin{align*}
g_{01}g_{10}(\Theta\otimes Y_2^{\otimes k})&=\frac{1}{2}g_{01}\left(\sum_{r=1}^\z \frac{\del}{\del\xi_r}\Theta\otimes\frac{\del}{\del\xi_r}Y_2\otimes Y_2^{\otimes k-1}+\frac{\del}{\del\mu_r}\Theta\otimes\frac{\del}{\del\mu_r}Y_2\otimes Y_2^{\otimes k-1}\right)\\
&=\frac{1}{2}g_{01}\left(\sum_{r=1}^\z \frac{\Theta}{\xi_r}\otimes(-\kappa_r\mu_r)\otimes Y_2^{\otimes k-1}+\left(-\frac{\Theta}{\mu_r}\right)\otimes(\kappa_r\xi_r)\otimes Y_2^{\otimes k-1}\right)\\
&=-\frac{1}{4}\sum_{r=1}^\z \frac{\Theta}{\xi_r\mu_r}\otimes\kappa_r\otimes Y_2^{\otimes k-1}+\frac{\Theta}{\xi_r\mu_r}\otimes \kappa_r\otimes Y_2^{\otimes k-1}\\
&=-\frac{1}{2}\sum_{r=1}^\z \kappa_r(\frac{\Theta}{\xi_r\mu_r}\otimes 1\otimes Y_2^{\otimes k-1}).
\end{align*}

When $j=2$, we have 

\begin{align*}
g_{01}g_{12}g_{20}(\Theta\otimes Y_2^{\otimes 2})&=\frac{1}{2^3}\sum_{r=1}^\z \frac{\del}{\del\xi_r}\frac{\del}{\del\xi_r}\Theta\otimes \kappa_r\otimes(-\kappa_r)+\frac{\del}{\del\mu_r}\frac{\del}{\del\mu_r}\Theta\otimes(-\kappa_r)\otimes\kappa_r=0.
\end{align*}

This parity pattern continues. 
If $j$ is even, then we have 
\[g_{01}\cdots g_{j0}(\Theta\otimes Y_2^{\otimes j})=\frac{1}{2^{j+1}}\sum_{r=1}^\z (-1)^{j/2}\frac{\del}{\del\xi_r}\frac{\del}{\del\xi_r}\Theta\otimes \kappa_r^{\otimes j}+(-1)^{j/2}\frac{\del}{\del\mu_r}\frac{\del}{\del\mu_r}\Theta\otimes\kappa_r^{\otimes j}=0.\]
If $j$ is odd, then we have 

\begin{align*}
g_{01}\cdots g_{j0}(\Theta\otimes Y_2^{\otimes j})&=\frac{1}{2^{j+1}}\sum_{r=1}^\z (-1)^{(j+1)/2}\frac{\del}{\del\mu_r}\frac{\del}{\del\xi_r}\Theta\otimes \kappa_r^{\otimes j}+(-1)^{(j-1)/2}\frac{\del}{\del\xi_r}\frac{\del}{\del\mu_r}\Theta\otimes\kappa_r^{\otimes j}\\
&=\frac{1}{2^{j+1}}\sum_{r=1}^\z (-1)^{(j+1)/2}\frac{\Theta}{\xi_r\mu_r}\otimes \kappa_r^{\otimes j}-(-1)^{(j-1)/2}\frac{\Theta}{\xi_r\mu_r}\otimes\kappa_r^{\otimes j}\\
&=\frac{(-1)^{(j+1)/2}}{2^j}\sum_{r=1}^\z \frac{\Theta}{\xi_r\mu_r}\otimes\kappa_r^{\otimes j}.
\end{align*}
\end{proof}

\begin{cor}
Consider a partition $j_1+\cdots j_l=j$. 
If $j_1,\dots, j_l$ are odd, then 
\[(g_{01}\cdots g_{j_1-1j_1})(g_{j_10}g_{0j_1+1}\cdots g_{j_1+j_2-1j_1+j_2}g_{j_1+j_20})\cdots(g_{0j-j_l}\cdots g_{j-1k}g_{j0})(\Theta\otimes Y_2^{\otimes k})\]
is equal to 
\[\frac{(-1)^{\frac{j_1+1}{2}}\cdots (-1)^{\frac{j_l+1}{2}}}{2^{j_1}\cdots 2^{j_l}}\sum_{|R|=l}\kappa_{r_1}^{j_1}\cdots\kappa_{r_l}^{j_l}\left(\frac{\Theta}{\xi_{r_1}\mu_{r_1}\cdots\xi_{r_l}\mu_{r_l}}\otimes 1^{\otimes j}\otimes Y_2^{k-j}\right)\]
where $R=(r_1,\dots,r_l)$ ranges over subsets of $\{1,\dots, t\}$ of size $l$. 
This term vanishes if any $j_i$ is even.
\end{cor}

\begin{lem}\label{lem-mujiY2}
In $\bb{K}$ we have 
\[\Upsilon_{2n\vert a,b}(\Theta\otimes 1^{\otimes j}\otimes Y_2^{\otimes k-j})=
\delta_{k,j}.
\]
\end{lem}

\begin{proof}
We have 
\[Y_2^{k-j}=\left(\sum_{r=1}^\z -\kappa_r\xi_r\mu_r\right)^{k-j}=\sum_{|I|=k-j}(-1)^{k-j}\kappa_I\xi_{i_1}\mu_{i_1}\cdots\xi_{i_{k-j}}\mu_{i_{k-j}}\] 
where $I$ ranges over all ordered subsets $I=(i_1,\dots, i_{k-j})$ of $\{1,\dots,t\}$ 
and $\kappa_I=\kappa_{i_1}\cdots\kappa_{i_l}$. 

Since $\xi_r^2=0=\mu_r^2$ and $\xi_r\mu_r=-\mu_r\xi_r$, 
we have 
\[\xi_r\mu_r\xi_r\mu_r=0.\]
Hence, 
\[\int\Theta\xi_r\mu_r=0.\]
Thus, 
\[\int\Theta \xi_{i_1}\mu_{i_1}\cdots\xi_{i_{k-j}}\mu_{i_{k-j}}=0.\]
\end{proof}

\begin{lem}
Consider a subset $R=(r_1,\dots,r_l)$of $\{1,\dots, t\}$. 
In $\bb{K}$, if $k-j-1=2s'$ is even, then
\[\Upsilon_{2n\vert a,b}\left(\frac{\Theta}{\xi_{r_1}\mu_{r_1}\cdots\xi_{r_l}\mu_{r_l}}\otimes 1^{\otimes j}\otimes Y_2^{k-j}\right)\]
is equal to 
\[(-1)^ll!\kappa_{r_1}\cdots\kappa_{r_l}\]
if $k-j=l$ and vanishes otherwise.
\end{lem}

\begin{proof}

Let $I=(i_1,\dots, i_{k-j})$ with $i_l\in\{1,\dots, \z\}$. 
If $r$ appears in $I$, say $r=i_{k-j}$, then 
\[\frac{\Theta}{\xi_r\mu_r} \xi_{i_1}\mu_{i_1}\cdots\xi_{i_{k-j}}\mu_{i_{k-j}}=\Theta \xi_{i_1}\mu_{i_1}\cdots\xi_{i_{k-j-1}}\mu_{i_{k-j-1}}\]
and we are in the situation of Lemma \ref{lem-mujiY2}. 

If $r$ does not appear in $I$, then $\frac{\Theta}{\xi_r\mu_r} \xi_{i_1}\mu_{i_1}\cdots\xi_{i_{k-j}}\mu_{i_{k-j}}$ contains no $\xi_r\mu_r$ term, 
and its Berezin integral therefore vanishes.

We therefore must have $R\subset I$ for the term to not vanish.

Thus
\[\Upsilon_{2n\vert a,b}\left(\frac{\Theta}{\xi_r\mu_r}\otimes 1^{\otimes j}\otimes Y_2^{\otimes k-j}\right)=\begin{cases}
-\kappa_r & k-j=1\\
0 & \text{otherwise}
\end{cases}\]
and similarly for $\frac{\Theta}{\xi_{r_1}\mu_{r_1}\cdots\xi_{r_l}\mu_{r_l}}$.
\end{proof}

\begin{lem}
In $\bb{K}$, for $j=0,\dots,k$, we have 
\[\Upsilon_{2n\vert a,b}\left(\frac{\Theta}{\zeta_r\eta_r}\otimes 1^{\otimes j}\otimes Y_2^{\otimes k-j}\right)=0.\]
\end{lem}

\begin{proof}
We have 
\[\Upsilon_{2n\vert a,b}\left(\frac{\Theta}{\zeta_r\eta_r}\otimes 1^{\otimes j}\otimes Y_2^{\otimes k-j}\right)=\int\frac{\Theta}{\zeta_r\eta_r} Y_2^{k-j}.\]
Since $Y_2$ contains no $\eta_r$ or $\zeta_r$ terms, 
we will be taking the Berezin integral of something with no $\eta_r$ or $\zeta_r$, 
which vanishes.
\end{proof}

In summary, we have shown the following:

\begin{cor}\label{cor-Y2}
For $j=2,\dots,k$, a cycle of length $j$ has $Y_2$ contribution 
\[\Upsilon_{2n\vert a,b}\left(g_{12}\cdots g_{j-1j}g_{j1}(\Theta\otimes Y_2^{\otimes k})\right)=\left(\frac{-(-1)^{j/2}}{2^{j-1}}\sum_{r=1}^\z \kappa_r^j\right)\left(\sum_{|J|=s}(-1)^s\left(\frac{\hbar}{2}\right)^{2s}\kappa_J^2\right)\]
if $k-j=2s$ is even, and vanishes if $k$ or $j$ is odd.

Moreover, 
\[\Upsilon_{2n\vert a,b}(g_{11}(\Theta\otimes Y_2^{\otimes k}))=0.\]
Consider a partition $j_1+\cdots j_l=j$. 
Then a $\Theta$-flower with $l$ petals of length $j_1,\dots,j_l$ has $Y_2$ contribution
\[\Upsilon_{2n\vert a,b}\left((g_{01}\cdots g_{j_1-1j_1})(g_{j_10}g_{0j_1+1}\cdots g_{j_1+j_2-1j_1+j_2}g_{j_1+j_20})\cdots(g_{0j-j_l}\cdots g_{j-1k}g_{j0})(\Theta\otimes Y_2^{\otimes k})\right).\]
If $k-j=l$, this is equal to 
\begin{align}\label{eq-Kappa}
\sum_{|R|=l}\frac{(-1)^{(j+1)/2}(-1)^l}{2^{j}}l!\kappa_{r_1}^{j_1+1}\cdots\kappa_{r_l}^{j_l+1},
\end{align}
where $R=(r_1,\dots,r_l)$ ranges over subsets of $\{1,\dots, \z\}$ of size $l$. 
This vanishes if $k-j\neq l$.

Moreover, 
\[\Upsilon_{2n\vert a,b}(g_{00}(\Theta\otimes Y_2^{\otimes k}))=0.\]
\end{cor}

\subsection{Proof of Theorem \ref{thm-LocalSAIT}}\label{subsec-PfThm}

We now combine the preliminary lemmas summarized in Corollaries \ref{cor-X1}, \ref{cor-Y1}, and \ref{cor-Y2} to prove Theorem \ref{thm-LocalSAIT}. 

Recall that we are trying to compute 
\begin{align}\label{eq-XP}
P_n(x,\dots,x)=\Upsilon_{2n\vert a,b}\int_{[0,1]^n}\omega_{2n\vert a,b}^{\leq 2}(\Theta\otimes X^{\otimes n})dv_1\cdots dv_n.
\end{align} 

\begin{proof}[Proof of Theorem \ref{thm-LocalSAIT}]

We need to classify all possible graph types and piece together the corresponding contributions from $X$. 

By \S\ref{subsec-GraphVanish}, the only terms that do not vanish after applying 
$\omega_{2n\vert a,b}^{\leq 2}$ 
are disjoint unions of cycles on $X$ labeled vertices and possibly a $\Theta$-flower.
We must also consider vertices with no edges. 
Each of the $n$ vertices labeled $X$ belongs to one of these three types: 
a cycle, a flower, or a vertex with no edges. 
Correspondingly, to each graph we have a partition of $n$ into three numbers, 
$n=n_1+n_2+n_3$ where
\begin{align*}
n_1& \text{ is the number of $X$-labeled vertices in cycles,}\\
n_2& \text{ is the number of $X$-labeled vertices in the $\Theta$-flower, and}\\
n_3& \text{ is the number of solo $X$-labeled vertices.}
\end{align*}
\noindent Note that if $n_2=0$ then the $\Theta$-labeled vertex has no edges.

The cycles part of the graph is determined by a partition of the $n_1$ vertices 
\[n_1=\sum_{j=2}^{n_1}jl_j\]
where $l_j$ denotes the number of cycles of length $j$. 
Note that this sum starts at $j=2$ since a cycle of length $1$ is a vertex with no edges. 
Let $\mathcal{P}(n_1)$ denote the set of such partitions of $n_1$.

The $\Theta$-flower part of the graph is determined by a partition of the $n_2$ $X$-labeled vertices 
\[n_2=\sum_{i=1}^{n_2}i\tilde{l}_i\]
where $\tilde{l}_i$ denotes the number of petals with $i$ $X$-labeled vertices. 
This sum starts at $i=1$ since a petal with one $X$-labeled vertex is allowed.
Let $\tilde{\mathcal{P}}(n_2)$ denote the set of such partitions of $n_2$.

The data of a partition of $n$ into $n_1+n_2+n_3$ along with the further partitions of $n_1$ and $n_2$ determine the graph. 
We will therefore be taking the sum over all such choices. 

By \S\ref{subsec-ProofSetUp}, 
each graph corresponds to a summand in the expanded product of $\omega^{\leq 2}_{2n\vert a,b}$. 
We would like to compute (\ref{eq-XP}). 
Since $\Upsilon_{2n\vert a,b}$ and integration are linear, 
we can pull out the sum over all graphs. 
Given a graph $G$, let $C_G$ denote the contribution from the corresponding summand of (\ref{eq-XP}). 
Let $G_\mrm{aut}$ be the automorphism group of $G$.
We then have 
\[P_n=\sum_{n=n_1+n_2+n_3}\sum_{\mathcal{P}(n_1)}\sum_{\mathcal{P}(n_2)} \frac{n!}{|G_\mrm{aut}|}C_G.\] 
To compute $|G_\mrm{aut}|$ and $C_G$, let us fix some notation. 
Say $G$ corresponds to the graph $n=n_1+n_2+n_3$ and the partitions 
\[n_1=\sum_{j=2}^{n_1}jl_j\]
and 
\[n_2=\sum_{i=1}^{n_2}i\tilde{l}_i.\]
Then we have 
\[|G_\mrm{aut}|=(n_3)!2^{\sum\limits^{n_1}_{i= 3} l_i}2^{\sum\limits^{n_2}_{i= 3} \tilde{l}_i}\prod_{j= 2}^{n_1}j^{l_j}\prod_{j=2}^{n_2}j^{{\tilde{l}_j}}.\] 
The $(n_3)!$ comes from permuting the vertices with no edges. 
We get a factor of $j$  from permuting the vertices within a cycle of length $j$. 
As there are $l_j$ of these cycles, we get a factor of $j^{l_j}$. 
Similarly, the term $j^{\tilde{l}_j}$ comes from permuting the vertices within petals of length $j$. 
Lastly, the powers of $2$ come from the reflection (or mirror) symmetry of cycles and petals of length $\geq 3$. 
One can compare this computation with \cite[Pg. 27]{FFSh} and \cite[Pg. 35]{Engeli}. 

The term $C_G$ is the product of the contributions of the connected subgraphs of $G$, 
\[C_G=\left(\prod_{j=2}^{n_1}\left(\text{length $j$ cycle contribution}\right)^{l_j}\right)\left(\text{$\Theta$-flower contribution}\right).\]
We first analyze the length $j$ cycle contribution. 
Recall that $X=X_1+Y_1+Y_2$ 
and let $\mathcal{X}^j_1, \mathcal{Y}^j_1,$ and $\mathcal{Y}^j_2$ denote the contribution to a length $j$ cycle from the $X_1, Y_1$, and $Y_2$ pieces, respectively. 
Expanding the product of the sum of these terms, we get the following
\[\prod_{j=2}^{n_1}(\text{length $j$ cycle contribution})^{l_j}=\sum_{l_j=l_j(X_1)+l_j(Y_1)+l_j(Y_2)}\prod_{j=2}^{n_1}C_j\left(\mathcal{X}_1^j\right)^{l_j(X_1)}\left(\mathcal{Y}_1^j\right)^{l_j(Y_1)}\left(\mathcal{Y}_2^j\right)^{l_j(Y_2)}\]
where
\[C_j=\frac{1}{\left(l_j(X_1)\right)!\left(l_j(Y_1)\right)!\left(l_j(Y_2)\right)!}\]
accounts for the additional graph automorphisms. 

In each of the terms $\mathcal{X}^j_1, \mathcal{Y}^j_1,$ and $\mathcal{Y}^j_2$ we see the same integral 
\[I_j=\int_{[0,1]^j}\psi(v_1-v_2)\cdots\psi(v_{j-1}-v_j)\psi(v_j-v_1)dv_1\cdots dv_j.\] 
By \cite[Lem. 5.4]{FFSh} or \cite[Pg. 34]{Engeli}, we have 

\begin{align}\label{eq-IB}
I_j=\begin{cases}
-\frac{(-2)^j}{j!}B_j & j\text{ even}\\
0 & j\text{ odd}
\end{cases}
\end{align}

where $B_j$ is the $j$th Bernoulli number. 
Thus, for $j\geq 2$ odd, the term $I_j$ vanishes. 
Assume $j$ is even. 

By Corollary \ref{cor-X1}, we have 
\[\mathcal{X}_1^j=\hbar^jI_j \Upsilon_{2n\vert a,b}(\alpha_{12}\cdots \alpha_{j-1j}\alpha_{j1})(\Theta\otimes X_1^{\otimes j})=\frac{\hbar^j}{2^{j-1}}I_j\sum\limits_{i=1}^n \gamma_i^j.\]
By Corollary \ref{cor-Y1}, we have 
\[\mathcal{Y}_1^j=\hbar^jI_j \Upsilon_{2n\vert a,b}\left(g_{12}\cdots g_{j-1}g_{j1}(\Theta\otimes Y_1^{\otimes j}) \right)=\frac{-\hbar^j}{2^{j-1}}I_j\sum_{r=1}^a\lambda_r^j.\]
By Corollary \ref{cor-Y2}, we have 
\[\mathcal{Y}_2^j=\hbar^jI_j \Upsilon_{2n\vert a,b}\left(g_{12}\cdots g_{j-1}g_{j1}(\Theta\otimes Y_2^{\otimes j+n_3}) \right)=\left(\frac{-(-1)^{j/2}\hbar^j}{2^{j-1}}I_j\sum_{s=1}^\z \kappa_s^j\right)\left(\sum_{\sum 2u_i}\prod_{i=1}^\z\frac{(-1)^{u_i}}{(2u_i)!}\left(\frac{\hbar\kappa_i}{2}\right)^{2u_i}\right)\]
where $\sum_{i=1}^\z 2u_i=n_3j$. 
The last term in this product is coming from the $Y_2$ contribution on the $n_3$ vertices with no edges. 
Note that since $I_j$ vanishes for $j$ odd, we may assume $j$ is even 
and the $(-2)^{j-1}$ contribution in $\mathcal{Y}_1^j$ and $\mathcal{Y}_2^j$ may be replaced with a $-2^{j-1}$. 

As in \cite{FFSh} and \cite{Engeli}, the $x_2$ contribution comes graphs with no vertices. 
On a graph of $j$ vertices, with no edges, each labeled by $x_2$, we get a contribution of $(x_2)^j$. 
Ignoring the $\Theta$-flowers for a moment, 
we can write the cycles piece $P_n$ as 
\begin{align*}
P_n^\mrm{cycle}=\sum_{\mathcal{P}(n)}\frac{n!}{(n_3)!}\prod_{j\geq 2}C_j&\left(\frac{I_j}{2^jj}\sum_{i=1}^n(\hbar\gamma_i)^j\right)^{l_j(X_1)}\left(\frac{-I_j}{2^jj}\sum_{r=1}^a(\hbar\lambda_r)^j\right)^{l_j(Y_1)}\left(-(-1)^{j/2}\frac{I_j}{2^jj}\sum_{s=1}^\z (\hbar\kappa_s)^j\right)^{l_j(Y_2)}\\&\cdot\left(\sum_{\sum 2u_i}\prod_{i=1}^\z\frac{(-1)^{u_i}}{(2u_i)!}\left(\frac{\hbar\kappa_i}{2}\right)^{2u_i}\right)(x_2)^{n_3}
\end{align*}
where $\mathcal{P}(n)$ ranges over partitions of $n$ as
\[n=\sum_{j\geq 2}j(l_j(X_1)+l_j(Y_1)+l_j(Y_2))+n_3\]
and we brought in the copies of $\frac{1}{j}$ and powers of 2 from $|G_\mrm{aut}|$.

Setting $P^\mrm{cycle}_0=1$, let 
\[P^\mrm{cycle}=\sum_{n=0}^\infty \frac{1}{n!} P^\mrm{cycle}_n.\]
Note that the terms coming from the $n_3$ vertices with no edges contribute
\begin{align*}
&\sum_{n_3=0}^\infty\frac{x_2^{n_3}}{(n_3)!}\\
&=\exp(x_2)
\end{align*}

Then, we have 
\[P^\mrm{cycle}=\exp\left(x_2+\sum_{j\geq 2}\frac{I_j}{2^j j}\sum_{i=1}^n(\hbar\gamma_i)^j-\frac{I_j}{2^j j}\sum_{r=1}^a(\hbar\lambda_r)^j-\frac{(-1)^{j/2}I_j}{2^j j}\sum_{s=1}^\z (\hbar\kappa_s)^j\right).\]
We will use the identity 
\[\sum_{j\geq 2}\frac{I_j}{2^jj}x^j=\mrm{log}\left(\frac{x/2}{\sinh(x/2)}\right)\]
found in \cite[Pg. 27]{FFSh} and \cite[Pg. 36]{Engeli} to identify the $\gamma_i$ and $\lambda_j$ terms. 
For the $\kappa_s$ term, we will use the similarly derived identity
\begin{align}\label{eq-NewId}
\sum_{j\geq 2}-\frac{(-1)^{j/2}I_j}{2^jj}x^j=\log\left(\frac{\sin(x/2)}{x/2}\right).
\end{align}
To see this identity, note that
\[\sum_{j\geq 2}-\frac{(-1)^{j/2}I_j}{2^j j}\sum_{s=1}^\z x^j=\sum_{j\geq 2}(-1)^{j/2}\frac{B_j}{j}\frac{x^j}{j!}.\]
We would like to identify this as $\log$ of some function $F(x)$.  
We have
\begin{align*}
\frac{d}{dx}\left(\sum_{j\geq 2}(-1)^{j/2}\frac{B_j}{j}\frac{x^j}{j!}\right)&=\sum_{j\geq 2}(-1)^{j/2}B_j\frac{x^{j-1}}{j!}\\
&=\left(\sum_{j=0}(-1)^{j/2}B_j\frac{x^{j-1}}{j!}\right)-\frac{1}{x}\\
&=\frac{1}{2}\cot(x/2)-\frac{1}{x}\\
&=\frac{\frac{d}{dx}F(x)}{F},
\end{align*}
when 
\[F(x)=\frac{\sin(x/2)}{x/2}.\]
This gives the identity (\ref{eq-NewId}). 
Thus, 
\[\sum_{j=0}-\frac{(-1)^{j/2}I_j}{2^j j}\sum_{s=1}^\z (\hbar\kappa_s)^j=\log\left(\frac{\sin(\hbar\kappa_s/2)}{\hbar\kappa_s/2}\right)\]
We then have
\[P^\mrm{cycle}=e^{x_2}\prod_{i=1}^n \left(\frac{\hbar\gamma_i/2}{\sinh(\hbar\gamma_i/2)}\right)\prod_{r=1}^a \left(\frac{\sinh(\hbar\lambda_r/2)}{\hbar\lambda_r/2}\right)\prod_{s=1}^\z \left(\frac{\sin(\hbar\kappa_s/2)}{\hbar\kappa_s/2}\right).\]
For the $\Theta$-flower contribution, only $Y_1$ and $Y_2$ parts contribute. 
Let $\widetilde{\mathcal{Y}}_1$ and $\widetilde{\mathcal{Y}}_2$ denote their contributions so that 
\[(\text{$\Theta$-flower contribution})=\left(\widetilde{\mathcal{Y}}_1\right)\left(\widetilde{\mathcal{Y}}_2\right).\]
To describe the $\Theta$-flower contribution term, we need to reorganize the data of our partition of $n_2$. 
Let 
\[l=\sum_{j=1}^{n_2}l_j\]
denote the number of petals of the $\Theta$-flower. 
Order the petals from smallest length to largest and let $j_i$ denote the length of the $i$th petal. 
For example, we have $j_i=1$ for $i=1,\dots,l_1$ corresponding to the $l_1$ petals of length $1$. 
Note that $j_1+\cdots+j_l=n_2$, the total number of $X$-labeled vertices in the flower.

Rewriting the $Y_1$ contribution from Corollary \ref{cor-Y1} 
we have 
\[\widetilde{\mathcal{Y}}_1=n_3!\sum_{|R|=l}\prod_{i=1}^l\frac{-\widetilde{I}_{j_i}j_i}{2^{j_i}}(\hbar\lambda_{r_i})^{j_i+1}.\]
The term $j_1\cdots j_l$ appears because each petal in the $\Theta$-flower could be connected to $\Theta$ at any one of its $j_i$ $X$-labeled vertices, 
 see the top of \cite[Pg. 34]{Engeli}.
 
Here, $\tilde{I}_j$ is the integral
\[\tilde{I}_j=\int_{[0,1]^j}\psi(v_0-v_1)\psi(v_1-v_2)\cdots\psi(v_{j-1}-v_j)\psi(v_j-v_0)dv_1\cdots dv_j.\] 
\noindent By \cite[Lem. 5.4]{FFSh} or \cite[Pg. 34]{Engeli}, we have $\tilde{I}_j=-I_{j+1}$. 

Ignoring the cycle and $\widetilde{Y}_2$ contributions for a moment, let 
\begin{align*}
P^{\Theta_1}_n&=n!\sum_{\widetilde{P}(n)}\left(\frac{1}{l!2^ll!\prod_{i=1}^l j_i}\right)l!\left(\sum_{|R|=l}\prod_{i=1}^l\frac{-\widetilde{I}_{j_i}j_j}{2^{j_i}}(\hbar\lambda_{r_i})^{j_i+1}\right)\\
&=n!\sum_{\widetilde{\mathcal{P}}(n)}\frac{1}{l!2^l}\sum_{|R|=l}\prod_{i=1}^l\frac{I_{j_i+1}}{2^{j_i}}(\hbar\lambda_{r_i})^{j_i+1}.
\end{align*}
where $\widetilde{\mathcal{P}}(n)$ ranges over all partitions 
\[n=\sum\limits_{i=1}^l j_i+l\]
where $j_i$ are all nonzero. 
We can view $\widetilde{P}(n)$ as the set of all decorated $\Theta$-flower graphs. 
Indeed, given such a partition, we have a corresponding $\Theta$-flower with 
$l$ petals of length $j_1,\dots, j_l$ 
and $l$ disjoint vertices, with $1$ spare vertex assigned to ``decorate" the each petal. 
The combinatorial term 
\[l!2^ll!\prod_{i=1}^l j_i\]
is the automorphism group of such a decorated $\Theta$-flower. 

Set $P^{\Theta_1}_0=1$ and let 
\[P^{\Theta_1}=\sum_{n=0}^\infty\frac{1}{n!}P^{\Theta_1}_n.\]
Now, 
\[\sum_{n=0}^\infty\sum_{\widetilde{\mathcal{P}}(n)}\frac{1}{l!}\sum_{|R|=l}\prod_{i=1}^a\lambda_i^{j_i+1}=\prod_{i=1}^a \left(\sum_{n'=0}^\infty\sum_{n'=j+1}\lambda_i^{j+1}\right).\]
Thus we have 

\begin{align*}
P^{\Theta_1}&=\sum_{n=0}^\infty \sum_{\widetilde{\mathcal{P}}(n)}\frac{1}{l!2^l}\sum_{|R|=l}\prod_{i=1}^l\frac{I_{j_i+1}}{2^{j_i}}(\hbar\lambda_{r_i})^{j_i+1}\\
&=\sum_{n=0}^\infty \sum_{\widetilde{\mathcal{P}}(n)}\frac{1}{l!}\sum_{|R|=l}\prod_{i=1}^l\frac{I_{j_i+1}}{2^{j_i+1}}(\hbar\lambda_{r_i})^{j_i+1}\\
&=\prod_{i=1}^a\left(\sum_{n'=0}^\infty \sum_{n'=j+1} \frac{I_{j+1}}{2^{j+1}}(\hbar\lambda_i)^{j+1}\right)\\
&=\prod_{i=1}^a\left(1+\sum_{j=0}^\infty\frac{I_{j+1}}{2^{j+1}}(\hbar\lambda_i)^{j+1}\right)\\
&=\prod_{i=1}^a\left(1+\sum_{j=0}^\infty\frac{I_{j+1}}{2^{j+1}}(\hbar\lambda_i)^{j+1}\right)\end{align*}

Since the contribution vanishes for $j+1$ odd and we have set $P_0^\Theta=1$, 
after substituting Equation (\ref{eq-IB}), 
this expression becomes

\begin{align*}
\prod_{i=1}^a\left(\sum_{j'=0}^\infty -B_{2j'}\frac{(\hbar\lambda_i)^{2j'}}{(2j')!}\right)=(-1)^a\prod_{i=1}^a(\hbar\lambda_i/2)\coth(\hbar\lambda_i/2).
\end{align*}

Lastly, we compute the $\mathcal{Y}_2$ contribution. 
Using the description (\ref{eq-Kappa}) we have
\[\widetilde{Y}_2=\sum_{|R|=l}\frac{(-1)^{(j+l)/2}n_3!}{2^{n_2+n_3-l}}\left(\prod_{i=1}^l-\widetilde{I}_{j_i}j_i(\hbar\kappa_{r_i})^{j_i+1}\right).\]

Let $P^{\Theta_2}_n$ be 

\begin{align*}
&n!\sum_{\widetilde{P}_2(n)}\left(\frac{l!\left(\frac{(-1)^{n/2}}{2^{n-l}}\right)}{l!2^ll!\prod_{i=1}^l j_i}\right)\left(\sum_{|R|=l}\left(\prod_{i=1}^l-\widetilde{I}_{j_i}j_i(\hbar\kappa_{r_i})^{j_i+1}\right)\right)\\
&=n!\sum_{\widetilde{P}_2(n)}\frac{(-1)^{n/2}}{l!2^n }\sum_{|R|=l}\left(\prod_{i=1}^l-\widetilde{I}_{j_i}(\hbar\kappa_{r_i})^{j_i+1}\right)\\
&=n!\sum_{\widetilde{P}_2(n)}\frac{1}{l!}\sum_{|R|=l}\left(\prod_{i=1}^l-\widetilde{I}_{j_i}(-1)^{\frac{j_i+1}{2}}\left(\frac{\hbar\kappa_{r_i}}{2}\right)^{j_i+1}\right)
\end{align*}

where $\widetilde{\mathcal{P}}_2(n)$ ranges over all partitions 
\[n=\sum\limits_{i=1}^l j_i+1\]
where $j_i$ are all nonzero.

Set $P^{\Theta_2}_0=1$ and let 
\[P^{\Theta_2}=\sum_{n=0}^\infty\frac{1}{n!}P^{\Theta_2}_n.\]
We have 

\begin{align*}
P^{\Theta_2}&=\sum_{n=0}^\infty \sum_{\widetilde{P}_2(n)}\frac{1}{l!}\sum_{|R|=l}\left(\prod_{i=1}^l-\widetilde{I}_{j_i}(-1)^{\frac{j_i+1}{2}}\left(\frac{\hbar\kappa_{r_i}}{2}\right)^{j_i+1}\right)\\
&=\prod_{i=1}^\z\left(\sum_{n=0}^\infty\sum_{n=j+1+2u}I_{j+1}(-1)^{\frac{j+1}{2}}\left(\frac{\hbar\kappa_i}{2}\right)^{j+1}\right)\\
&=\prod_{i=1}^\z \left(1+\sum_{j=0}^\infty I_{j+1}(-1)^{\frac{j+1}{2}}\left(\frac{\hbar\kappa_i}{2}\right)^{j+1}\right)\\
&=\prod_{i=1}^\z \left(\sum_{j'=0}^\infty -B_{2j'}(-1)^{j'}\frac{(\hbar \kappa_i)^{2j'}}{(2j')!}\right)\\
&= \prod_{i=1}^\z -(\hbar\kappa_i/2)\cot(\hbar\kappa_i/2).\\
\end{align*}
Putting the contributions from cycles and $\Theta$-flowers together, 
let 
\[P=\sum_{n=0}^\infty n! P_n.\]
Then $P$ is equal to

\begin{align*}
&(-1)^{a+\z }\prod_{i=1}^n \frac{\hbar\gamma_i/2}{\sinh(\hbar\gamma_i/2)}
\prod_{r=1}^a \sinh(\hbar\lambda_r/2)\coth(\hbar\lambda_r/2)
\prod_{s=1}^\z \sin(\hbar\kappa_s/2)\cot(\hbar\kappa_s/2)e^{x_2}\\
&=(-1)^{a+\z }\prod_{i=1}^n \frac{\hbar\gamma_i/2}{\sinh(\hbar\gamma_i/2)}
\prod_{r=1}^a \cosh(\hbar\lambda_r/2)
\prod_{s=1}^\z \cos(\hbar\kappa_s/2)e^{x_2}.
\end{align*}
\end{proof}

\begin{rmk}\label{rmk-CompareEngeli}
In type $(2n\vert a,a)$, Theorem \ref{thm-LocalSAIT} differs from Engeli's computation \cite[Lem. 2.25]{Engeli} by the sign $(-1)^a$.
The difference in sign $(-1)^a$ comes from the equality $\widetilde{I}_j=-I_{j+1}$. 
\end{rmk}

\subsection{Global Superalgebraic Index Theorem}\label{sec-GSAIT}

Consider the supertrace $\mrm{Tr}_\bb{M}$ on $\mathcal{A}_\sigma(\bb{M})$ from Theorem \ref{thm-DescendedTrace}. 
We would like to compute $\msf{Ev}_\bb{M}(\mrm{Tr}_\bb{M})$, the evaluation of $\mrm{Tr}_\bb{M}$ on the volume form from Definition \ref{def-EvalOnVol}. 

For this, we need a way of relating the characteristic class homomorphism $\mrm{char}_{(\bb{M},\sigma)}$ from \S\ref{subsec-GlobalInvariant}, the map $\chi$ from Definition \ref{def-Chi}, and  the classical Chern-Weil map \cite[Appendix C]{MilnorStasheff}.

Let $\mrm{pr}\colon\AQ\rta\mfrk{sp}_{2n}\oplus \mfrk{so}_{a,b}\oplus Z$ 
be the map used to define $\chi$ 
and $A\in\Omega^1(\Fr_\bb{M};\mfrk{g}^\hbar_{2n\vert a,b})$ the connection 1-form on $\Fr_\bb{M}$ used to define $\mrm{char}_{(\bb{M},\sigma)}$. 
Then $A$ is a flat connection by \cite[Def. 1.7]{GGW}. 

Recall from \S\ref{sec-Connections}, that we have a lift $\widetilde{A}$ of $A$ to a $\AQ$-valued connection. 
Under the identification 
\[H^\bullet_\mrm{Lie}(\mfrk{g}_{2n\vert a,b}^\hbar,\mfrk{sp}_{2n}\oplus \mfrk{so}_{a,b})\simeq H^\bullet_\mrm{Lie}(\AQ,\mfrk{sp}_{2n}\oplus \mfrk{so}_{a,b}\oplus Z),\]
the map $\mrm{char}_{(\bb{M},\sigma)}$ becomes evaluation on a wedge power of $\widetilde{A}$ rather than of $A$. 

Since $A$ is $\mrm{Sp}(2n)\times\mathrm{SO}(a,b)$ invariant and satisfies \cite[Def. 1.7(1)]{GGW}, 
the 1-form $\mrm{pr}(\widetilde{A})$ is a connection 1-form on $\Fr_\bb{M}$ valued in $\mfrk{sp}_{2n}\oplus \mfrk{so}_{a,b}\oplus Z$.

Since $\widetilde{A}$ is flat, its curvature $F_A$ is zero. 
However, since $\mrm{pr}$ is not a Lie algebra map, the connection $\mrm{pr}(\widetilde{A})$ may not be flat. 
Let $F_{\mrm{pr}(\widetilde{A})}\in\Omega^2$ denote the curvature of this connection. 
We use the notation $\mrm{CW}_\bb{M}$ for the map 
\[\mrm{CW}_\bb{M}\colon (\widehat{\mrm{Sym}}^\bullet((\mfrk{sp}_{2n}\oplus \mfrk{so}_{a,b}\oplus Z)^*))^{\mfrk{sp}_{2n}\oplus \mfrk{so}_{a,b}\oplus Z}\rta H_\mrm{dR}^{2\bullet}(\bb{M};\bb{K})\]
given by evaluating an invariant polynomial 
on $F_{\mrm{pr}(\widetilde{A})}$. 
Note that, as we are viewing $\mrm{CW}_\bb{M}$ as a map landing in cohomology, 
it is independent of the choice of connection whose curvature on which we evaluate polynomials. 

\begin{lem}\label{lem-CharCW}
The diagram
\[\begin{xymatrix}
{
H^{2\bullet}_\mrm{Lie}(\AQ,\mfrk{sp}_{2n}\oplus \mfrk{so}_{a,b}\oplus Z;\bb{K})\ar@{=}[r] & H^{2\bullet}_\mrm{Lie}(\mfrk{g}^\hbar_{2n\vert a,b},\mfrk{sp}_{2n}\oplus \mfrk{so}_{a,b};\bb{K})\arw[rrr]^-{\mrm{char}_{(\bb{M},\sigma)}(\bb{K})} & & & H^{2\bullet}_\mrm{dR}(\bb{M};\bb{K})\\
(\widehat{\mrm{Sym}}^\bullet(\mfrk{sp}_{2n}\oplus \mfrk{so}_{a,b}\oplus Z)^*))^{\mfrk{sp}_{2n}\oplus \mfrk{so}_{a,b}\oplus Z}\arw[u]^\chi\arw[urrrr]_{\mrm{CW}_\bb{M}} &  & & &
}
\end{xymatrix}\]
commutes. 
\end{lem}

\begin{rmk}
Lemma \ref{lem-CharCW} holds for any super Harish-Chandra pair $(\mfrk{g},K)$ and principal $(\mfrk{g},K)$-bundle $P$, with the same proof.
To cut down on notation, we will therefore prove this for 
$K=\mathrm{Sp}(2n\vert a,b)$ instead.
\end{rmk}

\begin{proof}
Let $P\in \mrm{Sym}^m((\mfrk{sp}_{2n\vert a,b}\oplus Z)^*)^{\mfrk{sp}_{2n\vert a,b}\oplus Z}$. 
Let $C\in\mrm{Hom}(\Lambda^2\AQ,\mfrk{sp}_{2n\vert a,b}\oplus Z)$ be the curvature of the projection $\mrm{pr}$ used to define $\chi$, 
see Definition \ref{def-Chi}. 
Then, $\mrm{char}_{(\bb{M},\sigma)}(\bb{K})(\chi(P))$ is the cohomology class of the $2m$-form 
$\chi(P)_*(\widetilde{A}^{\wedge 2m})$ 
where $\chi(P)_*$ is the map 
\[\chi(P)_*\colon\Omega^{2m}(\Fr_\bb{M};\Lambda^{2m}\mfrk{g}^\hbar_{2n\vert a,b}/\mfrk{sp}_{2n\vert a,b})\rta\Omega^{2m}(\Fr_\bb{M};\bb{K}).\] 
Using the definition of $\chi(P)$, we have
\[\chi(P)_*(\widetilde{A}^{\wedge 2m})=\frac{1}{m!}\sum_{s\in\Sigma_{2m}/(\Sigma_2)^{\times m}}\mrm{sign}(s)P(C(\widetilde{A},\widetilde{A}),\dots,C(\widetilde{A},\widetilde{A})).\]
The permutation has no effect on the term, so we may rewrite this as 
\[\chi(P)_*(\widetilde{A}^{\wedge 2m})=\frac{1}{m!}P(C(\widetilde{A},\widetilde{A}),\cdots,C(\widetilde{A},\widetilde{A})).\]
This is the definition of the Chern-Weil map $\mrm{CW}_{\bb{M}}$, 
assuming that the 2-form $F_{\mrm{pr}(\widetilde{A})}$ on $\Fr_\bb{M}$ is given by $C(\widetilde{A},\widetilde{A})$.  

To see this, note that the curvature $F_{\mrm{pr}(\widetilde{A})}$ measures the failure of $\mrm{pr}(\widetilde{A})$ to satisfy the Maurer-Cartan equation. 
Now
$\mrm{pr}(\widetilde{A})$ is a Maurer-Cartan element exactly when the corresponding morphism 
\[\mrm{pr}(\widetilde{A})\colon C^\bullet_\mrm{Lie}(\mfrk{g}^\hbar_{2n\vert a,b})\rta \Omega^\bullet(\bb{M};\bb{K})\]
is an algebra map. 
The map $\mrm{pr}$ induces a map 
\[\mrm{pr}^*\colon C^\bullet_\mrm{Lie}(\AQ)\rta C^\bullet_\mrm{Lie}(\mfrk{sp}_{2n\vert a,b}\oplus Z)\]
which is an algebra map if and only if $\mrm{pr}$ is a map of Lie algebras. 
The failure of $\mrm{pr}$ to be a Lie algebra map is measured by $C(\widetilde{A},\widetilde{A})$. 
Since the diagram
\[
\begin{xymatrix}
{
C^\bullet_\mrm{Lie}(\AQ)\arw[r]^{\widetilde{A}} & \Omega^\bullet(\bb{M};\bb{K})\\
C^\bullet_\mrm{Lie}(\mfrk{sp}_{2n\vert a,b}\oplus Z)\arw[u]^{\mrm{pr}^*}\arw[ur]_{\mrm{pr}(\widetilde{A})} & 
}
\end{xymatrix}
\]
commutes, $\mrm{pr}(\widetilde{A})$ is an algebra map if and only if $\mrm{pr}^*$ is; 
that is, $F_{\mrm{pr}(\widetilde{A})}$ and $C(\widetilde{A},\widetilde{A})$ are the same measurement.
\end{proof}

Write $F_{\mrm{pr}(A)}=R+S$ so that
 $R$ is a $\mfrk{sp}_{2n}$-valued form and 
 $S$ is a $\mfrk{so}_{a,b}$-valued form. 
 
 Let $\Omega$ denote the characteristic class of the deformation $\AQ$ as introduced in \S\ref{sec-Connections}.
 
\begin{thm}[Superalgebraic Index Theorem]\label{thm-main}
The evaluation of the unique normalized supertrace $\mrm{Tr}_\bb{M}$ on the volume form $1\otimes\Theta_\bb{M}$ is 
\[\msf{Ev}_\bb{M}(\mrm{Tr}_\bb{M})=(-1)^{n+a+\z }\hbar^n\int_\bb{M}\widehat{A}(R)\widehat{BC}(S)\exp(-\Omega/\hbar).\]
\end{thm}

Note that the $\widehat{A}$ appearing in Theorem \ref{thm-main} is not the $\widehat{A}$-genus of the supermanifold $\bb{M}$ as in \cite[Def. 3.1(6)]{Taniguchi}, 
but rather closer to the $\widehat{A}$ genus of the reduced manifold $M$ of $\bb{M}$; see Example \ref{ex-1}.

\begin{proof}

Since $\mrm{Tr}_\bb{M}$ is defined by descending the supertrace $\tau_{2n\vert a,b}$,  
by Theorem \ref{thm-GlobalizeInvariant}, we have
\[\msf{Ev}_\bb{M}(\mrm{Tr}_\bb{M})=\int_\bb{M}\mrm{char}_{(\bb{M},\sigma)}(\bb{K})(\msf{Ev}_\mrm{loc}(\tau_{2n\vert a,b})).\]
We saw in Lemma \ref{lem-ChiHitsZloc} that 
\[\msf{Ev}_\mrm{loc}(\tau_{2n\vert a,b})=(-1)^n\chi(P_n).\]
By Lemma \ref{lem-CharCW}, we can relate $\mrm{char}(\chi)$ to the Chern-Weil map and obtain the following 
\begin{align*}
\int_\bb{M}\mrm{char}_{(\bb{M},\sigma)}(\bb{K})(\msf{Ev}_\mrm{loc}(\tau_{2n\vert a,b}))&=\int_\bb{M}\mrm{char}_{(\bb{M},\sigma)}(\bb{K})\left((-1)^n\chi(P_n)\right)\\
&=(-1)^n\int_\bb{M}\mrm{CW}_\bb{M}(P_n).\\
\end{align*}
The Chern-Weil map $\mrm{CW}_\bb{M}$ evaluates an ad invariant polynomial on the curvature $F_{\mrm{pr}(\widetilde{A})}$. 
By the proof of Lemma \ref{lem-CharCW}, 
we have $F_{\mrm{pr}(\widetilde{A})}=C(\widetilde{A},\widetilde{A})$. 
As in \cite[\S 5.7]{FFSh} and \cite[Thm. 2.26]{Engeli}, 
we have 
\[C(\widetilde{A},\widetilde{A})=(R+S)-\Omega.\]
See also \S \ref{sec-Connections}. 
Using the description in Theorem \ref{thm-LocalSAIT}, 
of the polynomial $P_n$, 
we have 
\[\msf{Ev}_\bb{M}(\mrm{Tr}_\bb{M})=(-1)^{n+a+\z }\int_\bb{M}\left[\widehat{A}(\hbar R)\widehat{BC}(\hbar S)\exp(-\Omega)\right]_n.\]
The degree $n$ homogeneous part is 
\[\left[\widehat{A}(\hbar R)\widehat{BC}(\hbar S)\exp(-\Omega)\right]_n=\hbar^n\left[\widehat{A}(R)\widehat{BC}(S)\exp(-\Omega/\hbar)\right]_n.\]
\end{proof}

In the purely even case, Theorem \ref{thm-main} recovers the algebraic index theorem of \cite{FFSh}. 

\subsection{Examples}

We can rephrase Theorem \ref{thm-main} in terms of the reduced (non-super) manifold $M$ of $\bb{M}$. 
Note that the $\gKH$-bundle $\Fr_\bb{M}$ determines a $\gKH$-bundle $F_M$ on $M$ given (as a space) by the reduced manifold of $\Fr_\bb{M}$. 
The connection 1-form $A$ on $\Fr_\bb{M}$ is sent to the connection 1-form $A_\mrm{red}$ on $F_M$ by the Berezin integral 
\[\int(-)d\Theta\colon\Omega^1(\Fr_\bb{M};\mfrk{g}^\hbar_{2n\vert a,b})\rta \Omega^1(F_M;\mfrk{g}^\hbar_{2n\vert a,b}).\]
Since the characteristic map $\mrm{char}_P$ from \S\ref{subsec-GlobalInvariant} is defined in terms of the connection 1-form on the principal bundle $P$, we have a commutative diagram 
\[\begin{xymatrix}
{
C^\bullet_\mrm{Lie}(\mfrk{g}^\hbar_{2n}\oplus \mfrk{so}_{a,b},\mrm{sp}_{2n\vert a,b})\arw[rr]^-{\mrm{char}_{(\bb{M},\sigma)}}\arw[drr]_-{\mrm{char}_{F_M}} && \Omega^\bullet(\bb{M};\bb{K})\arw[d]^{\int(-)d\Theta}\\
&&  \Omega^\bullet(M;\bb{K}).
}
\end{xymatrix}\]

\begin{cor}
There is an equivalence of maps $C^\bullet_\mrm{Lie}(\mfrk{g}^\hbar_{2n\vert a,b},\mrm{sp}_{2n}\oplus \mfrk{so}_{a,b})\rta \bb{K}$ by
\[\int_\bb{M}\mrm{char}_{(\bb{M},\sigma)}=\int_M\mrm{char}_{F_M}.\]
\end{cor}

We can therefore interpret Theorem \ref{thm-main} in terms of characteristic classes for the bundle $F_M\rta M$. 
A particularly nice expression is obtained when the symplectic supermanifold $\bb{M}$ is ``split."  

A Theorem of Rothstein \cite{Rothstein} says that all symplectic supermanifolds are non-canonically isomorphic to one of the form $E[1]$ where $E\rta M$ is a quadratic vector bundle on a symplectic manifold $M$. 
Call $\bb{M}$ \emph{split} if we have chosen an identification $\bb{M}=E[1]$. 

\begin{ex}\label{ex-1}
Let $\bb{M}=E[1]$ be a split symplectic supermanifold. 
Then Theorem \ref{thm-main} gives the following computation:
\[\msf{Ev}_{\bb{M}}(\mrm{Tr}_\bb{M})=(-1)^{n+a+\z }\hbar^n \int_M\widehat{A}(M)\widehat{BC}(E)\exp(-\Omega/\hbar)\]
in terms of the characteristic series for th  $\widehat{A}$-genus of $M$ and the characteristic series $\widehat{BC}$ from Example \ref{ex-hatBC} of the vector bundle $E$.
\end{ex}

\begin{ex}[L-genus]
As a special case of the above example, 
consider the vector bundle $\pi\colon TM\rta M$. 
Since $M$ is a symplectic manifold, 
we get an identification $T^*M\cong TM$. 
Using this identification, we can consider the evaluation pairing 
\[TM\otimes TM\cong T^*M\otimes TM\xrta{\mrm{ev}} \bb{R}.\]
With this pairing, $TM$ becomes a quadratic vector bundle on $M$. 
We use the notation $T^*[1]M$ for the associated symplectic supermanifold. 
Note that $T^*[1]M$ has type $(2n\vert n,n)$. 
Moreover, the tangent bundle of $TM$ restricted to $M$ is 
\[\pi^*(TTM)=TM\oplus TM.\]
Example \ref{ex-1} then becomes 
\[\msf{Ev}_{T^*[1]M}(\mrm{Tr}_{T^*[1]M})=(-1)^{2n}\hbar^n\int_M\widehat{A}(M)\widehat{B}(M)\exp(-\Omega/\hbar).\]
The characteristic series for $\widehat{A}\widehat{B}$ is 
\[\frac{t_i/2}{\sinh(t_i/2)}\cosh(t_i/2)=\frac{t_i/2}{\tanh(t_i/2)},\]
which is the characteristic series for the L-genus; see Remark \ref{rmk-Lgenus}.
Thus, 
\[\msf{Ev}_{T^*[1]M}(\mrm{Tr}_{T^*[1]M})= \hbar^n \int_M L(M)\exp(-\Omega/\hbar).\] 
\end{ex}

\bibliography{vabib}
\bibliographystyle{alpha}
\end{document}